\documentclass[reqno,a4paper]{amsart}
\usepackage[left=1in,right=1in,top=1.5in,bottom=1in]{geometry}

\usepackage{Style}
\usetikzlibrary{calc,arrows.meta}
\usepackage{enumitem}
\usepackage{ytableau}
\ytableausetup{boxsize=.3em, centertableaux}
\usepackage{algorithm}
\usepackage{algpseudocode}

\algnewcommand{\InlineIfThen}[1]{\State\algorithmicif\ #1\ \algorithmicthen}
\algnewcommand{\EndInlineIf}{\unskip}

\newcommand{\g}{\mathfrak{g}}

\renewcommand{\sl}{\mathfrak{sl}}
\newcommand{\GL}{\mathrm{GL}}
\newcommand{\poly}{\mathrm{poly}}
\newcommand{\qpoly}{\mathrm{qpoly}}

\newcommand{\qbinom}[2]{\genfrac{[}{]}{0pt}{}{#1}{#2}}

\newcommand{\SYT}{\mathrm{SYT}}
\newcommand{\SSYT}{\mathrm{SSYT}}
\newcommand{\hl}{\mathsf{hl}}
\newcommand{\ct}{\mathsf{ct}}    

\newcommand{\yes}{\mathsf{yes}}
\newcommand{\no}{\mathsf{no}}
\renewcommand{\L}{\mathsf{L}}
\newcommand{\sharpL}{\#\mathsf{L}}
\newcommand{\sharpTISP}{\#\mathsf{TISP}}
\newcommand{\GapL}{\mathsf{GapL}}
\newcommand{\FL}{\mathsf{FL}}
\newcommand{\NL}{\mathsf{NL}}
\newcommand{\coNL}{\mathsf{coNL}}
\renewcommand{\P}{\mathsf{P}}
\newcommand{\FP}{\mathsf{FP}}
\newcommand{\FQP}{\mathsf{FQP}}
\newcommand{\GapP}{\mathsf{GapP}}
\newcommand{\sharpP}{\#\mathsf{P}}
\newcommand{\NP}{\mathsf{NP}}
\newcommand{\coNP}{\mathsf{coNP}}

\renewcommand{\leq}{\leqslant}
\renewcommand{\le}{\leqslant}
\renewcommand{\geq}{\geqslant}
\renewcommand{\ge}{\geqslant}

\newcommand{\ua}{\ensuremath{\mathord{\uparrow}}}
\newcommand{\da}{\ensuremath{\mathord{\downarrow}}}

\makeatletter
\let\c@equation\c@de

\makeatother

\usepackage{hyperref}
\usepackage{etoolbox}
\makeatletter
\patchcmd{\@maketitle}
  {\ifx\@empty\@dedicatory}
  {\ifx\@empty\@date \else {\vskip3ex \centering\footnotesize\@date\par\vskip1ex}\fi
   \ifx\@empty\@dedicatory}
  {}{}
\patchcmd{\@adminfootnotes}
  {\ifx\@empty\@date\else \@footnotetext{\@setdate}\fi}
  {}{}{}
\makeatother

\author[Á.~Gutiérrez]{Álvaro Gutiérrez}
\address[ÁG]{University of Bristol}
\email{a.gutierrezcaceres@bristol.ac.uk}

\author[C.~Ikenmeyer]{Christian Ikenmeyer}
\address[CI]{University of Warwick}
\email{christian.ikenmeyer@warwick.ac.uk}

\author[G.~Panova]{Greta Panova}
\address[GP]{University of Southern California}
\email{gpanova@usc.edu}

\date{July 17 2026}

\thanks{This material is based
upon work supported by the National Science Foundation under Grant No.\ DMS-1929284, while the authors were in residence at the Institute for Computational and Experimental Research in Mathematics in Providence, RI,
during the semester program ``Categorification and Computation in Algebraic Combinatorics'' in Fall 2025.
The authors thank Joshua P.\ Swanson and Michał Szwej for helpful conversations. ChatGPT 5.6 was used for proofreading.\\
ÁG was funded by a University of Bristol Research Training Support Grant.\\ GP was partially funded by NSF CCF:AF and DMS grants.\\For the purpose of open access, the authors have applied a Creative Commons Attribution (CC-BY) license to any Author
Accepted Manuscript version arising from this submission.}

\title{Counting in logarithmic space}

\begin{document}
\raggedbottom

\begin{abstract}
We study the class $\#\mathsf{L}$ of functions counting accepting paths of non-deterministic log-space Turing machines and construct methods to prove containment in $\#\mathsf{L}$. 
We prove that a large number of classical combinatorial and number theoretic functions belong to this class: classical functions from enumerative combinatorics (multinomial coefficients, Catalan numbers, linear extensions of trees, Stirling numbers, etc), algebraic combinatorics (number of standard Young tableaux, etc), discrete geometry, number theoretic functions, representation theoretic multiplicities in a large class of cases. We show that $\mathrm{GL}_2$-plethysm coefficients of bounded length outer partition can be counted by log$^2$-space polytime verifiers. We pose numerous questions and conjectures on $\#\mathsf{L}$ containment and its generalizations, that suggest venues for conditionally disproving $\#\mathsf{P}$-completeness. 
While studying which combinatorial functions are in $\#\mathsf{P}$ provides a formal way of (dis)proving the existence of combinatorial interpretations, the lower class $\#\mathsf{L}$ serves as an analogue for functions computable in polynomial time. \\

\end{abstract}

\maketitle

\tableofcontents

\section{Introduction}
Algebraic Combinatorics studies objects and quantities originating in algebra, representation theory, and geometry using combinatorial tools. 
A typical problem that arises is 
\begin{quote}
\itshape
The values of the function $X$ are the dimensions of vector spaces, and hence $X$ outputs only non-negative integer values.
Find a combinatorial interpretation for $X$, i.e.,
a family of nice combinatorial objects which are counted by~$X$.
\end{quote}
A flagship problem of this type that has been successfully resolved is the combinatorial interpretation of the Littlewood--Richardson coefficients. Yet many other quantities, like Kronecker and plethysm coefficients, remain mysterious. 

The above question contains the undefined ``nice combinatorial objects'' and ``combinatorial interpretation''. A way to formalize these concepts is to understand ``combinatorial interpretation'' as ``being an element of the complexity class $\sharpP$'',
as considered in \cite{Mul07}.
This brings computational complexity tools and provides ways to disprove the existence of combinatorial interpretations, see for example \cite{WhatIsIn} and \cite{PosGroupCharacters}.
The survey \cite{WhatIsAComb} almost equates having a combinatorial interpretation with the membership in $\#\mathsf{P}$, see \cite[\S3.1]{WhatIsAComb}, stating that these terms ``usually coincide but can also differ in several special cases''. Indeed, $\sharpP$ by definition is the complexity class of counting [possibly exponentially many] witnesses, each verifiable in polynomial time, which does not always give aesthetically satisfying ``combinatorial interpretations''. In many cases we have combinatorial functions computable in polynomial time, hence in $\FP \subseteq \sharpP$, but the emerging counting formula  would just be the pre-computed integers themselves;
for example the description $\binom{n}{k} = \#\{x\in\N \mid k!\,(n-k)!\,x \le n!\}$ places the binomial coefficient in $\sharpP$ without giving a satisfying combinatorial interpretation.
It becomes fruitful to understand the lower counting complexity classes (counting complexity classes form a partially ordered set by inclusion) that important counting problems belong to. 
This direction has been pioneered in \cite[end of \S1.1]{DI24}, and in \cite{BCDI25} (where the counting problems have exponentially large input).
\medskip

We study the complexity class $\sharpL$ of functions counting the number of accepting paths of a non-deterministic Turing machine that runs in \emph{logarithmic space}. This is a subclass of $\sharpP$.
We explicitly construct logarithmic space Turing machines to establish that the following functions are in $\sharpL$.
\begin{itemize}[leftmargin=2em]
    \item \textbf{From enumerative combinatorics.} Binomial and multinomial coefficients, as well as the factorial. Catalan, Narayana, Stirling, Fibonacci, and Euler numbers, and any sequence given by a 1- or 2-dimensional linear recursion. For partitions $\lambda$ of a bounded length: the number of ballot sequences of type $\lambda$, which also give standard and semistandard tableaux of shape $\lambda$. The number of linear extensions of a rooted tree poset.  The differences involved in the unimodality of binomial coefficients, and the log-concavity of binomial and Stirling coefficients. The number of permutations of a given cycle type, and the number of permutations of $n$ which are involutions, the number of certain pattern avoiding permutations. The determinant of the distance matrix of a tree is even log-space computable.
    \item \textbf{From representation theory and geometry.} 
    The number of integer points in a polytope $P$  of fixed dimension. The number of contingency tables where at least one dimension is constant.
    For fixed $n$, the $\GL_n$-Littlewood--Richardson coefficients, $\GL_n$-Kostka numbers, and the product constants of monomial symmetric polynomials.
    In addition, one can compute in log-space the raising operator of crystals, Kashiwara's tensor product rule, and decide whether a given element is a highest weight element of a prescribed weight.
    \item \textbf{From number theory.} The partition function (which implies its non-trivial containment in $\FP$), the coefficients in $q$ of the $q$-binomial and the $q$-multinomial. The Euler totient function, and the divisor function $\sigma_k(n)$. Since one can perform Euclidean division in log-space and also decide the primality of a number, the divisor function $\sigma_k(n)$ is even computable in log-space for each fixed $k$. If for a given formula its $p$-adic valuation is log-space computable, then the formula is in $\sharpL$; this places the radical function and the factorial (again) in $\sharpL$, as well as the hook-length formula for the number of standard Young tableaux of shapes of \emph{unbounded} length (strengthening the result from the section on enumerative combinatorics).
\end{itemize}
We then study a family of coefficients that are simultaneously $\GL_2$-plethysm coefficients and rectangular Kronecker coefficients \cite{IOT}, and which we call \emph{Hermite coefficients}. Hinging on a combinatorial interpretation for these coefficients \cite{PakPanovaSwanson}, we construct a non-deterministic $\poly$-time $\mathrm{polylog}$-space Turing machine that solves this problem. An even more general result is the following. See~\S\ref{sec:zoo} for the precise definition of the complexity class.
\begin{thm}[$\GL_2$-plethysm coefficients]
Fix a constant $C\in\N$. Let $\mu$ be a partition of length at most $C$.
The function $$(1^{\mu_1}\,0\,1^{\mu_2}\,0\,\dots\,0\,1^{\mu_C}\,0\,1^k\,0\,1^r) \mapsto a_{\mu[k]}^{(|\mu|k-r,r)}$$ is in $\sharpTISP(\poly(n),\log^2(n))\cap \GapL$.
\end{thm}

We also propose open problems which complement existing conjectures related to other complexity classes. 
\begin{enumerate}[leftmargin=2.5cm]
    \item[{Question~\makebox[2em][r]{\ref{q:CT}}}] Is the number of contingency tables of unbounded sizes $\log^k$-space computable for some $k$? If the answer is yes this would give evidence against their conjectured $\sharpP$-completeness.
    \item[{Question~\makebox[2em][r]{\ref{q:skewSYT}}}] Is the number of skew standard Young tableaux in $\sharpL$?
    \item[{Question~\makebox[2em][r]{\ref{q:Hermite}}}] Are the Hermite coefficients $a_{n[k]}^{(m+t,m)}$ in $\sharpL$?
    \item[{Question~\makebox[2em][r]{\ref{q:schur ps}}}] Are the coefficients of the principal specialisation of $s_\lambda$ in $\sharpL$ when $\ell(\lambda)$ is unbounded?
    \item[{Question~\makebox[2em][r]{\ref{q:GL2 plet}}}] Are the $\GL_2$-plethysm coefficients $a_{\mu[k]}^{(m+t,m)}$ $\log^2$-space computable when $\ell(\mu)$ is unbounded?
\end{enumerate}

\medskip

In recent years there has been a growing interest for the interplay between complexity theory and algebraic combinatorics. 
One early and influential source is the Geometric Complexity Theory programme towards $\P\ne\NP$ \cite{GCT1}. Mulmuley and Sohoni's programme raises important questions about representation-theoretic constants that had already been studied since the 1930s: plethysm and Kronecker coefficients.
Finding combinatorial interpretations for these constants constitute Problems~9 and~10 in Stanley's influential list \cite{StanleyList} of open problems in algebraic combinatorics.
Motivated by the positive results for the Littlewood-Richardson coefficients, Mulmuley \cite{Mul07,GCT6} conjectured that these constants are in $\sharpP$, which is the complexity class of functions that count the number of accepting paths of a non-deterministic Turing machine that runs in \emph{polynomial time}.

Stanley's problem and Mulmuley's conjecture have been subject to intense scrutiny. The simpler problem of deciding positivity of plethysm and Kronecker coefficients is shown to be $\NP$-hard in \cite{IMW} and \cite{FI}, hence determining their exact values is $\sharpP$-hard. 
Plethysm and Kronecker coefficients are known to be in $\GapP = \sharpP-\sharpP$,
as first shown in \cite{BI08,FI}, respectively, and many other quantities are as well, see \cite{PR}.
Kirillov's 2004 conjecture about $\sharpP$ containment of so-called \emph{reduced} Kronecker coefficients turned out to be equivalent to Stanley's problem 10 \cite{IP24}.
A related representation-theoretic constant (character values squared) does not belong to $\sharpP$, as shown in \cite{PosGroupCharacters}, under standard complexity theoretic assumptions, which implies that there cannot be a positive combinatorial interpretation. Parallel to that Pak conjectured that the Kronecker coefficients would not be in $\sharpP$~\cite{WhatIsAComb} under standard complexity theoretic assumptions, and hence not have a nice positive combinatorial interpretation. 
Both plethysm and Kronecker coefficients are shown to be in a quantum analogue of $\sharpP$, namely $\#\mathsf{BQP}$, in \cite{BCGHZ,IS25,christandl2026plethysm}. Further relationships, results and open problems on their complexity are described in~\cite{panova2025computational}. 
While belonging to $\sharpP$ is an intriguing question for the plethysm and Kronecker coefficients,
the $\sharpP$ membership question is too coarse for a large class of counting problems which trivially belong to $\sharpP$ but for which the combinatorial interpretations are interesting. For example, large families of Kronecker and plethysm coefficients are computable in poly-time and hence belong to $\FP \subset \sharpP$, see e.g.~\cite{P25}, yet there is no satisfactory positive combinatorial interpretation beyond computing the final answer itself. 
Hence, we study the containment in the counting class $\sharpL \subseteq \FP$.

\medskip

The paper is organised as follows. In \S\ref{sec:preliminaries} we review the main complexity theory concepts that we use throughout. We discuss combinatorial functions in $\sharpL$ in \S\ref{sec:combinatorics}, representation-theoretic functions in \S\ref{sec:structure constants}, and number-theoretic functions in~\S\ref{sec:number theory}. We finish by discussing $\GL_2$-plethysm coefficients and rectangular Kronecker coefficients in \S\ref{sec:OHara}. We define the combinatorial objects as they arise, referring to~\cite{EC1,StanleyEC2} for the necessary background in enumerative and algebraic combinatorics. We also state open problems and remarks on the various topics within their sections. 

\section{Preliminaries}\label{sec:preliminaries}

\subsection{Decision complexity: time and space}
We primarily work over the alphabet $\Sigma=\{0,1\}$ and languages $L$ are subsets of finite length strings, i.e., $L\subseteq\Sigma^*$.
Following \cite[\S1.5 and Ch.~2]{AroraBarak}, define
\[
\P = \bigcup_{c\ge1}\mathsf{DTIME}(n^c)
\quad
\text{and}
\quad
\NP = \bigcup_{c\ge1}\mathsf{NTIME}(n^c)
\]
as the classes of decision problems solved by deterministic and non-deterministic poly-time Turing machines, respectively.
In this paper, all nondeterministic Turing machines are required to halt on every input after a finite number of steps.
An alternative and useful description of $\NP$ can be given as follows:
 the class $\NP$ is the class of languages $L$ for which there is a polynomial $p$ and a poly-time deterministic Turing machine $M$ satisfying
\[
x\in L
\iff
\exists w\in\Sigma^{p(|x|)} : M(x,w) = \yes.
\]
We say $M$ is a \emph{verifier $\NP$ machine} to distinguish it from the non-deterministic $\NP$ machines considered above.

The class $\coNP$ is the class of languages $L$ for which there is a polynomial $p$ and a poly-time deterministic Turing machine $M$ satisfying
\[
x\not\in L
\iff
\exists w\in\Sigma^{p(|x|)} : M(x,w) = \no.
\]
It is a major open problem to show that $\NP \ne \coNP$.\medskip

Following \cite[Ch.~4]{AroraBarak}, define
\[
\L = \mathsf{SPACE}(\log(n))
\quad
\text{and}
\quad
\NL = \mathsf{NSPACE}(\log(n)).
\]
To use a verifier approach for defining $\NL$ one has to be careful (see~\cite[\S4.4.1]{AroraBarak}, the discussion before~\cite[Definition 2.2]{Satanic}, or \cite{Burtschick}), see Figure~\ref{fig:sharpL machine}. 
\begin{thm}\label{thm:NL machine}
The class $\NL$ is the class of languages $L$ for which there is a polynomial $p$ and a log-space deterministic Turing machine $M$ with an additional \emph{read-once, left-to-right} input tape (dedicated to the witness) satisfying
\[
x\in L
\iff
\exists w\in\Sigma^{p(|x|)} : M(x,w) = \yes.
\]
\end{thm}
One can define $\coNL$ in a similar manner. A groundbreaking result of Immerman \cite{Immerman} and Szlepcs\'enyi \cite{Szelepcsenyi} is that $\NL = \coNL$.\medskip

The machines considered in Theorem~\ref{thm:NL machine} are called \emph{verifier $\NL$ machines}. The additional input tape is called the \emph{witness tape}. We illustrate a verifier $\NL$ machine in Figure~\ref{fig:sharpL machine}.
\begin{figure}[h]
    \centering
    \begin{tikzpicture}[x=1em, y=-1em]
        \draw (0,0) node[anchor=west, rectangle, white, fill=blue!70, rounded corners = .7em] {\texttt{read-only input tape}};
        \draw (11.5,0) node[anchor=west, rectangle, white, fill=green!70!black, rounded corners = .7em] {\texttt{read-once left-to-right witness tape}};
        \draw (0,2) node[anchor=west, rectangle, fill=gray!50, rounded corners = .7em] {\texttt{work tape}};
        \draw (0,4) node[anchor=west, rectangle, white, fill=blue!70, rounded corners = .7em] {\texttt{y/n}};
    \end{tikzpicture}
    \caption{A verifier $\NL$ machine: the input tape is of size $n$ and read-only; the witness tape is of $\mathrm{poly}(n)$ size and left-to-right read-once only; the work tape is of $\log(n)$ size; the output is $1$-bit.}
    \label{fig:sharpL machine}
\end{figure}
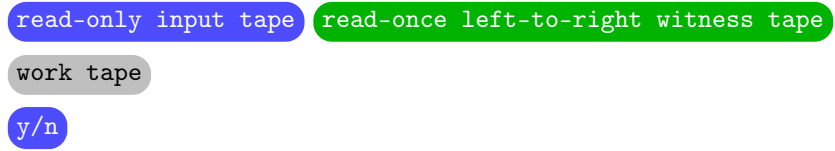

\subsection{On verifier \texorpdfstring{$\NL$}{NL} machines}\label{subsec:onvNLmachines}
The ``read-once, left-to-right'' restriction on the witness tape of a verifier $\NL$ machine plays an important role in Theorem~\ref{thm:NL machine}.
A verifier $\NP$ machine can simply copy the witness onto the work tape, since the witness is polynomial in size. Once in the work tape, the witness can be read multiple times. Hence a verifier $\NP$ machine with a read-once, left-to-right witness tape is simply a verifier $\NP$ machine. However, these reductions are outside of the capabilities of a verifier $\NL$ machine. If this ``read-once, left-to-right'' restriction is removed from the definition of verifier $\NL$ machines, then the computational model is much more powerful: we recover the notion of a verifier $\NP$ machine since one can record the successive states of the machine tape in the witness itself. 

The content of the work tape is a binary string, which we call a binary-encoded counter. With this notion, all verifier $\NL$-machines can be assumed without loss of generality to run an algorithm of the form of Algorithm~\ref{alg: simple NL machine}.
\begin{algorithm}
\caption{Simple verifier $\NL$ machine}\label{alg: simple NL machine}
\begin{algorithmic}
\Require $x, w$
\State {Initialize a $\log(|x|)$-space binary-encoded counter $\texttt{a}$}
\For{$w_i \in w$}
\State {Update $\texttt{a}$ via operations with $\log(|x|)$-space complexity}
\EndFor
\State {Check $\texttt{a}=0$}
\end{algorithmic}
\end{algorithm}

Nevertheless, there are three ways of relaxing the definition of a verifier $\NL$ machine which prove to be very useful in practice. Firstly, we can suppose the alphabet of symbols $\Sigma$ which can be written in the work and the witness tapes to be any alphabet of a fixed finite size $s$. Any letter of this new alphabet can be simulated with $\log_2(s)$ bits in the alphabet $\{0,1\}$. If the machine $M_\Sigma$ uses $O(\log(n))$ space, then the simulated machine $M'_{\{0,1\}}$ uses $\log_2(s)\cdot O(\log(n))=O(\log(n))$ space.

Secondly, we can suppose that a verifier $\NL$ machine has a constant number $c$ of work tapes (as opposed to just one). One can simulate a machine $M_c$ with a constant amount of log-space tapes with a machine $M'_1$ that has just one tape, by using an alphabet that has $|\Sigma|^c$ times as many letters, where $\Sigma$ is the alphabet of symbols that each of the $c$ work tapes uses. The entries of the work tape can thus be interpreted as $c$-tuples of letters of the original alphabet. That is, we represent the $c$ tapes of $M_c$ not next to each other, but ``on top of each other'' \cite[Problem 2.8.6]{Papadimitriou}. To simulate the positions of the $c$ pointers,
one increases the alphabet by another factor of $2^c$, so that each cell contains the information about which simulated machine heads are currently at this position.

Thirdly, we can suppose that the machine $M_C$ can read the witness tape a constant amount $C$ of times, left-to-right \cite{SE}. To simulate this with a verifier $\NL$ machine $M'_1$ with a read-once witness tape, start by guessing the state of the work tape after the 1st, 2nd, \ldots, $(C-1)$st pass and store the guesses. Then, run $C$ parallel computations in $C$ different work tapes, all the while reading the witness once. Finally, check that the state of the $i$th work tape after the computation matches the predicted state that launched the $(i+1)$st work tape. 

In Figure~\ref{fig:NL machine weak} we illustrate this more general version of verifier $\NL$ machines. See Algorithm~\ref{alg:cap} for the corresponding pseudo-code for this machine. All of the algorithms of this paper will be of the form of Algorithm~\ref{alg:cap} with only one exception in~\S\ref{sec:OHara}.

\begin{figure}[h]
    \centering
    \begin{tikzpicture}[x=1em, y=-1em]
        \draw (0,0) node[anchor=west, rectangle, white, fill=blue!70, rounded corners = .7em] {\texttt{read-only input tape}};
        \draw (11.5,0) node[anchor=west, rectangle, white, fill=green!70!black, rounded corners = .7em] {\texttt{read-$C$-times left-to-right witness tape}} ++(21.7,.5) node {\scriptsize$\Sigma$};
        \draw (0,2) node[anchor=west, rectangle, fill=gray!50, rounded corners = .7em] {\texttt{work tape 1}} ++(6.7,.5) node {\scriptsize$\Sigma$};
        \draw (0,4) node[anchor=west, rectangle, fill=gray!50, rounded corners = .7em] {\texttt{work tape 2}} ++(6.7,.5) node {\scriptsize$\Sigma$};
        \draw (2,5.2) node {\vdots};
        \draw (0,7) node[anchor=west, rectangle, fill=gray!50, rounded corners = .7em] {\texttt{work tape $c$}} ++(6.65,.5) node {\scriptsize$\Sigma$};
        \draw (0,9) node[anchor=west, rectangle, white, fill=blue!70, rounded corners = .7em] {\texttt{y/n}};
    \end{tikzpicture}
    \caption{A machine over an arbitrary alphabet $\Sigma$, with a constant number $c$ of work tapes, and in which the witness tape can be read a constant $C$ of times. This machine can be simulated by the verifier $\NL$ machine of Figure~\ref{fig:sharpL machine}.}
    \label{fig:NL machine weak}
\end{figure}
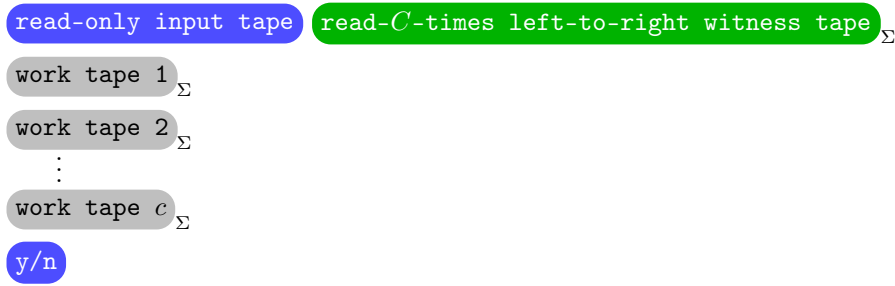

    \begin{algorithm}
    \caption{Verifier $\NL$ machine of Figure~\ref{fig:NL machine weak}}\label{alg:cap}
    \begin{algorithmic}
    \Require $x, w$
    \State {Initialize $\log(|x|)$-space $\Sigma$-counters $\mathtt{a_1}, \ldots, \mathtt{a}_{c}$}
    \For{$j=1,\ldots, C$}
    \For{$w_i \in w$}
    \State {Update $\mathtt{a_1}, \ldots, \mathtt{a}_{c}$ via operations with $\log(|x|)$-space complexity}
    \EndFor
    \EndFor
    \State {Check $\mathtt{a_1} = 0 = \mathtt{a_2} =  \cdots = \mathtt{a}_{c}$}
    \end{algorithmic}
    \end{algorithm}

\subsection{Counting complexity: time and space}

Most problems in algebraic combinatorics are \emph{counting problems}. The study of the complexity of counting was initiated by Valiant in~\cite{Valiant} and pioneered in algebraic combinatorics by Mulmuley~\cite{Mul07}.

\begin{de}
    The class $\FP$ is the class of functions $f:\Sigma^*\to\Sigma^*$ that are computable by a poly-time deterministic Turing machine with an output tape. The class $\sharpP$ is the class of functions
    counting the number of accepting paths of a non-deterministic $\NP$ machine.
    The class $\FP_{\ge0}$ is the class of poly-time computable functions with non-negative output, namely $\FP\cap\{f:\Sigma^*\to\N_0\}$.
    The class $\GapP=\sharpP - \sharpP$ is the class of functions which can be expressed as the difference between two $\sharpP$ functions.
\end{de}

    Equivalently, $\sharpP$ is the class of functions such that there exists a polynomial $p$ and a verifier $\NP$ machine $M$ such that
    \[
    f(x) = \#\{
    w\in\Sigma^{p(|x|)} \mid M(x,w) = \yes
    \}.
    \]
\begin{de}
    The class $\FL$ is the class of functions $f:\Sigma^*\to\Sigma^*$ that are computable by a log-space deterministic Turing machine with a write-only output tape. The class $\sharpL$ is the class of functions counting the number of accepting paths of a non-deterministic $\NL$ machine.
    The class $\FL_{\ge0}$ is defined as $\FL\cap\{f:\Sigma^*\to\N_0\}$. The class $\GapL$ is $\sharpL - \sharpL$.
\end{de}

The problem \textsc{stCon} for source-sink connectivity in directed graphs is $\sharpL$-complete.
The problem of computing the determinant of an integer matrix is complete for $\GapL$, see \cite{AO94}, where containment in $\GapL$ can be proved for example via a generalization of the Cayley--Hamilton theorem \cite{Ike25}, and hardness follows from Valiant's simulation of algebraic branching programs via determinants \cite{Val79completeness}.

    As an equivalent definition, $\sharpL$ is the class of functions such that there exists a polynomial $p$ and a verifier $\NL$ machine $M$ such that
    \[
    f(x) = \#\{
    w\in\Sigma^{p(|x|)} \mid M(x,w) = \yes
    \}.
    \]
    Recall that there exists a read-once constraint on the witness tape of verifier $\NL$ machines. This notion of $\sharpL$ coincides with the one studied in~\cite{AJ93, Burtschick} and included in \cite{zoo}, but beware: it does not come from an application of the $\#$ operators given in~\cite[Definitions 1.1 and 1.2]{Satanic} (which are not well suited for subpolynomial classes, see also~\cite[\S2]{Satanic}).

\subsection{Log-space reduction}\label{sec:log-space reduction}
In this section we illustrate some well-known facts about log-space reductions.
Following \cite[Def.~8.1]{Papadimitriou}, a language $L$ is \emph{reducible} to $L'$ if there is a function $R:\Sigma^*\to\Sigma^*$ computable by a deterministic log-space Turing machine with a write-only output tape and such that
\[
x\in L \iff R(x)\in L'.
\]
We call $R$ a \emph{log-space reduction} and the machine that computes $R$ a \emph{log-space transducer}.

The classes $\L$ and $\NL$ are closed under log-space reduction \cite[Prop.~8.3]{Papadimitriou}, meaning that if $L$ is reducible to a language $L'$ in $\L$ (resp.~in $\NL$) then also $L$ is in $\L$ (resp.~in $\NL$). We can illustrate this reduction as follows:
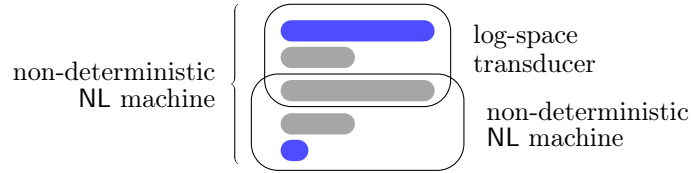
\begin{center}
    \begin{tikzpicture}[x=1em, y=-1em]
        \draw [decorate, decoration = {calligraphic brace}] (-2,5) -- (-2,-1);
        \node[anchor=east] (*) at (-2.5,1.5) {non-deterministic};
        \node[anchor=east] (*) at (-2.5,2.5) {$\NL$ machine};
        \draw[rounded corners = 1em] (-1,-1) rectangle (6,2.85);
        \draw (6.5,.2) node[anchor=west] {log-space};
        \draw (6.5,1.2) node[anchor=west] {transducer};
        \draw (7,3) node[anchor=west] {non-deterministic};
        \draw (7,4) node[anchor=west] {$\NL$ machine};
        \draw[rounded corners = 1em, line cap = round, line width = .8em, blue!70] (0,0) -- (5,0);
        \draw[rounded corners = 1em, line cap = round, line width = .8em, gray!70] (0,1) -- (2,1);
        \filldraw[rounded corners = 1em, line cap = round, line width = .8em, gray!70] (0,2.25) -- (5,2.25);
        \draw[rounded corners = 1em] (-1.5,1.625) rectangle (6.5,5.25);
        \draw[rounded corners = 1em, line cap = round, line width = .8em, gray!70] (0,3.5) -- (2,3.5);
        \draw[rounded corners = 1em, line cap = round, line width = .8em, blue!70] (0,4.5) -- (.25,4.5);
    \end{tikzpicture}
\end{center}
There is one subtlety to point out: the input of tape non-deterministic machine might not fit in log-space, in which case the whole deterministic machine is run once for each bit of the input tape of the non-deterministic machine.
Using the same idea, one can compose several log-space machines into a log-space machine. In particular, $\FL$ is closed under composition:
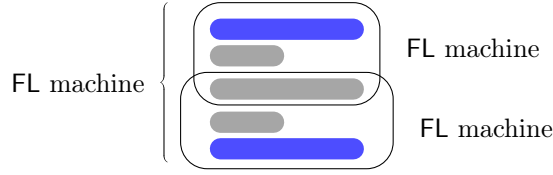
\begin{center}
    \begin{tikzpicture}[x=1em, y=-1em]
        \draw [decorate, decoration = {calligraphic brace}] (-2,5) -- (-2,-1);
        \node[anchor=east] (*) at (-2.5,2) {$\FL$ machine};
        \draw (9.5,.7) node {$\FL$ machine};
        \draw (10,3.7) node {$\FL$ machine};
        \draw[rounded corners = 1em] (-1,-1) rectangle (6,2.85);
        \draw[rounded corners = 1em, line cap = round, line width = .8em, blue!70] (0,0) -- (5,0);
        \draw[rounded corners = 1em, line cap = round, line width = .8em, gray!70] (0,1) -- (2,1);
        \draw[rounded corners = 1em, line cap = round, line width = .8em, gray!70] (0,2.25) -- (5,2.25);
        \draw[rounded corners = 1em] (-1.5,1.625) rectangle (6.5,5.25);
        \draw[rounded corners = 1em, line cap = round, line width = .8em, gray!70] (0,3.5) -- (2,3.5);
        \draw[rounded corners = 1em, line cap = round, line width = .8em, blue!70] (0,4.5) -- (5,4.5);
    \end{tikzpicture}
\end{center}
In this way we can show that 
$\FL_{\ge0}\subseteq\sharpL$. Indeed, the following picture shows how to create a verifier $\NL$ machine starting from a $\FL$ machine $M_1$ (in yellow). We do this by composing $M_1$ with the verifier $\NL$ machine $M_2$ that takes an input $x$ and a witness is a number $k$, and outputs whether $0\le k < M_2(x)$ or not.
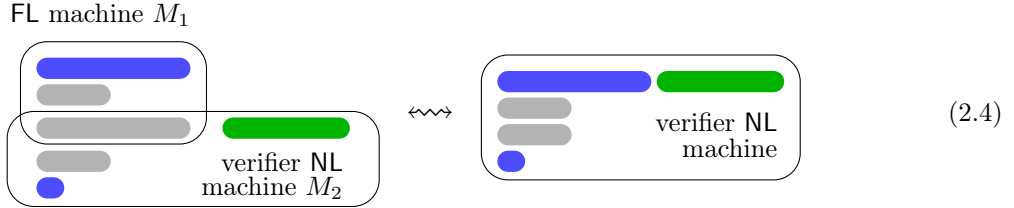
\begin{equation}\label{eq:FL composition}
    \begin{tikzpicture}[x=1em, y=-1em, baseline=-2em]
        \draw (2,-2) node {$\FL$ machine $M_1$};
        \draw (11.5,3.5) node[anchor=east] {verifier $\NL$};
        \draw (11.5,4.5) node[anchor=east] {machine $M_2$};
        \draw[rounded corners = 1em] (-1,-1) rectangle (6,2.85);
        \draw[rounded corners = 1em, line cap = round, line width = .8em, blue!70] (0,0) -- (5,0);
        \draw[rounded corners = 1em, line cap = round, line width = .8em, green!70!black] (7,2.25) -- (11,2.25);
        \draw[rounded corners = 1em, line cap = round, line width = .8em, opacity=.3] (0,1) -- (2,1);
        \draw[rounded corners = 1em, line cap = round, line width = .8em, opacity=.3] (0,2.25) -- (5,2.25);
        \draw[rounded corners = 1em] (-1.5,1.625) rectangle (12.5,5.2);
        \draw[rounded corners = 1em, line cap = round, line width = .8em, opacity=.3] (0,3.5) -- (2,3.5);
        \draw[rounded corners = 1em, line cap = round, line width = .8em, blue!70] (0,4.5) -- (.25,4.5);
    \end{tikzpicture}
    \hspace{1em}
    \scalebox{1.3}{$\leftrightsquigarrow$}
    \hspace{1em}
    \begin{tikzpicture}[x=1em, y=-1em, baseline=-2.5em]
        \draw (10.5,2.5) node[anchor=east] {verifier $\NL$};
        \draw (10.5,3.5) node[anchor=east] {machine};
        \draw[rounded corners = 1em] (-1,0) rectangle (11,4.7);
        \draw[rounded corners = 1em, line cap = round, line width = .8em, blue!70] (0,1) -- (5,1);
        \draw[rounded corners = 1em, line cap = round, line width = .8em, green!70!black] (6,1) -- (10,1);
        \draw[rounded corners = 1em, line cap = round, line width = .8em, opacity=.3] (0,2) -- (2,2);
        \draw[rounded corners = 1em, line cap = round, line width = .8em, opacity=.3] (0,3) -- (2,3);
        \draw[rounded corners = 1em, line cap = round, line width = .8em, blue!70] (0,4) -- (.25,4);
    \end{tikzpicture}
\end{equation}
\medskip

In practice, we often want to compose more than one machine, so we would like a framework that allows for such compositions. 
The following very general result applies to a wide range of problems, specially useful for the analysis of combinatorial formulas.
\begin{thm}\label{thm:product}
    Let $f : \Sigma^*\to\N_0$ be a $\sharpL$ function. Let $g : \Sigma^*\times\N_0 \to \Sigma^*$ be an $\FL$ function. Let $p : \Sigma^*\to\N_0$ be an $\FL_{\ge0}$ function whose output is at most polynomial in the size of the input. Then the function
    \[
    F : x \mapsto \prod_{i=1}^{p(x)} f(g(x,i))
    \]
    is in $\sharpL$.
\end{thm}
\begin{proof}
    We construct a verifier $\NL$ machine for $F$ based on the verifier $\NL$ machine for $f$. 
    The witness is a concatenation of the witnesses for $f(g(x,1)), f(g(x,2)), \ldots, f(g(x,p(x)))$, which we call $w^1, w^2, \ldots, w^{p(x)}$.
    
    Begin by composing with an $\FL$ machine as in~\eqref{eq:FL composition} to compute $p(x)$ and store its output in a counter~$\mathtt{I}$. We can store this in binary, since $\log(\mathrm{poly}(|x|)) = O(\log(|x|))$.

    Initialise a counter $\mathtt{i}=1$. While $\mathtt{i} \le \mathtt{I}$, compose with an $\FL$ machine to compute $g(x,\mathtt{i})$, then check the validity of $w^{\mathtt{i}}$, and finally increase $\mathtt{i}$ by~$1$.
\end{proof}

A very similar idea as in the proof of Theorem~\ref{thm:product} can be applied to more complicated functions than the product.

\begin{thm}\label{thm:detofsharpL}
    Let $f : \Sigma^*\to\N_0$ be a $\sharpL$ function. Let $g : \Sigma^*\times\N\times\N \to \Sigma^*$ be an $\FL$ function. Let $p : \Sigma^*\to\N_0$ be an $\FL_{\ge0}$ function whose output is at most polynomial in the size of the input. Then the function
    \[
F : x \mapsto \det\begin{pmatrix}
f(g(x,1,1)) & \cdots & f(g(x,1,p(x)))
\\
\vdots&\ddots&\vdots
\\
f(g(x,p(x),1)) & \cdots & f(g(x,p(x),p(x)))
\end{pmatrix}
\]
is in $\GapL$.
\end{thm}
\begin{proof}
Let ${\det}_d$ be the polynomial
\[
{\det}_d
\begin{pmatrix}
y_{1,1} & \cdots & y_{1,d}
\\
\vdots & \ddots & \vdots
\\
y_{d,1} & \cdots & y_{d,d}
\end{pmatrix}
= \sum_{\pi\in S_d} \sgn(\pi) \prod_{i=1}^d y_{i,\pi(i)}.
\]
Define $y_\bot := 0$.
We crucially use the non-trivial fact (one can deduce this from the references in \cite{AO94}, see also the recent \cite{Ike25})
which is explicitly elaborated in \cite[\S3]{MV97}
that there exist
polynomials ${\det}_{d,+}(\mathbf{y})$ and ${\det}_{d,-}(\mathbf{y})$
with ${\det}_{d}={\det}_{d,+}-{\det}_{d,-}$
and
functions
\[
t_+:\N\times\N\times\N\times\N \to (\N\times\N)\cup\{\bot\},
\qquad
t_-:\N\times\N\times\N\times\N \to (\N\times\N)\cup\{\bot\}
\]
in $\FL$ such that
\[
{\det}_{n,+}\begin{pmatrix}
y_{1,1} & \cdots & y_{1,d}
\\
\vdots & \ddots & \vdots
\\
y_{d,1} & \cdots & y_{d,d}
\end{pmatrix}
= \sum_{\substack{1\leq q_0 \leq 2d\\1\leq q_1,\ldots,q_{d-1}\leq 2d^2}} \,
\prod_{i=1}^d
y_{t_+(d,i,q_{i-1},q_{i})}
\]
and
\[
{\det}_{d,-}\begin{pmatrix}
y_{1,1} & \cdots & y_{1,d}
\\
\vdots & \ddots & \vdots
\\
y_{d,1} & \cdots & y_{d,d}
\end{pmatrix}
= \sum_{\substack{1\leq q_0 \leq 2d\\1\leq q_1,\ldots,q_{d-1}\leq 2d^2}} \,
\prod_{i=1}^d
y_{t_-(d,i,q_{i-1},q_{i})}.
\]
We discuss the witnesses for ${\det}_{d,+}$, where the witnesses for ${\det}_{d,-}$ are analogous.
Let $n:=|x|$.
Let $d:=p(x) \in \poly(n)$ and $q_d:=1$.
A witness is given by a list $(q_0,q_1,w^1,q_2,w^2,\ldots,q_d,w^d)$, where $w^i$ is a witness for $f(g(x,t_+(d,i,q_{i-1},q_i)))$.
When traversing the witness, we have two counters in which we always store $q_{i-1}$ and $q_i$,
which can be done in space $O(\log(n))$, because $q_i\leq 2d^2\in\poly(n)$.
The witness~$w^i$ is then verified
via composition as in~\eqref{eq:FL composition}.
\end{proof}
One can take care of the signs and prove Theorem~\ref{thm:detofsharpL} for $f \in \GapL$ instead of $\sharpL$ with an analogous proof, but we do not need it in this paper.

\subsection{The logarithmic zoo}\label{sec:zoo}
Let $\sharpTISP(f(n),g(n))$ denote the of functions counting the number of accepting paths of an $\mathsf{NSPACE}(g(n))$ machine for which there exists a constant $c$ such that the machine takes at most $c\cdot f(n)$ steps on each computation path (TISP stands for TIme and SPace).

Let $\poly(n)=n^{O(1)}$ and $\qpoly(n)=n^{\poly(\log(n))}$,
and let $\FQP$ denote the set of functions computable in $\qpoly(n)$ time.
We will discuss the following counting complexity classes:
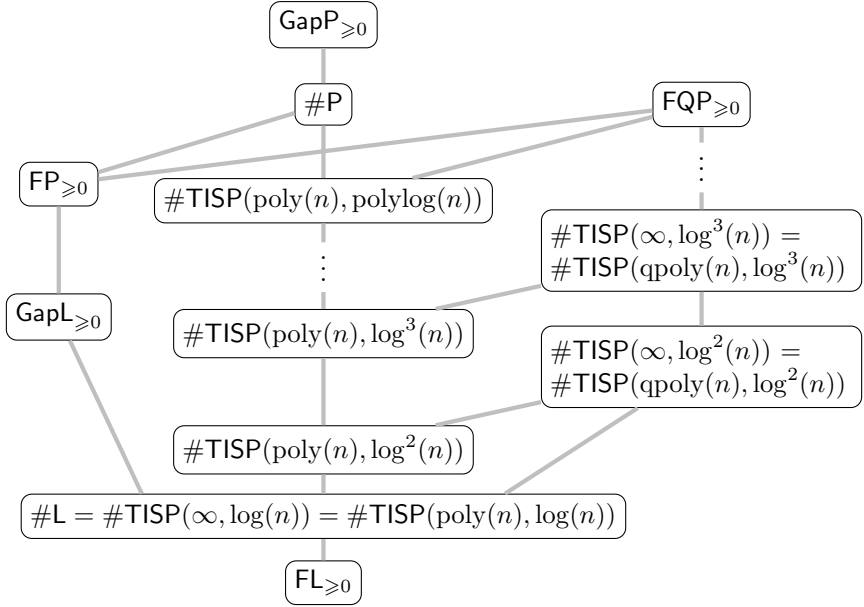
\[
\begin{tikzpicture}[y=2.5em]
    \node[draw,rounded corners] (FL) at (0,0) {$\FL_{\ge0}$};
    \node[draw,rounded corners] (sL) at (0,1) {$\sharpL = \sharpTISP(\infty,\log(n)) = \sharpTISP(\poly(n),\log(n))$};
    \node[draw,rounded corners] (sL2p) at (0,2) {$\sharpTISP(\poly(n),\log^2(n))$};
    \node[draw,rounded corners] (sL3p) at (0,3.75) {$\sharpTISP(\poly(n),\log^3(n))$};
    \node[outer sep=-.4em] (sLkp) at (0,4.75) {\raisebox{.5em}{$\vdots$}};
    \node[draw,rounded corners] (ULkp) at (0,5.75) {$\sharpTISP(\poly(n),\mathrm{polylog}(n))$};
    \node[draw,rounded corners] (sL2) at (5,3.25) {\begin{minipage}{4cm}$\sharpTISP(\infty,\log^2(n))$ $=$\\ $\sharpTISP(\qpoly(n),\log^2(n))$\end{minipage}};
    \node[draw,rounded corners] (sL3) at (5,5) {\begin{minipage}{4cm}$\sharpTISP(\infty,\log^3(n))$ $=$\\ $\sharpTISP(\qpoly(n),\log^3(n))$\end{minipage}};
    \node[outer sep=-.4em] (sLk) at (5,6.25) {\raisebox{.5em}{$\vdots$}};
    \node[draw,rounded corners] (GapL) at (-3.5,4) {$\GapL_{\ge0}$};
    \node[draw,rounded corners] (FP) at (-3.5,6) {$\FP_{\ge0}$};
    \node[draw,rounded corners] (sP) at (0,7.2) {$\sharpP$};
    \node[draw,rounded corners] (GapP) at (0,8.4) {$\GapP_{\ge0}$};
    \node[draw,rounded corners] (FQP) at (5,7.2) {$\FQP_{\ge0}$};
    \draw[ultra thick, lightgray] 
    (FL) -- (sL)
    ($(sL.north west)!0.8!(sL.north east)$) -- (sL2) -- (sL3) -- (sLk) -- (FQP)
    ($(sL.north west)!0.2!(sL.north east)$) -- (GapL) -- (FP) -- (sP) -- (GapP)
    (sL3p) -- (sLkp) -- (ULkp) -- (sP)
    (sL) -- (sL2p) -- (sL3p)
    (ULkp) -- (FQP)
    (FP) -- (FQP)
    (sL2) -- (sL2p)
    (sL3) -- (sL3p);
\end{tikzpicture}
\]

A major open question is whether the containment $\sharpL \subseteq \sharpP$ is strict.
\begin{thm}
    If $\sharpL = \sharpP$ then $\P=\NP$.
\end{thm}
\begin{proof}
    If $\sharpL \subseteq \FP_{\ge0} \subseteq \sharpP = \sharpL$ then $\FP_{\ge0}=\sharpP$.
    The class $\P$ can be identified with the subclass of $\FP_{\ge0}$ of functions with a 1-bit output. Given a function $f$ in $\sharpP$, the decision problem $f>^?0$ is at most as hard as computing $f$ itself, hence we may write $\NP\subseteq\sharpP=\FP_{\ge0}$. So $\NP$ is a subset of $\FP_{\ge0}$ composed of functions with a $1$-bit output, that is $\NP \subseteq \P$.
\end{proof}

We have $\sharpL\subseteq \FP$, but such a containment might not hold for $\sharpTISP(-,\log^2(n))$, as the following proposition shows.
\begin{prop}
If $\sharpTISP(\poly(n),\log^2(n))\subseteq \FP$, then \textup{\textsc{\#3Sat}} can be solved in time $2^{O(\sqrt{n})}$. This violates the exponential time hypothesis (ETH).
\end{prop}
\begin{proof}
Let $\ast^a$ denote a sequence of $a$ many arbitrary symbols.
Consider the following counting function \textsc{\#padded3Sat}.
\[
\textsc{\#padded3Sat}(\varphi,\ast^{2^{\sqrt{n}}-|\varphi|-1})
 = \textsc{\#3Sat}(\varphi),
\]
where $\varphi$ denotes the binary encoding of a Boolean formula on $n$ variables and $n$ clauses.
The input length is $N := 2^{\sqrt{n}}$,
and a witness is a string of $n=(\log(N))^2$ many bits that represents a truth assignment.
Checking a witness can be done by copying it to the work tape and then going over all clauses in $\varphi$. This can be done in polynomial time.
Hence, $\textsc{\#padded3Sat}\in\sharpTISP(\poly(n),\log^2(n))$.
By assumption we have $\sharpTISP(\poly(n),\log^2(n))\subseteq\FP$, so there exists $c$ such that $\textsc{\#padded3Sat}$ can be solved in time $N^c = 2^{c\sqrt{n}}$.
Hence, also $\textsc{\#3Sat}$ can be solved in time $2^{c\sqrt{n}}$, using the same algorithm.
\end{proof}

\begin{note}
For Kronecker and plethysm coefficients the problem whether they are non-zero is already $\NP$-hard (even if the input is encoded in unary) \cite{IMW,FI}, hence if their computation is in $\GapL$ or $\sharpL$, which are subsets of $\FP$, it follows that $\P=\NP$. In our Theorem~\ref{thm:GL2 in L2}, we study  important subfamilies of inputs where the coefficients are in $\GapL$.
\end{note}

\subsection{Worst-case running time for log-space machines}
To separate computational classes, we can check what is the largest growing function computable in a given class.
\begin{thm}
\label{thm:sharpLgrowth}
    If $f\in\sharpL$ then $\exists c : f \in O(2^{n^c})$.
\end{thm}
\begin{proof}
An $\NL$ machine with $c\log(n)$ space in the work tape can have at most $2^{c\log(n)} = n^c$ many configurations. Since the machine is not allowed to loop endlessly, the whole configuration graph must have no directed cycles. Since the non-deterministic guesses of non-deterministic Turing machines are in binary, a configuration can only lead to one of two possible configurations.\medskip

We claim that the graph with the most possible source-sink paths on $N$ vertices and with the above contraints is the following:
\[
\bullet \rightrightarrows \bullet \rightrightarrows \bullet \cdots \bullet \rightrightarrows \bullet
\]
Indeed, suppose by induction that it is true for graphs on $N-1$ vertices. Note that it has $2^{N-1}$ source-sink paths. Consider a graph on $N$ vertices now. Let $s$ be the source and $a, b$ be its children; let $t$ be the sink. By induction hypothesis, the number of $a$-$t$ paths is at most $2^{N-1}$, and the number of $b$-$t$ too by a similar argument. Since there are no cycles, there are at most $2^N$ many $s$-$t$ paths. The above graph realizes this maximum. \medskip

Since $\sharpL$ is the class of functions that count paths in the configuration graph, this count is bounded above by $2^{n^c}$.
\end{proof}

\begin{thm}\label{thm:FL upper bound}
    If $f\in\FL$ then $f \in O(2^{n^c})$. Moreover, there exists $f\in\FL$ such that $f\in\Theta(2^{n^c})$.
\end{thm}
\begin{proof}
    The upper bound is given by the previous theorem. The existence is given by the following construction. Consider the machine that takes as input $n$, computes and stores $n^c$ in a counter $\mathtt{a}$ (occupying $c\log(n)$ bits). Then outputs a single $1$, and while $\mathtt{a} > 0$ it outputs a $0$ then decreases $\mathtt{a}$ by~$1$.
\end{proof}

The following lemma will be useful.
\begin{lem}
\label{lem:elementary}
The functions $(x,y)\mapsto x+y$ and $(x,y)\mapsto xy$ are in $\FL$, where the input and output are encoded in binary.
\end{lem}
\begin{proof}
The elementary school algorithm suffices, with a minor modification for the multiplication: Instead of adding up all shifted products of a factor with the digits of the other factor, an additions is performed after every product of a digit with a factor.
\end{proof}

\section{Enumerative combinatorics}\label{sec:combinatorics}

We begin by considering a problem in greater depth, in order to illustrate the remaining proofs of this section.
\begin{thm}[Binomial coefficients]\label{thm:binomial}
The function $(1^n \, 0 \, 1^k)\mapsto \binom{n}{k}$ is in $\sharpL$.
\end{thm}
In the definition of the function, $1^n$ denotes the string $11\!\stackrel{n}{\ldots}\!1$, that is, a sequence of $n$ many 1s. Any string which is not of the form $11\!\stackrel{n}{\ldots}\!1011\!\stackrel{k}{\ldots}\!1$  is assumed to be mapped to $0$ by the function.
\begin{proof}
    The binomial coefficient $\binom{n}{k}$ counts the number of $k$-subsets of $[n]$. 
    Encode a $k$-subset $S$ of $[n]$ as a word $w_1w_2\ldots w_n$ where $w_i = 1$ if $i\in S$ and $0$ otherwise.
    We construct a verifier $\NL$ machine with $2$ binary-encoded counters that takes as input $(1^n \, 0 \, 1^k)$ and a read-once left-to-right witness $w$ as a $\{0,1\}$-string, and accepts if and only if $w$ is encoding a $k$-subset of $n$.

    The first counter $\mathtt{a}$ is allocated $\log(n)$ space, and the second counter $\mathtt{b}$ is allocated $\log(k)$ space. Note that both are in $O(\log(n+k+1))$-space.
    
    We begin by reading $1^n$ and $1^k$ and storing their values in the two counters $\mathtt{a}$ and $\mathtt{b}$.
    We then read $w$ once left-to-right. For each $w_i$, the first counter always decreases by one, and the second counter decreases by one only if $w_i = 1$. After reading the whole witness $w$, the machine accepts if and only if both counters are at~$0$.
\end{proof}

The pseudo-code of the above machine is included in Algorithm~\ref{alg:binom}. This verifier $\NL$ machine has been implemented in Ugarte's Turing Machine Simulator~\cite{TMsimulator} and is available at \url{https://turingmachinesimulator.com/shared/ubutbvatkh}.
    \begin{algorithm}
    \caption{Verifier $\NL$ machine for binomial coefficients}\label{alg:binom}
    \begin{algorithmic}
    \Require $(1^n\,0\,1^k), w$
    \State {Initialize $\mathtt{a}=n$ and $\mathtt{b}=k$ (encoded in binary)}
    \For{$w_i \in w$}
    \State {Decrease $\mathtt{a}$ by~$1$}
    \InlineIfThen{$w_i=1$}
        {Decrease $\mathtt{b}$ by~$1$}\EndInlineIf
    \EndFor
    \State {Check $\mathtt{a} = 0 = \mathtt{b}$}
    \end{algorithmic}
    \end{algorithm}

By sequentially iterating this same basic algorithm, we can construct a verifier $\NL$ machine for multinomial coefficients.

\begin{thm}[Multinomial coefficients]\label{thm:multinomial}
    The function $(1^n\, 0\, 1^{a_1}\, 0\,\dots\, 0\, 1^{a_k}) \mapsto \binom{n}{a_1,\ldots,a_k}$ is in $\sharpL$. Here $k$ is not assumed to be fixed.
\end{thm}
\begin{proof}
    We use the formula $\binom{n}{a_1,\ldots,a_k}= \binom{n}{a_1} \binom{n-a_1}{a_2}  \cdots = \binom{m_1}{a_1}\binom{m_2}{a_2}\cdots$ to reduce the problem to the setting of Theorem~\ref{thm:product}. 

    The function $f : (1^n\,0\,1^k) \mapsto \binom{n}{k}$ is in $\sharpL$ by Theorem~\ref{thm:binomial}.
Next, we claim the function
\[g : (x \,0\,0\, i) \mapsto (1^{n-a_1-\cdots-a_{i-1}}\,0\,1^{a_i})
\]
is in $\FL$,
where $x=(1^n\, 0\, 1^{a_1}\, 0\,\dots\, 0\, 1^{a_k})$, $k\leq n$, and
$1\le i\le k$ is encoded in binary. Initialise binary-encoded counters $\mathtt{A} = 2n$ and $\mathtt{b} = i$. Read the input $x$ from the left until we read a $0$, then decrease $\mathtt{b}$ by~$1$ and keep reading. When $\mathtt{b} = 0$ stop; we have found the $i$th $0$ in $x$. Now start reading from the current position and to the left; every time we find a $1$ decrease $\mathtt{A}$ by~$1$. When you reach the beginning of the input tape we have $\mathtt{A}=n-a_1-\cdots-a_{i-1}$. Write a $1$ in the output tape, decrease $\mathtt{A}$ by~$1$, repeat until $\mathtt{A}=0$. Then write a $0$ in the output. Use the same technique to find the $i$th $0$ of $x$ again, and read to the right to set $\mathtt{A} \leftarrow a_i$. Output $\mathtt{A}$ in unary again to finish.

    By similar considerations, the function $p : x\mapsto k$ (the number of $0$s in $x$, written in binary) is in $\FL$ and its output is $O(\log(|x|))$ in size. Hence we are in the hypotheses of Theorem~\ref{thm:product}.
\end{proof}

The machine constructed for binomial coefficients is part of a very simple family of combinatorial log-space machines. We include three more examples.

\begin{thm}[Catalan numbers]\label{thm:catalan}
The function $(1^n) \mapsto \frac{1}{n+1}\binom{2n}{n}$ is in $\sharpL$.
\end{thm}
\begin{proof}
    Catalan numbers count the number of Dyck paths of length $2n$.
    We construct an algorithm of the type of Algorithm~\ref{alg:binom}.
    Encode a Dyck path as a binary string by interpreting $1$ as an up-step and $0$ as a down-step. Initialise two binary-encoded counters: $\mathtt{a}=2n$ will check the length of the string and $\mathtt{b}=0$ will check that the Dyck path stays at non-negative height. For each bit $w_i$ in the witness: if $w_i=1$ then decrease $\mathtt{a}$ by one, increase $\mathtt{b}$ by one; if $w_i=0$ then decrease both counters by one. Then check that both counters are at least $0$.

    At the end, check that $\mathtt{a}=0=\mathtt{b}$.
\end{proof}
\begin{note}
    The machine for Catalan numbers is generalised in Theorem~\ref{thm:ballots}.
\end{note}

\begin{thm}[Narayana numbers]
The function $(1^n\,0\,1^k) \mapsto \frac{1}{n}\binom{n}{k}\binom{n}{k-1}$ is in~$\sharpL$.
\end{thm}
\begin{proof}
    The Narayana number $N(n,k)$ counts the number of Dyck paths of length $2n$ with exactly $k$ peaks.
    We take the algorithm from Theorem~\ref{thm:catalan} and make a small modifications to it.
    We add one more binary-encoded counter $\mathtt{c}$ initialised at $k$, which will count the number of peaks. We also add a $1$-bit counter $\mathtt{d}$ to detect these peaks.
    
In the main loop, if $w_i=1$ then set $\mathtt{d}=1$; if $w_i=0$ and $\mathtt{d}=0$ then do nothing; if $w_i=0$ and $\mathtt{d}=1$ then set $\mathtt{d}=0$ and decrease $\mathtt{c}$ by one.

    At the end, check also that $\mathtt{c}=0=\mathtt{d}$.
\end{proof}

\begin{thm}[Fibonacci numbers]
The function $(1^n) \mapsto F_n$ is in $\sharpL$.
\end{thm}
\begin{proof}
    The Fibonacci number $F_n$, with $F_1=1,F_2=2$, counts domino sequences of $n$: the number of $\{1,2\}$-strings which sum to $n$. If each $1$ and $2$ is written in binary, then a domino sequence is a $\{0,1\}$-string of length $n$ such that each $0$ is preceded by a $1$; this will be our witness.
    We construct an algorithm of the form of Algorithm~\ref{alg:cap}.

    Initialise a binary-encoded counter $\mathtt{a}=n$ and a $1$-bit counter $\mathtt{b}=0$ which we use to remember the last bit encountered. For each $w_i$, decrease $\mathtt{a}$ by one. If $w_i=1$ 
    then set $\mathtt{b}=1$. If $w_i=0$ then check $\mathtt{b}=1$ (otherwise reject) and set $\mathtt{b}=0$. 

    At the end, accept if $\mathtt{a}=0$, otherwise reject.
\end{proof}

The machine for Fibonacci numbers is a special case of the following machine.

\begin{thm}\label{thm:recurrence}
    Let $A(1), A(2), \ldots$ be a sequence of non-negative integers satisfying a linear recurrence
    \[
    A(n)= \sum_{i=1}^C D(n,i) A({n-i})
    \]
    with a constant number $C$ of terms and whose coefficients $D(n,i)$ are non-negative integers, of $\mathrm{poly}(n)$ size, and such that $(1^n\,0\,1^i)\mapsto D(n,i)$ is in $\FL_{\ge0}$. 
    Then the function $(1^n) \mapsto A(n)$ is in $\sharpL$. Here $A(1), \ldots, A(C)$ are given constants.
\end{thm}
\begin{proof}
    Consider the directed multigraph whose vertices are the non-negative integers and with
    \begin{itemize}
        \item $D(k,i)$ arcs from $k$ to $k-i$ for each $k>C$ and $1\le i\le C$,
        \item $A(k)$ arcs from $k$ to $0$ for each $1\le k \le C$.
    \end{itemize}

    Then $\sum_{i=1}^C D(n,i) A({n-i}) = A(n)$ counts the number of paths from $n$ to $0$. 
    For instance, the following is the graph associated to the Fibonacci numbers $\{F_n\}_{n\ge1}$, where the base cases are $n=1$ and $n=2$ for consistency:
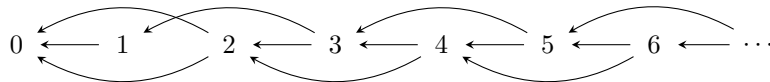
\[
\begin{tikzpicture}[x=4em]
    \foreach\i in {0,...,6}{
        \node[circle] (\i) at (\i,0) {$\i$};
    }
    \node[circle] (7) at (7,0) {$\cdots$};
    \draw[-stealth] (1) to (0);
    \draw[-stealth] (2) to[bend left] (0);
    \draw[-stealth] (2) to[bend right] (0);
    \draw[-stealth] (3) to (2);
    \draw[-stealth] (3) to[bend right] (1);
    \draw[-stealth] (4) to (3);
    \draw[-stealth] (4) to[bend left] (2);
    \draw[-stealth] (5) to (4);
    \draw[-stealth] (5) to[bend right] (3);
    \draw[-stealth] (6) to (5);
    \draw[-stealth] (6) to[bend left] (4);
    \draw[-stealth] (7) to (6);
    \draw[-stealth] (7) to[bend right] (5);
\end{tikzpicture}
\]

A path $(p_0 = n, p_1, p_2, \ldots, p_\ell, 0)$ from $n$ to $0$ can be encoded with a tuple of numbers
    \[
    i_1, d_1, i_2, d_2, \ldots, i_\ell, d_\ell, a
    \]
such that the $1\le i_j \le C$
encode the jump sizes from a given node $p_{j-1}$ of the path to the next node $p_j = p_{j-1}-i_j$, the
$1 \le d_j \le D(p_{j-1},i_j)$ choose which of the arcs to use for this one step, and $1\le a\le A(p_{\ell})$ chooses which of the arcs to take from $p_{\ell}$ to $0$. In the case of Fibonacci numbers we have $d_j=1$ always, and the tuple of the $i_j\in\{1,2\}$ is a domino sequence.

A sequence encodes a valid path if and only if
\[
\begin{cases}
    1\le i_j \le C & \text{for all~}1\le j\le \ell,\\
    1\le d_j \le D(p_{j-1},i_j) & \text{for all~}1\le j\le \ell,\\
    1\le a\le A(p_{\ell}),\\
    1 \le p_\ell \le C < p_{\ell-1}.
\end{cases}
\]
Note that these properties can be checked sequentially: first check the inequalities for $j=1$, then for $j=2$, etc. Two counters can be used to keep track of $p_{j-1}$ and $p_j = p_{j-1} - i_j$ at any given time.
The parameters $C$ and $A(1),\ldots,A(C)$ are constants; the parameters $D(k,i)$ are $\mathrm{poly}(k)$ size (and so in particular $\mathrm{poly}(n)$ size) and computable in $O(\log(k+i)) = O(\log(n))$ space.
Therefore all of these properties can be checked in log-space using an algorithm like the one in Algorithm~\ref{alg:cap}.
\end{proof}

\begin{cor}[Factorial]\label{cor:factorial}
    The function $(1^n) \mapsto n!$ is in $\sharpL$.
\end{cor}
\begin{proof}[First proof of Corollary~\ref{cor:factorial}]
    We have $n! = n\cdot(n-1)!$ and thus it follows from the previous theorem.
\end{proof}

\begin{proof}[Second proof of Corollary~\ref{cor:factorial}]
    The factorial is in the hypotheses of Theorem~\ref{thm:product}, where $f = p : (1^n) \mapsto n$ is in $\FL_{\ge0}$ and poly-size, and $g : (1^n\,0\,1^i) \mapsto i$ in $\FL$.
\end{proof}

\begin{note}
In the first proof of Corollary~\ref{cor:factorial}, the graph for $n!$ is the following:
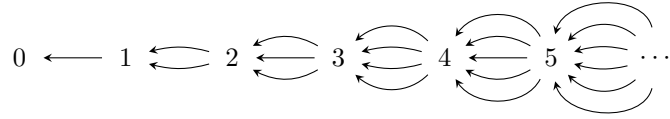
\[
\begin{tikzpicture}[x=4em]
    \foreach\i in {0,...,5}{
        \node[circle] (\i) at (\i,0) {$\i$};
    }
    \node[circle] (6) at (6,0) {$\cdots$};
    \draw[-stealth] (1) to (0);
    \draw[-stealth] (2) to[bend left = 15] (1);
    \draw[-stealth] (2) to[bend right = 15] (1);
    \draw[-stealth] (3) to[bend left = 30] (2);
    \draw[-stealth] (3) to (2);
    \draw[-stealth] (3) to[bend right = 30] (2);
    \draw[-stealth] (4) to[bend left = 45] (3);
    \draw[-stealth] (4) to[bend left = 15] (3);
    \draw[-stealth] (4) to[bend right = 15] (3);
    \draw[-stealth] (4) to[bend right = 45] (3);
    \draw[-stealth] (5) to[bend left = 63] (4);
    \draw[-stealth] (5) to[bend left = 30] (4);
    \draw[-stealth] (5) to (4);
    \draw[-stealth] (5) to[bend right = 30] (4);
    \draw[-stealth] (5) to[bend right = 63] (4);
    \draw[-stealth] (6) to[bend left = 80] (5);
    \draw[-stealth] (6) to[bend left = 45] (5);
    \draw[-stealth] (6) to[bend left = 15] (5);
    \draw[-stealth] (6) to[bend right = 15] (5);
    \draw[-stealth] (6) to[bend right = 45] (5);
    \draw[-stealth] (6) to[bend right = 80] (5);
\end{tikzpicture}
\]
The witness that one constructs is a Lehmer code, namely a list $(a_1, a_2, \ldots, a_n)$ such that $1\le a_i \le i$ for all $i$.

In the second proof, one constructs the same witness.
\end{note}
\smallskip

The above recursion only applies to sequences. We can show a similar result for for $2$-dimensional or higher dimensional recursions. The quintessential example are the binomial coefficients, which are subject to the Pascal identity
\(\binom{n}{k} = \binom{n-1}{k} + \binom{n-1}{k-1}\). The binomial coefficients are in $\sharpL$ by Theorem~\ref{thm:binomial}.
\begin{thm}\label{thm:2-dim}
    Let $A(1,1), A(1,2), A(2,1), \ldots$ be a 2-dimensional array of non-negative integers satisfying a linear recurrence
    \[
    A(n,k) = \sum_{\ell=1}^C D(n,k,i_\ell,j_\ell) A(n-i_\ell,k-j_\ell)
    \]
    for some fixed set of $C$ pairs $(i_1,j_1), \ldots, (i_C,j_C)$ such that
    \((n-i_\ell,k-j_\ell)<_{\mathrm{lex}}(n,k)\)
    for all $n,k,$ and $\ell$.
    Suppose that the coefficients $D(n,k,i,j)$ are positive integers of $\mathrm{poly}(n+k)$ size and the function $(1^n\,0\,1^k\,0\,1^i\,0\,1^j)\mapsto D(n,k,i,j)$ is in $\FL_{\ge0}$.
    Then the function $$(1^n\,0\,1^k) \mapsto A(n,k)$$ is in $\sharpL$. 
    Here, the set of base cases $\{A(x,y)\mid (x,y)\in B\}$ is given and its elements are constants.
\end{thm}
\begin{proof}
    The witness is a sequence $(n_0,k_0,d_0),(n_1,k_1,d_1),\ldots,(n_m,k_m,d_m)$, where the numbers $n_i,k_i,d_i$ are encoded in binary. They are subject to
    \[
    \begin{cases}
        (n_0,k_0) = (n,k)\\
        \exists \ell=\ell(s) : (n_s,k_s) = (n_{s-1}-i_\ell, k_{s-1}-j_\ell) &\text{for all}~1\le s < m,\\
        1\le d_{s-1} \le D(n_{s-1},k_{s-1},i_{\ell(s)},j_{\ell(s)})&\text{for all}~1\le s < m,\\
        (n_m,k_m)\in B,\\
        1\le d_m \le A(n_m,k_m).
    \end{cases}
    \]
    
    We keep 4 main counters $\mathtt{a},\mathtt{b},\mathtt{c},\mathtt{d}$. 
    Initialise $(\mathtt{a},\mathtt{b}) \leftarrow (n_0, k_0)$ and check that they are equal to $n$ and $k$ respectively.
    In the $s$th step, the counters $(\mathtt{a},\mathtt{b})$ store $(n_{s-1}, k_{s-1})$ and we read  $(\mathtt{c},\mathtt{d}) \leftarrow (n_s, k_s)$.
    Check the local validity of the witness: for $\ell=1,\ldots,C$ check if $(\mathtt{c},\mathtt{d}) = (\mathtt{a} - i_\ell, \mathtt{b} - j_\ell)$. If all checks fail, then reject the witness. If on the contrary a pair $(i_\ell,j_\ell)$ is found, compute and store $\mathtt{D}\leftarrow D(n,k,i_\ell,j_\ell)$ and check that $1\le d_{s-1}\le \mathtt{D}$. If $(\mathtt{c},\mathtt{d})\in B$ is a base case, then check $1\le d_s \le A(\mathtt{c},\mathtt{d})$ and that this is the last entry of the witness.
    Otherwise reassign $(\mathtt{a}, \mathtt{b}) \leftarrow (\mathtt{c}, \mathtt{d})$ and repeat.
\end{proof}

The above theorem places many more well-known functions in $\sharpL$. We highlight two that are of interest later in the paper.

\begin{cor}[Stirling numbers]\label{cor:stirling}
The functions sending $(1^n\,0\,1^k)$ to the
Stirling numbers of the first type $c(n,k)$ and second type $S(n,k)$ are both in $\sharpL$.
\end{cor}
\begin{proof}
    They are subject to the $2$-dimensional recursions
    \[
    c(n+1,k) = n\, c(n,k) + c(n,k-1)
    \quad\text{and}\quad
    S(n+1,k) = k\, S(n,k) + S(n,k-1),
    \]
    and thus the result follows from Theorem~\ref{thm:2-dim}.
\end{proof}
\begin{note}
    The witnesses obtained for the Stirling numbers are classical combinatorial interpretations.     
    On the one hand, by \cite[Corollary~2.2]{SaganLogConcave}, we can express the Stirling numbers of the first kind as the specialisation of some elementary symmetric polynomial,
    \(c(n,k) = e_{n-k}(1,2,\ldots,n-1)\).
    Hence Stirling numbers of the first kind count sequences
    \[
    T_1, a_1, T_2, a_2, \ldots, T_{n-k}, a_{n-k}
    \]
    such that $0 < T_1 < T_2 < \cdots < T_{n-k} < n$ and $1\le a_i \le T_i$ for all $i$.

    On the other hand, Corollary~\ref{cor:stirling} implies that $c(n,k)$ counts sequences
    \[
    n, k_0, d_0, n-1, k_1, d_1, n-2, k_2, d_2,\ldots
    \]
    where $n-s\ge k_s\ge1$, we have $k_{s+1}\in\{k_s, k_s+1\}$, and where $d_s$ is $1$ if $k_{s+1}=k_s$ or a number between $1$ and $n-s$ otherwise.

    These two sets are in bijection, where the $T_i$ are precisely the values $k_{s+1}$ where there is a jump, $k_{s+1} \ne k_s$, and the $a_i$ get mapped to the $d_s$.

    The situation with Stirling numbers of the second kind is analogous.
\end{note}

\begin{thm}[Ballot sequences]\label{thm:ballots}
Let $C$ be a constant and let $\mathbf{a}:=(a_1 \geq \cdots \geq a_C)$ be a sequence of non-negative integers encoded in unary. Let $b_n(\mathbf{a})$ be the number of $\mathbf{a}$-shifted ballot sequence $w_1,\ldots,w_n$ with $w_i\in\{1,\ldots,C\}$, that is, sequences such that for every $1\leq i < C$ and $j\leq n$ we have 
$$\#\{k\mid w_k =i,\, k\leq j\} +a_i \geq \#\{k\mid w_k=i+1,\, k\leq j\}+a_{i+1}$$
for every $k$. Then the function $(\mathbf{a}, 1^n) \to b_n(\mathbf{a})$ is in $\sharpL$.
\end{thm}
\begin{proof}
    Initialize a counter $\mathtt{j}=0$
    and $C$ many counters $\mathtt{b}_i = a_i$
    for $i=1,\ldots,C$ encoded in binary in $O(\log(n+|\mathbf{a}|))$ space. The witness is the sequence $w_1,\ldots,w_n$ encoded with the $C$ symbols: $1,2,\ldots,C$.  We read the witness symbol by symbol. 
    At each new symbol $w_j$ increase $\mathtt{j}$ by one (so that $\mathtt{j} = j$). Set $i:=w_{\mathtt{j}}$ and increase $\mathtt{b}_i$ by one, then, if $i>1$, check if $\mathtt{b}_i\leq \mathtt{b}_{i-1}$. If that check fails then the witness is not a ballot sequence. When the witness has been fully read, check $\mathtt{j}=n$.
\end{proof}

The \emph{Young diagram} of a partition $\lambda$ is the set $[\lambda]=\{(i,j)\mid 1\le j \le \lambda_i\}$, which we represent by the English convention \cite[page 29]{EC1}. The \emph{transpose} of $\lambda$ is the partition with Young diagram $[\lambda']=\{(j,i)\mid (i,j)\in[\lambda]\}$.
If $\lambda\vdash n$, a \emph{standard Young tableaux} of shape $\lambda$ is a bijection $T:[\lambda]\to[n]$ which is increasing along the rows and down the columns. Let $\SYT(\lambda)$ be the set of standard Young tableaux of shape~$\lambda$. The cardinality of $\SYT(\lambda)$ is given by the \emph{hook-length formula} \cite[Corollary 7.21.6]{StanleyEC2}
    \[
    \#\SYT(\lambda) = \frac{n!}{\prod_{c\in[\lambda]} \hl_\lambda(c)}
    \]
where $\hl_\lambda(i,j) = \lambda_i+\lambda_j' - i - j + 1$.
As a corollary of the previous theorem, we obtain $\#\SYT(\lambda)$ in $\sharpL$ when the length (or the width) of $\lambda$ is bounded. In Theorem~\ref{thm:hook-length} we will see how to place $\#\SYT(\lambda)$ in $\sharpL$ without any constraints.
\begin{cor}[SYTs of fixed length]\label{cor:syt_fixed_ell}
    Let $C$ be a fixed constant and let $\lambda$ be a partition of $n$ with $C$ parts. 
    The function $(1^{\lambda_1}\,0\,1^{\lambda_2}\,0\,\dots \,0\,1^{\lambda_C}) \mapsto \# \SYT(\lambda)$ is in $\sharpL$.
    Similarly, let $\lambda$ be a partition of $n$ with $\lambda_1\le C$. 
    The function $(1^{\lambda_1}\,0\,1^{\lambda_2}\,0\,\dots) \mapsto \# \SYT(\lambda)$ is in $\sharpL$.
\end{cor}
\begin{proof}
    Standard Young tableaux of a given shape are in correspondence to ballot sequences as follows. An SYT $T$ gives the ballot sequence $w(T)$ where $w(T)_i$ is equal to the row number of the row where the box with entry $i$ in $T$ is. Apply the algorithm from Theorem~\ref{thm:ballots} with $a_i=0$ and adding a check that $\mathtt{b}_i=\lambda_i$ in the end. 

    If the width of $\lambda$ is bounded instead, then change the role of \emph{rows} for \emph{columns} in the previous correspondence. Compose the $\FL$ function from Lemma~\ref{lem:transposition} and the verifier $\NL$ machine of Theorem~\ref{thm:ballots} as in~\eqref{eq:FL composition}.
\end{proof}
\begin{lem}[Transposition]\label{lem:transposition}
    Fix a constant $C$. The transposition of partitions of length at most $C$ is a log-space operation. More precisely, the function
    \[
    (1^{\lambda_1}\,0\,1^{\lambda_2}\,0\,\dots\,0\,1^{\lambda_C}\,0) 
    \mapsto
    (1^{\lambda'_1}\,0\,1^{\lambda'_2}\,0\,\dots\,0\,1^{\lambda'_{\lambda_1}}\,0)
    \]
    is in $\FL$. Similarly, the transposition of partitions of width at most $C$ is a log-space operation.
\end{lem}
\begin{proof}
    Suppose $\ell(\lambda)\le C$.
    The output of the function can be rewritten as 
    \[
    \Big(\underbrace{
    \underline{1^C\,0}~
    \underline{1^C\,0}~
    \dots
    \underline{1^C\,0}
    }_{\lambda_C~\text{times}}
    \,
    \dots
    \,
    \underbrace{
    \underline{1^2\,0}~
    \underline{1^2\,0}~
    \dots
    \underline{1^2\,0}
    }_{\lambda_2-\lambda_3~\text{times}}
    \,
    \underbrace{
    \underline{1\,0}~
    \underline{1\,0}~
    \dots
    \underline{1\,0}
    }_{\lambda_1-\lambda_2~\text{times}}
    \Big)\,.
    \]
    Initialise $C$ counters $\mathtt{a}_1 = \lambda_1, \dots, \mathtt{a}_C = \lambda_C$.
    While $\mathtt{a}_C > 0$, write $1^C\,0$ in the output tape and decrease all counters by one. Once the while-loop is finished, delete the counter $\mathtt{a}_C$.
    At this stage we have $\mathtt{a}_i = \lambda_i-\lambda_C$ for all $i$.
    Run a new while-loop, this time controlled by $\mathtt{a}_{C-1} > 0$, and writing $1^{C-1}\,0$ in each iteration. After this second while-loop, $\mathtt{a}_i = \lambda_i-\lambda_{C-1}$ for all $i$.
    Iterate this process until all counters have been deleted.\medskip

    Suppose now $\lambda_1\le C$. Initialise a counter $\mathtt{a}=1$. For each run of $1$s in the input, if the run is of length at least $\mathtt{a}$ then write a $1$ in the output tape. Then write a $0$ in the output, increase $\mathtt{a}$, and repeat. Once $\mathtt{a} > \lambda_1$, stop.
\end{proof}

A permutation $\sigma$ is an involution if $\sigma^2$ is the identity.
The number of permutations of $S_n$ that are involutions is given by $\sum_{\lambda\vdash n} \#\SYT(\lambda)$ \cite[page~52]{Fulton}.
\begin{cor}[Number of involutions of $S_n$]
    The function $(1^n) \mapsto \sum_{\lambda\vdash n} \#\SYT(\lambda)$ is in~$\sharpL$.
\end{cor}
\begin{proof}
    This function satisfies the recursion
    \(
    F(n) = F(n-1) + (n-1)F(n-2)
    \)
    and so the claim follows from Theorem \ref{thm:recurrence}.
\end{proof}

A \emph{semistandard Young tableau}  of shape $\lambda$ and entries in $[N]$ is a map $T:[\lambda]\to[N]$ such that $T(i,j) \le T(i,j+1)$ and $T(i,j) < T(i+1,j)$ whenever the expressions are defined.
The  set $\SSYT_N(\lambda)$ of semistandard Young tableaux of shape $\lambda$ and entries in $[N]$ is counted by the \emph{hook-content formula} \cite[Corollary 7.21.4]{StanleyEC2}
  \[\#\SSYT_N(\lambda) = \prod_{c\in[\lambda]}  \frac{N+\ct(c)}{\hl_\lambda(c)}\]
where $\ct(i,j) = j-i$.
\begin{cor}[$\#\SSYT$s of shape $\lambda$ and bounded entries]\label{cor:ssyt_fixed_ell}
  Let $C$ be a fixed constant and let $\lambda$ be a partition of $n$ with $C$ parts. Let $N$ be an integer no larger than $n$.
  Then the function $(1^N\,0\,1^{\lambda_1}\,0\,1^{\lambda_2}\,0\,\dots\,0\,1^{\lambda_C}) \mapsto \# \SSYT_N(\lambda)$ is in $\sharpL$.

  Similarly, let $\lambda$ be a partition with $\lambda_1\le C$, then the function $(1^N\,0\,1^{\lambda_1}\,0\,1^{\lambda_2}\,0\,\dots) \mapsto \# \SSYT_N(\lambda)$ is in $\sharpL$.
\end{cor}
\begin{proof}
    Suppose that the length of $\lambda$ is bounded by $C$.
    It is well known that SSYTs of shape $\lambda$ and entries $\{1,\ldots,m\}$ correspond via \emph{standardization} to SYTs of shape $\lambda$ with at most $m-1$ descents and set $R \subset [n] $ of $m$ elements, such that the descents of $T$ are a subset of $R$ and $n\in R$.
    Here $i$ is a descent of $T$ if it appears in a higher row than $i+1$ in $T$. Thus SSYTs with entries $\leq N$ are in bijection with 
    \[
    \Big\{(S,R, T) \,\Big|\, 
    \begin{array}{ll}
    S \subseteq [N], ~R \subseteq [n],~ |S|=|R| \geq \ell(\lambda), \\[-.6em]
          T\in\SYT(\lambda),~\text{ descents of $T$ are $\subseteq R$}
    \end{array} \Big\}.
    \]
    The correspondence is as follows: Given an SSYT $P$, let $S=\{s_1,\ldots,s_m\}$ be the labels of the elements which appear at least once in $P$. Then $T$ is the standardization of $P$: reading the smallest entries with values $s_1$ left-to-right, replace them with $1,2,\ldots,a_1$ in that order and put $a_1$ in the set $R$. Then read the elements equal to $s_2$, left-to-right, replace them with $a_1+1,\ldots,a_2$, add $a_2$ to $R$ and continue. Since entries of the same values in $Q$ form horizontal strips, their standardized values would not form any descents in $T$.
    
    Here our witness is thus an indicator sequence giving the set $S$ and a ballot sequence $w$ corresponding to an SYT $T$ as in Corollary~\ref{cor:syt_fixed_ell} with some symbols `$r$' between entries in $w_i$. Set a counter $\mathtt{m}=|S|$. While reading the ballot sequence $w$ we set $\mathtt{m} \leftarrow \mathtt{m}-1$ every time $w_{i+1}>w_{i}$ as this is equivalent to $i$ being a descent of the corresponding $T$ or if we have the symbol `$r$' written between $w_i$ and $w_{i+1}$ with $w_i\geq w_{i+1}$, which corresponds to an entry in the set $R$ that does not come from a descent (if `$r$' is between $w_i>  w_{i+1}$ we consider the witness invalid). In the end we should have $\mathtt{m}=1$. \medskip

    Suppose now the width of $\lambda$ is bounded instead. Then encode the tableau $T$ by its column ballot sequence $w$ instead. Now there is a descent if $w_i \geq w_{i+1}$. The rest of the algorithm works similarly.
\end{proof}

While on the topic of hook-length formulas, we consider the following result of Knuth~\cite[\S5.1.4, Ex.~20]{Knuth}, which expresses the number of linear extensions of a rooted tree poset $T$ (in computer science this is also called the number of topological orderings of a rooted tree $T$) as 
    \[
    e(T) = \frac{n!}{\prod_{v\in T} \hl_T(v)},
    \]
where the \emph{hook-length} $\hl_T(v)$ of a node $v$ is the number of descendants (including the node itself).

A plane rooted tree $T$ can be encoded by its Depth-First-Search traversal $\mathrm{dfs}(T)$, a sequence of $0$s and $1$s in which each $1$ represents going down an edge away from the root and a $0$ represents returning towards the root. For instance,
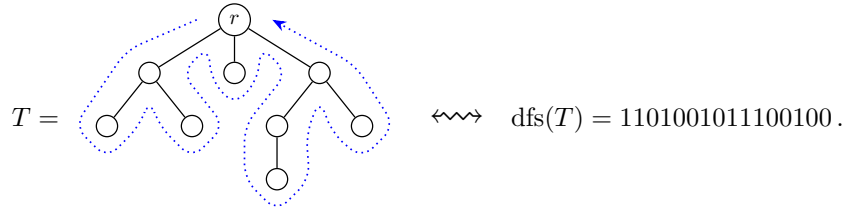
\[
T =
\scalebox{.8}{
\begin{tikzpicture}[
    >=Stealth,
    every node/.style={circle, draw, minimum size=1em, inner sep=3pt},
    level distance=2.5em,
    sibling distance=4em,
    baseline=-5em, semithick
]

\node (R) {$r$}
    child { node (A) {} 
        child { node (B) {} }
        child { node (C) {} }
    }
    child { node (H) {} }
    child { node (D) {}
        child { node (E) {} 
            child { node (F) {} }
        }
        child { node (G) {} }
    };

\draw[dotted, thick, ->, blue, rounded corners=7pt] 
    ($(R.west) + (-1em,0)$) --
    ($(A.north west) + (-0.7em,0.7em)$) --
    ($(B.west) + (-1em,0)$) -- 
    ($(B.south) + (0,-1em)$) -- 
    ($(B.east) + (1em,0)$) -- 
    ($(A.south) + (0,-.7em)$) -- 
    ($(C.west) + (-1em,0)$) -- 
    ($(C.south) + (0,-1em)$) -- 
    ($(C.east) + (1em,0)$) -- 
    ($(A.east) + (1em,0)$) -- 
    ($(R.south west) + (-.7em,-.7em)$) --
    ($(H.south west) + (-.7em,0)$) -- 
    ($(H.south) + (0,-1em)$) -- 
    ($(H.south east) + (.7em,0)$) -- 
    ($(R.south east) + (.7em,-.7em)$) -- 
    ($(D.west) + (-1em,0)$) --
    ($(E.west) + (-1em,0)$) --
    ($(F.west) + (-1em,0)$) --
    ($(F.south) + (0,-1em)$) --
    ($(F.east) + (1em,0)$) --
    ($(E.south east) + (1em,0)$) --
    ($(D.south) + (0,-.7em)$) --
    ($(G.south west) + (-1em,0)$) --
    ($(G.south) + (0,-1em)$) --
    ($(G.east) + (1em,0)$) --
    ($(D.north east) + (1.3em,0)$) --
    ($(R.east) + (1em,0)$);
\end{tikzpicture}
}
\quad{\scalebox{1.5}{$\leftrightsquigarrow$}}\quad
\mathrm{dfs}(T) = 1101001011100100\,.
\]
\begin{thm}[Tree hook-length formula]
    Let $T$ be a rooted tree poset. 
    The function $\mathrm{dfs}(T) \mapsto e(T)$ is in $\sharpL$.
\end{thm}
\begin{proof}
    Let $n$ be the number of vertices of $T$. We have $|\mathrm{dfs}(T)| = 2(n-1) \in O(n)$, and hence we need a verifier $\NL$ machine in $O(\log(n))$ space.
    
    Let $r$ be the root of $T$, let $\Delta(r) = (v_1,v_2,\ldots,v_{\delta(r)})$ be its children.
    Knuth's formula admits the recursion
    \[
    e(T) = \binom{\hl_T(r)-1}{\hl_T(v_1),\ldots,\hl_T(v_{\delta(r)})} \prod_{v_i\in\Delta(r)} e(T_{v_i}),
    \]
    where $T_v$ is the subtree rooted at $v$ (and thus $T_v$ has $\hl_T(v)$ nodes). 
    Use $\mathrm{Coeff}_T(r)$ to denote the multinomial coefficient in this formula. We then have
    \[
    e(T) = \prod_{v\in T} \mathrm{Coeff}_T(v).
    \]
    Since the multinomial coefficients are in $\sharpL$ (Theorem~\ref{thm:multinomial}), we are almost in the hypotheses of Theorem~\ref{thm:product}; it remains to show that the hook-lengths are $\FL$ computable.
    
    Let $v_1, v_2, \ldots, v_n$ be the nodes of $T$, ordered by their first visit in $\mathrm{dfs}(T)$. That is, $v_1$ is the root, and for each $1$ in $\mathrm{dfs}(T)$ we find a new node. To compute $\hl_T(v_i)$ initialise two counters $\mathtt{a}=1$ and $\mathtt{b}=1$. Scan $\mathrm{dfs}(T)$ left-to-right starting at $v_i$. For each $1$ increase $\mathtt{a}$ and $\mathtt{b}$. For each $0$ decrease $\mathtt{b}$. When $\mathtt{b}=0$ we have scanned through all of $T_{v_i}$ and return $\mathtt{a} = \hl_T(v_i)$.

    Apply Theorem~\ref{thm:product} to conclude.
\end{proof}

As corollaries of the hook-length and hook-content formulas, we can establish unimodality and (strong\footnote{For sequences of integers, strong log-concavity is equivalent to log-concavity.}) log-concavity of the binomial coefficients injectively.

\begin{cor}[Binomial unimodality]
    The function $(1^n\,0\,1^k)\mapsto\binom{n}{k}-\binom{n}{k-1}$ for $2k\le n$ is in $\sharpL$.
\end{cor}
\begin{proof}
    It is shown in \cite{Pak:inequalities} that the difference $\binom{n}{k}-\binom{n}{k-1}$ counts the number of $k$-subsets of $[n]$ whose encoding word $w$ is a \emph{ballot sequence}: for each prefix $w_1\ldots w_i$ there are at most as many $1$s as $0$s. We apply now the algorithm from Theorem~\ref{thm:ballots} with $a_i=0$ and an additional counter $\mathtt{k}$ which reads the number of $1$s. 

    Alternatively, notice that $\#\SYT((n-k,k)) = \binom{n}{k}-\binom{n}{k-1}$ and the result follows from Corollary~\ref{cor:syt_fixed_ell}. 
\end{proof}

For the next two results, we need to define the \emph{Schur polynomials} $s_\lambda(x_1, \ldots, x_N)$, which are the generating functions of semistandard Young tableaux,
\[
s_\lambda(x_1,\ldots,x_N) = \sum_{T\in\SSYT_N(\lambda)}
\prod_{u\in[\lambda]} x_{T(u)}.
\]
They are symmetric polynomials, and hence some linear combinations of elementary symmetric polynomials (similarly, of complete homogeneous symmetric polynomials). We point to \cite[Chapter 7]{StanleyEC2} for a general reference.

\begin{cor}[Binomial log-concavity]
    The functions sending $(1^n\,0\,1^a\,0\,1^b)$ where $1 \le a \le b \le n$ to the differences 
    \[
    \binom{n}{a}\binom{n}{b}-\binom{n}{a-1}\binom{n}{b+1}
    \quad\text{and}\quad
    \binom{n+a}{n}\binom{n+b}{n}-\binom{n+a-1}{n}\binom{n+b+1}{n}
    \]
    are both in $\sharpL$.
\end{cor}
\begin{proof}
    By the dual and the primal Jacobi--Trudi identities~\cite[Theorems 7.16.1 and 7.16.2]{StanleyEC2}, the differences equal
    \[
    \#\SSYT_n((b,a)')
    \quad\text{and}\quad
    \#\SSYT_n((b,a)),
    \]
    respectively.
    The results follow therefore from Corollary~\ref{cor:ssyt_fixed_ell}.
\end{proof}

\begin{cor}[Stirling log-concavity]
    The functions sending $(1^n\,0\,1^a\,0\,1^b)$ where $1 \le a \le b \le n$ to the differences 
    \begin{multline*}  
    c(n,a)c(n,b)-c(n,a-1)c(n,b+1)
    \\\text{and}\quad
    S(n+a,n)S(n+b,n)-S(n+a-1,n)S(n+b+1,n)
    \end{multline*}
    involving Stirling numbers of the first and second type are both in $\sharpL$.
\end{cor}
\begin{proof}
By \cite[Corollary~2.2]{SaganLogConcave}, we can express the Stirling numbers as the specialisation of some elementary or complete homogeneous symmetric polynomials,
    \[
    c(n,k) = e_{n-k}(1,2,\ldots,n-1)
    ~~\text{and}~~
    S(n,k) = h_{n-k}(1,2,\ldots,k).
    \]
    By the dual and the primal Jacobi--Trudi identities~\cite[Theorems 7.16.1 and 7.16.2]{StanleyEC2}, the differences in the statement equal specialisations of Schur polynomials
    \[
    s_{(n-a,n-b)'}(1,2,\ldots,n-1)
    \quad\text{and}\quad
    s_{(b,a)}(1,2,\ldots,n),
    \]
    respectively. 
    Recall that $s_\lambda(x_1, \ldots, x_n)$ is the generating function of $\SSYT_n(\lambda)$, which is in bijection with
    \[
    \Big\{(S,R, T) \,\Big|\, 
    \begin{array}{ll}
    S \subseteq [N], ~R \subseteq [n],~ |S|=|R| \geq \ell(\lambda), \\[-.6em]
          T\in\SYT(\lambda),~\text{ descents of $T$ are $\subseteq R$}
    \end{array} \Big\}
    \]
    as in the proof of Corollary~\ref{cor:ssyt_fixed_ell}.
    Hence the specialisation $s_{(a,b)}(1,2,\ldots,n)$ counts the number of tuples
    $(S,R,W,C)$ such that
    \begin{itemize}
        \item $(S,R,T)$ belong to the set above,
        \item $W$ is the (row) ballot sequence of $T$,
        \item if the descents of $W$ are at positions $d_1, d_2, \ldots, d_{m-1}$ then 
        for all $1\le j\le m$ we have
        $1\le c_i \le s_j$ for all $d_{j-1} < i \le d_j$. (Here $i_0$ is set to $0$.)
    \end{itemize}
    To encode this as a witness, we order it as 
    \newcommand{\Circled}[1]{
    \begin{tikzpicture}[baseline=-.15em]
        \node[circle, draw, inner sep=.2em] (*) at (0,0) {$#1$};
    \end{tikzpicture}
  }
    \begin{multline*}
        w_1, *, \Circled{s_1}, c_1, w_2, *, c_2, \ldots, w_{d_1}, *, c_{d_1},\\
        w_{d_1+1}, \Circled{s_2}, *, c_{d_1+1}, \ldots, w_{d_2}, *, c_{d_2}, \\
        \ldots, w_{d_{m-1}+1}, \Circled{s_m}, *, c_{d_{m-1}+1}, \ldots, w_{d_m}, *, c_{d_m}
    \end{multline*}
    where each $*$ is either a 0 or a symbol `$r$'. That is, we alternate $W$, the indicator function for $R$, and $C$, and we place the $j$th element of $S$ right after the $(j-1)$st descent $(w_{d_{j-1}}, w_{d_j})$ of $W$.
    
The elements of $S$ are written in increasing order.
The validity of $(S,R,T)$ can be checked in log-space as in the proof of Corollary~\ref{cor:ssyt_fixed_ell}.
   To check the validity of $C$, initialize a counter $\mathtt{b}=1$. After the $j$th descent, set $\mathtt{b} = s_j$. Use one work tape to check $W$ is a ballot sequence, use another work tape to check when there is a descent, another to check $1 \le c_i \le \mathtt{b}$ for each $c_i$ read, and another to check $S$ is an increasing tuple. Each work tape takes up log-space (as they appear either in the machines of Corollaries~\ref{cor:factorial}, ~\ref{cor:stirling} or~\ref{cor:ssyt_fixed_ell}).

    For the Stirling numbers of the first kind, the proof is similar, using column ballot sequences instead.
\end{proof}

\begin{note}
Establishing Stirling unimodality injectively is open (it is not even known where the peak of the sequences is).
\end{note}\smallskip

The number of permutations in $S_n$ of a given cycle type $\alpha=(1^{m_1}2^{m_2}\ldots n^{m_n})$ is given by
$$\frac{n!}{z_\alpha} = \frac{ n!}{\prod_{i=1}^n i^{m_i} m_i!}.$$
The size of the permutation is recovered from $\alpha$ via $n = \sum_i m_i i$.

\begin{thm}
    The number of permutations of cycle type $\alpha=(1^{m_1}2^{m_2}\ldots)$  is in $\sharpL$ as the function $(1^{m_1}\,0\,1^{m_2}\,0\,\dots) \mapsto n!/z_\alpha$.
\end{thm}

\begin{proof}
    Suppose that we have to choose from $N$ many elements to arrange into $m$ many cycles of length~$i$ each. Assume the elements are chosen from $[N]$.  Choose integers $1 \leq a_1<a_2< \cdots <a_{m} \leq N$, then choose a set $S_1 \subset [a_1-1]$ of $i-1$ many elements, arrange them in $(i-1)!$ many ways, with $a_1$ at the beginning these form the first cycle of length $i$. Then choose the $i-1$ many elements $S_2$ from the $a_2-i-1$ remaining elements and so on. The number of cycles with maximal elements $a_1,\ldots,a_{m}$ is then 
    \begin{multline*}
        Z(N,m,i;\,a_1, \ldots, a_{m}) := \\ \binom{a_1-1}{i-1} (i-1)! \binom{a_2-i-1}{i-1} (i-1)! \cdots \binom{a_{m} - (m-1)i -1}{i-1} (i-1)!.
    \end{multline*}
    
    To create a permutation of $n$ with cycle type $\alpha$, choose the elements in $1$-cycles from $[n]$, the elements in $2$-cycles from $[n-m_1]$, those in $3$-cycles from $[n-m_1-2m_2]$, etc. We obtain
    \[
    \frac{n!}{z_\alpha} = \prod_{i=1}^n ~\sum_{\mathbf{a}\in\binom{[N_i]}{m_i}} Z(N_i,m_i,i;\,\mathbf{a})
    \]
    where $N_i = n- \sum_{j=1}^{i-1} j m_j = N_{i-1} - (i-1) m_{i-1}$.

    If $x=(1^{m_1}\,0\,1^{m_2}\,0\,\dots)$ denotes the input, the functions sending $(1^i\,0\, x)$ to $N_i$ and $m_i$ are in $\FL$.
    We now show that
    \[
    (1^N\,0\,1^m\,0\,1^i)\mapsto \sum_{\mathbf{a}\in\binom{[N]}{m}} Z(N,m,i;\,\mathbf{a})
    \]
    is in $\sharpL$, after which the proof concludes by an application of Theorem~\ref{thm:product}.
    
    The algorithm then proceeds as follows. We keep counters $
    \mathtt{a}, 
    \mathtt{A}, 
    \mathtt{b}$. 
    The witness encodes $m$ elements $a_1,\ldots,a_{m}$ and the witnesses for 
    binomials and factorials which we have constructed in Theorem~\ref{thm:binomial} and Corollary~\ref{cor:factorial}.
    The witness is ordered
    \[\textstyle
        a_1,~ \mathtt{witness}\big(\binom{a_1-1}{i-1}\big),~ \mathtt{witness}((i-1)!);\quad
        a_2,~ \mathtt{witness}\big(\binom{a_2-i-1}{i-1}\big),~ \mathtt{witness}((i-1)!);~~ \ldots
    \]
    Set $\mathtt{a}=a_1$ and $\mathtt{A} = \mathtt{a}-1$. Check that the next two blocks of the witness are a witness for $\binom{\mathtt{A}}{i-1}$ and $(i-1)!$. Then set $\mathtt{b}=\mathtt{a}$ and read $\mathtt{a}\leftarrow a_2$. Check $\mathtt{b}<\mathtt{a}\leq N$. Set $\mathtt{A} \leftarrow \mathtt{A}+(\mathtt{a}-\mathtt{b}) - i$ and once again check that the next two blocks of the witness are a witness for $\binom{\mathtt{A}}{i-1}$ and $(i-1)!$. Repeat until having read all of the witness.
\end{proof}

The number of permutations with a given descent set $D=\{d_1,d_2,\ldots,d_k\}$ is denoted by $\beta_n(D)$ and can be defined as
$$\beta_n(D):=\#\{ w\in S_n: w_i>w_{i+1} \text{ iff }i\in D\}.$$
By the formula from \cite[\S 2.2]{EC1} we have that, setting $d_{k+1}=n$ and $d_0=0$,
$$\beta_n(D) = \det \left[ \binom{n-d_i}{d_{j+1}-d_i} \right]_{i,j=0}^k =n! \det \left[ \frac{1}{(d_{j+1}-d_i)!}\right]_{i,j=0}^{k} \,.$$

\begin{thm}
    The function $(1^n01^{d_1}01^{d_2}0\cdots) \mapsto \beta_n(d_1,d_2,\ldots)$ is in $\GapL$.
\end{thm}
\begin{proof}
    The coefficients of the determinant are computable in $\sharpL$ per Theorem~\ref{thm:binomial}. The determinant is then in $\GapL$ from Theorem~\ref{thm:detofsharpL}. 
\end{proof}

A particular case of interest here are the alternating (zig-zag) permutations, corresponding to descent sets $D=\{1,3,5,\ldots\}$. The \emph{Euler numbers} count the numbers of zigzag permutations,
$E_n=\#\{ w\in S_n: w_1>w_2<w_3>\cdots\}$. We refer to~\cite{EC1} for their properties listed below. Their exponential generating function is not rational and is given by
$$\sum_{n\geq 0} E_n \frac{x^n}{n!} = \sec(x) + \tan(x)$$
This gives a quadratic and a signed recurrence relation, but neither allows us to see $\sharpL$ containment easily. 
A refinement of these numbers are the \emph{Entringer numbers} $E(n,k)=\#\mathcal{E}(n,k)$ for $\mathcal{E}(n,k)=\{ w\in S_{n+1}: k+1=w_1>w_2<w_3>\cdots\}$ as defined in~\cite{entringer1966combinatorial} with $E(1,1)=1$, $E(1,0)=0$ and $E(n,0)=0$ for $n>1$. Note that $E_n=E(n,n)$, by mapping $w\in\mathcal{E}(n,k)$ with $w_1=n+1$ to the zigzag permutation $ (n+1-w_2)>(n+1-w_3)<\cdots (n+1-w_{n+1})$. These numbers satisfy the following recurrence
\begin{equation}\label{eq:euler}
    E(n,k) = E(n,k-1) + E(n-1,n-k).
\end{equation}
One way to see this identity is as follows. Let $w \in \mathcal{E}(n,k)$, so $w_1=k+1>w_2$. If $w_2=k$, this corresponds to permutations $(n-k+1)>(n+1 - w_3')<(n+1-w_4')>\cdots \in \mathcal{E}(n-1,n-k)$, where $w'_i = w_i$ if $w_i <k$ and $w'_i = w_i-1$ if $w_i >k+1$ for $i\geq 3$. If $w_2 <k$, then the permutation $w$ corresponds to $w(1,j) \in \mathcal{E}(n,k-1)$, where $j$ is the index for which $w_j=k$. 

\begin{thm}[Euler--Entringer numbers]\label{thm:euler} The function $(1^n\,0\,1^k) \mapsto E(n,k)$ is in $\sharpL$. In particular, computing the Euler numbers $E_n=E(n,n)$ is in $\sharpL$. 
\end{thm}
\begin{proof}
    The idea is similar to Theorem~\ref{thm:2-dim}, but it does not follow directly because of the term $E(n-1,n-k)$ of slightly different form.
     The witness is a sequence ${(n_0,k_0),(n_1,k_1),\ldots,(n_m,k_m)}$, where the numbers $n_i,k_i$ are encoded in binary. We keep 4 counters $\mathtt{a},\mathtt{b},\mathtt{c},\mathtt{d}$. Read the witness $\mathtt{a} \leftarrow n_0$, $\mathtt{b}\leftarrow k_0$ and check that they are equal to $n$ and $k$ respectively. Next read $\mathtt{c}\leftarrow n_1$, $\mathtt{d}\leftarrow k_1$ and check the local validity of the witness: if $\mathtt{a}=1$ and $\mathtt{b}=1$ we stop and accept the witness, and if $\mathtt{b}=0$ we reject the witness; otherwise if $\mathtt{a}>1$ check if $\mathtt{c} = \mathtt{a}$ and $\mathtt{d}= \mathtt{b}-1>0$ or else if $\mathtt{c}=\mathtt{a}-1$ and $\mathtt{d}=\mathtt{a}-\mathtt{b}>0$. If neither condition is satisfied, reject the witness; otherwise assign $\mathtt{a} \leftarrow \mathtt{c}, \mathtt{b} \leftarrow \mathtt{d}$, and read into $(\mathtt{c},\mathtt{d}) \leftarrow (n_2,k_2)$. Perform the checks and continue with the reassignments of counter values and reading the witness. The end of the witness should be $(0,0)$, we then stop. 
     The number of pairs that are read is $m<n^2$ and the witness is polynomially sized. 
\end{proof}

A related vast topic in enumerative combinatorics is counting pattern avoiding permutations. Let $\sigma \in S_k$ be a permutation, denote $${\rm Av}_n(\sigma):=\{ w \in S_n: \nexists \; i_1<i_2<\cdots<i_k
\ \forall 1\leq j < l \leq k
: w_{i_j}<w_{i_l} \Leftrightarrow \sigma_j<\sigma_l \},$$
the set of permutations avoiding the pattern $\sigma$. That is, permutations where there is no subsequence of elements which have the same relative ordering as $\sigma$. A classical fact is that ${\rm Av}_n(\sigma)=C_n$ for every pattern $\sigma$ of length 3. At the same time $\#{\rm Av}_n(1324)$ remains unsolved in the sense that there is no explicit generating function or formula known, and no algorithm to compute them that runs in $\poly(n)$ time. In general, we know that the size grows exponentially. We have that $\#{\rm Av}_n(\sigma) \leq C_\sigma^n$ for some constant $C_\sigma$ depending on $\sigma$; this is the famous Stanley--Wilf conjecture proven by Marcus and Tardos~\cite{marcus2004excluded}. At the same time, since permutations avoiding the initial pattern $\sigma_1\sigma_2\sigma_3$ also avoid $\sigma$, we have that $\#{\rm Av}_n(\sigma) \geq C_n \sim \frac{4^n}{\sqrt{\pi}n^{3/2}}$.  For many patterns $\sigma$ of fixed length there is a poly-time algorithm computing $\#{\rm Av}_n(\sigma)$, yet this is not known in general to be true. Containment in $\sharpL$ would imply that, and so we ask the following question.
\begin{question}\label{q:pattern avoidance}
    For which patterns $\sigma$ of fixed length $k\geq 4$ is the function $(1^n001^{\sigma_1}01^{\sigma_2}0\cdots01^{\sigma_k}) \mapsto \#{\rm Av}_n(\sigma)$ in $\sharpL$? For which cases is it in $\GapL$?
\end{question}
Note that the case $k=3$ is given by the Catalan numbers which we already showed to be in $\sharpL$, see Theorem~\ref{thm:catalan}.

A special subproblem concerns patterns $\sigma=12\ldots k$, which is the equivalent problem of counting permutations with longest increasing subsequence at most $k-1$. 

\begin{thm}[Longest Increasing Subsequences]\label{thm:lis}
    Let $k$ be a fixed integer. The function mapping $(1^n) \mapsto \#{\rm Av}_n(12\ldots k)$ is in $\sharpL$.
\end{thm}
\begin{proof}
    A crucial part of the problem is that we cannot encode a permutation directly as a witness, as we cannot check if it is indeed a sequence of distinct numbers by reading the witness only a fixed number of times. 
    In this case, the RSK algorithm, which maps permutations $w$ to SYTs $(P,Q)$ of the same shape $\lambda$, gives that the length of the longest increasing subsequence of $w$ is equal to the length of the first row, i.e. $\lambda_1$ via Greene's theorem, see e.g.~\cite{EC1}. Since $\#\SYT(\lambda) = \#\SYT(\lambda')$, we can consider instead having $\ell(\lambda) <k$.
    So
    $$\#{\rm Av}_n(1\ldots k) = \sum_{\lambda \vdash n, \ell(\lambda)<k} (\#\SYT(\lambda))^2.$$
    Then our witness consists of the partition $\lambda$ encoded as $(1^{\lambda_1}\,0\,1^{\lambda_2}\,0\dots0\,1^{\lambda_\ell})$ and two witnesses for $\# \SYT(\lambda)$ from Corollary~\ref{cor:syt_fixed_ell}. It is crucial that the partition $\lambda$ can be stored in memory with a fixed number of log-space counters (encoded in binary): $\mathtt{\lambda_1,\lambda_2,\ldots,\lambda_{\ell}}$ with $\ell<k$ and the partition can thus be read many times to verify the shapes of the two SYTs. 
\end{proof}

We finish with a problem on trees. 
Given a tree $T$ with vertex set $[n]$, let $D(T) = (d(i,j))_{1\le i,j\le n}$ be its distance matrix. The \emph{Graham--Pollak formula} gives $\det D(T) = (-1)^{n-1} (n-1)2^{n-2}$, which remarkably does not depend on the tree structure~\cite{GrahamPollak}.
\begin{thm}[Graham--Pollak formula]
    Let $T$ be a plane rooted tree on $n$ vertices. 
    The function $\mathrm{dfs}(T) \mapsto (-1)^{n-1}(n-1)2^{n-2}$ is in $\FL$,
    where the output's first bit denotes the sign of the binary-encoded output integer.
\end{thm}
\begin{proof}
    The length of the input is $2(n-1)$, so we can initialize a counter $\mathtt{a}\leftarrow n-1$. Output the least significant digit of $\mathtt{a}$; that is the sign $(-1)^{n-1}$. Copy $\mathtt{a}$ to the output tape. Then, while $\mathtt{a} > 1$, write a $0$ in the output tape and decrease $\mathtt{a}$ by~$1$.
\end{proof}

    \begin{note}
        There are several generalisations of the Graham--Pollak formula obtained by placing weights on the tree in various ways.
        It is shown in \cite{BEGLR} that all (additive) generalisations of the determinant in the literature count weighted \emph{unital arrowflows} (a choice of edge $e\in E(T)$ and an orientation of $T\setminus\{e\}$). To place such a formula in $\sharpL$ requires us to witness an orientation of a tree in log-space.
        One possible way of encoding an orientation is by alternating the bits of $\mathrm{dfs}(T)$ with some extra symbols (for instance $\uparrow$ for arcs oriented towards the root, $\downarrow$ otherwise). Each edge appears twice in $\mathrm{dfs}(T)$, and hence one has to check that the two symbols associated to this edge are equal; this in general requires $O(n)$ many counters. We do not know if there is any better way of encoding an orientation so that it can be witnessed by a verifier $\NL$ machine.
    \end{note}

\section{Representation theory and geometry}\label{sec:structure constants}

A fundamental object arising in this section and beyond is the set of integer points in a polytope. 

\begin{thm}[Integer points in polytopes]\label{thm:polytopes}
Consider the polytope $P$ of points $\mathbf{x} \in \mathbb{R}^n$ given by inequalities $A \mathbf{x} \leq \mathbf{b}$ where $A$ is an  $m \times n$ integer matrix and $\mathbf{b} \in \mathbb{Z}^m$. For fixed $n$  the function $(A,\mathbf{b}) \mapsto \#\{ \mathbf{x} \in \mathbb{Z}^n \cap P\} $ is in $\sharpL$, where $A_{i,j}$ and $b_j$ are encoded in unary, with one extra bit to indicate their sign.
\end{thm}
\begin{proof}
    Here, the input size is $N = \sum_{i,j} |A_{i,j}|+ \sum_i |b_i| + 2nm+2m$. The witness is $\mathbf{x}=(x_1,\ldots,x_n)$ with each integer encoded in binary. Note that a rough upper bound on the coordinates is $|x_i|<(Nn)^n$. We will read the witness into $n$ counters $\mathtt{x_i} \leftarrow x_i$. Keep counters $\mathtt{s}, \mathtt{a}, \mathtt{i}, \mathtt{j}, \mathtt{b}$ in $\log(N)$-space encoded in binary. Set $\mathtt{i}=1$, $\mathtt{s}=0$ and $\mathtt{j}=1$. Read $\mathtt{a} \leftarrow A_{i,j}$, set $\mathtt{s} \leftarrow \mathtt{s} + \mathtt{a}*\mathtt{x_j}$, see Lemma~\ref{lem:elementary}, and keep increasing $\mathtt{j}$ until it reaches $n$. Read $\mathtt{b}\leftarrow b_i$ and check if $\mathtt{s} \leq \mathtt{b}$. Increase $\mathtt{i}$ by 1, reset the counters and repeat. 
\end{proof}

\begin{note}
If $n$ is fixed, then Theorem~\ref{thm:polytopes} implies that counting the integer points in a polytope is in $\FP$, which is generally true for fixed dimensional polytopes via Barvinok's algorithm~\cite{barvinok1994polynomial}, but allows inputs encoded in binary.
If $n$ (and hence $m$) are not fixed then the number of integer points is $\sharpP$-complete even when the entries of $A$ and $\mathbf{b}$ are bounded, as $\#\textsc{SAT}$ can be reduced to that problem.
\end{note}

\begin{cor}[Ehrhart polynomial]
    Let $P=\{A\mathbf{x}\le\mathbf{b}\}$ be a polytope as in Theorem~\ref{thm:polytopes}. For fixed $n$ and $m$ the Ehrhart function $(A, \mathbf{b},1^k)\mapsto\mathrm{ehr}_P(k) = \#\{\mathbf{x}\in\ZZ^n\cap kP\}$ is in $\sharpL$.
\end{cor}
\begin{proof}
We have $kP = \#\{A\mathbf{x}\le k\mathbf{b}\}$.
    Compose the function $(1^{b_i},1^k) \mapsto (1^{kb_i})$ (which is in $\FL$) with the verifier $\NL$ machine of Theorem~\ref{thm:polytopes}.
\end{proof}

An important class of polytopes are the transportation polytopes. Their integer points correspond to contingency tables, which play a role in algebraic combinatorics, as shown below.

Given two compositions $\alpha=(\alpha_1,\ldots,\alpha_n), \beta=(\beta_1,\ldots,\beta_m)$ of the same integer and $n, m$ many parts, respectively, a \emph{contingency table} with marginals $\alpha, \beta$ is an $n \times m$ matrix $A$ with $A_{i,j} \in \mathbb{N}_0$, such that 
$$\sum_i A_{ij} = \beta_j \text{ for $j=1,\ldots,m$}, \qquad \sum_j A_{ij} =\alpha_i \text{ for $i=1,\ldots,n$}.$$
We denote the set of such matrices by $CT(\alpha,\beta)$.
A \emph{0-1 contingency table}, sometimes referred to as ``binary'', is such a matrix $A$ with the additional requirement that $A_{i,j}\in\{0,1\}$. We denote the set of such 0-1 matrices by $CT_0(\alpha,\beta)$. To view these as integer points in a polytope, we can convert the linear equalities $a=b$ into inequalities $a\leq b$ and $-a \leq -b$.

These quantities appear in the theory of symmetric functions as the transfer coefficients between the complete homogenous/elementary symmetric functions and the monomial symmetric functions \cite[Theorems~7.4.1 and~7.5.1]{StanleyEC2}. Namely,
$$h_\lambda = \sum_{\mu} \#CT(\lambda,\mu) m_\mu, \qquad e_\lambda = \sum_\mu \#CT_0(\lambda,\mu) m_\mu.$$

The following comes directly from Theorem~\ref{thm:polytopes} as the polytopes are of dimensions $mn$, the variables corresponding to $A_{ij}$'s, and the number of inequalities is $2(m+n) + 2mn$. 
\begin{cor}
    Let $m,n$ be constant integers. Then the functions sending 
    $$(1^{\alpha_1}\,0\,1^{\alpha_2}\,0\,\cdots \,0\,1^{\alpha_n}\,0\,0\,1^{\beta_1}\,0\,\cdots\,0\,1^{\beta_m})$$ 
    to $\# CT(\alpha,\beta)$ and $\# CT_0(\alpha,\beta)$ are in $\sharpL$.
\end{cor}

However, here we can say more, these functions are still in $\sharpL$ when one of the dimensions is a variable. The result implies the well-known belonging to $\FP$, which is usually shown with dynamic programming.

\begin{thm}[Contingency tables]
\label{thm:contingencytables}
    Let $m$ be a constant. The functions sending
    $$(1^n\,0\,0\,1^{\alpha_1}\,0\,1^{\alpha_2}\,0\,\cdots \,0\,1^{\alpha_n}\,0\,0\,1^{\beta_1}\,0\,\cdots\,0\,1^{\beta_m})$$ 
    to $\# CT(\alpha,\beta)$ and $\# CT_0(\alpha,\beta)$ are in $\sharpL$.
\end{thm}

\begin{proof}
     The input size is $N = 2|\alpha|+n+m$ since for a valid input we should have $|\beta|=|\alpha|$. We consider the general case, the 0-1 contingency tables are treated in the same way. The witness is 
     $$w = A_{1,1},A_{1,2},A_{1,3},\cdots,A_{1,m};A_{2,1},A_{2,2},\ldots,A_{2,m};\ldots;A_{n,1},\ldots,A_{n,m},$$
     with $A_{i,j}$ encoded in unary  and we use the symbols $,$ and $;$ as delimiters. 
     In the algorithm we keep $m$ many counters $\mathtt{b_i}$ initialized at $0$ for $i=1,\ldots,m$, counters $\mathtt{a}$ and $\mathtt{A}$, indices $\mathtt{i},\mathtt{j}$ and $\mathtt{m},\mathtt{n}$. We read the first $n$ bits from $w$ into $\mathtt{n}$ recorded in $\log(n)$ space in binary. Next read $\mathtt{a} \leftarrow \alpha_1$ in binary. Set $\mathtt{i}\leftarrow 1, \mathtt{j} \leftarrow 1$. Read $\mathtt{A} \leftarrow A_{1,1}$ (i.e. the first part of $w$ before the first $,$) in binary, set $\mathtt{a} \leftarrow \mathtt{a}-\mathtt{A}$. Set $\mathtt{b_j} \leftarrow \mathtt{b_j}+\mathtt{A}$. Next, increase $\mathtt{j} $ by 1, and repeat the last steps: $\mathtt{A} \leftarrow A_{1,2}$, $\mathtt{a} \leftarrow \mathtt{a} - \mathtt{A}$, $\mathtt{b_j} \leftarrow \mathtt{b_j}+\mathtt{A}$. When reaching the first $;$ we check if $\mathtt{a}=0$, otherwise the witness is invalid. Then increase $\mathtt{i}$, reset $\mathtt{j}\leftarrow 1$, read the input $\mathtt{a} \leftarrow \alpha_i$ and repeat the above steps. 
     In the end we have the counters $\mathtt{b_1},\ldots,\mathtt{b_m}$ and we check if $\mathtt{b_j}=\beta_j$ from reading the input. 
\end{proof}
It is a classical result that counting contingency tables is $\sharpP$-complete when the marginals are encoded in binary~\cite{dyer1997sampling} and at least one dimension is variable. 
Pak and Panova had conjectured that counting contingency tables of non-fixed sizes and unary input (so input size is $O(m+n+\sum |A_{i,j}|)$) is still $\sharpP$-complete, see e.g.~\cite{panova2025computational}. One way of disproving such conjecture, conditionally, would be to consider $\sharpTISP(-,\log^k(n))$ containment.  
\begin{question}
    \label{q:CT}
Let $m$ and $n$ be arbitrary. Is the function sending
    $$(1^n\,0\,0\,1^{\alpha_1}\,0\,1^{\alpha_2}\,0\,\cdots \,0\,1^{\alpha_n}\,0\,0\,1^{\beta_1}\,0\,\cdots\,0\,1^{\beta_m})$$ 
    to $\# CT(\alpha,\beta)$ and $\# CT_0(\alpha,\beta)$ in $\sharpTISP(\infty,\log^k(n))\subseteq \FQP$ for some $k$?
\end{question}

\smallskip

The different bases of the algebra $\Lambda[x_1,\ldots,x_n]$ of symmetric polynomials in $n$ variables are labelled by partitions of length or width at most $n$. For a given basis $\{b_\lambda\}$, the structure constants are the coefficients in the expansion
\[
b_\mu(x_1,\ldots,x_n) \circledast b_\nu(x_1,\ldots,x_n) = \sum_\lambda d_{\mu,\nu}^\lambda b_\lambda(x_1,\ldots,x_n)
\]
of a certain binary operation $\circledast$ on the basis. For instance, if $\circledast$ is the usual product and the basis is the Schur basis, then the structure constants are the $\GL_n$-Littlewood--Richardson coefficients.

In this section we study problems that take as input a tuple of partitions encoded in unary via
\[
\lambda \leftrightsquigarrow (1^{\lambda_1}\,0\,1^{\lambda_2}\,0\,\dots\,0\,1^{\lambda_{\ell(\lambda)}}).
\]
Note that one cannot store a partition of $k$ in $\log(k)$ space\footnote{The Hardy--Ramanujan asymptotic for the partition function shows that there are exponentially many partitions of $k$. Hence no matter the encoding, one needs polynomial amounts of space to distinguish between all partitions.}. However, for any fixed $n$, one can encode a partition of $k$ of length at most $n$ with $n$ many $\log(k)$ counters.

\begin{thm}[$\GL_n$-Littlewood--Richardson coefficients]
Fix $n$. Let $\lambda, \mu, \nu$ be partitions of length at most $n$. Then, the function 
$$(1^{\lambda_1}\,0\,1^{\lambda_2}\,0\,\dots\,0\,1^{\lambda_n}\,0
\,1^{\mu_1}\,0\dots\,0\,1^{\mu_n}
\dots\,1^{\nu_1}\,0\dots\,0\,1^{\nu_n}) \mapsto c_{\mu,\nu}^\lambda$$
     is in~$\sharpL$.
\end{thm}
\begin{proof}
    Littlewood--Richardson coefficients count integer points in the hive polytope which is a triangular array of variables with local inequalities and boundary conditions. The dimension of the polytope is $\binom{n}{2}$, the inequalities are of the type $a+c \geq b+d$ for all rhombi in a triangular grid of the variables. The boundary conditions give that the variables on the sides are $\lambda$, $\mu$, $\nu$, respectively. Then the resulting polytope has a fixed number of variables, inequalities, the linear inequalities have constant coefficients $\{-1,0,1\}$ and the vector $b$ has entries $0$ or $\in \{\lambda_i,\mu_i,\nu_i, i=1,\ldots,n\}$. Theorem~\ref{thm:polytopes} then applies and the result follows.
\end{proof}

The Kostka numbers are a special case of the LR numbers, see e.g.~\cite{Nar06}, so we have the following corollary.

\begin{cor}[$\GL_n$-Kotska numbers]
    For fixed $n$, the function
    \[
    (1^{\lambda_1}\,0\,1^{\lambda_2}\,0\,\dots\,0\,1^{\lambda_n}\,0
    \,1^{\alpha_1}\,0\dots\,0\,1^{\alpha_n}) \mapsto K_{\lambda,\alpha}
    \] 
    is in~$\sharpL$.
\end{cor}

\begin{note}
    It is worth noting the following relation, corollary of the RSK correspondence:
    $$\sum_\lambda K_{\lambda \alpha}K_{\lambda \beta} = \#CT(\alpha,\beta).$$
    If it is possible to construct a verifier $\NL$ machine which checks if two tableaux are SSYTs of the same shape and types $\alpha$, $\beta$ respectively for variable lengths of $\alpha$, $\beta$ we would have $\#CT(\alpha,\beta) \in \sharpL$, which would generalize Theorem~\ref{thm:contingencytables}.
\end{note}
\smallskip

Recall that a composition is a finite sequence of positive integers, and a weak composition is a finite sequence of non-negative integers. Let $\mathrm{Comp}$ and $\mathrm{WComp}$ be the sets of compositions and weak compositions.
\begin{thm}
For fixed $n$, the structure constants for the product of the monomial basis of $\Lambda[x_1,\ldots,x_n]$ are in $\sharpL$.
\end{thm}
\begin{proof}
    We have
    \[
    m_\mu\cdot m_\nu(x_1,\ldots,x_n) = \sum k_{\mu,\nu}^\lambda m_\lambda(x_1,\ldots,x_n),
    \]
    where
    \[
    k_{\mu,\nu}^\lambda = \#\{
    (\alpha,\beta)\in\mathrm{WComp}^2 \mid \mathrm{sort}(\alpha) = \mu, ~
    \mathrm{sort}(\beta) = \nu, ~
    \alpha + \beta = \lambda
    \}.
    \]

    We construct a machine taking input $\lambda,\mu,\nu$ in unary and witness $w$, which checks if $w$ is in the above set.
    Initialize $3n$ counters
    \[
    \mathtt{c_1^\lambda}, \ldots,
    \mathtt{c_n^\lambda},
    \mathtt{c_1^\mu}, \ldots,
    \mathtt{c_n^\mu},
    \mathtt{c_1^\nu}, \ldots,
    \mathtt{c_n^\nu}
    \]
    to the values of the parts of $\lambda,\mu$ and $\nu$.
    We then read the witness and use another $2n$ (fixed number) counters to store $\alpha$ and $\beta$. With another index counter $\mathtt{i}$ we check that $\mathtt{\alpha_i}+\mathtt{\beta_i}=\lambda_i.$
    We then sort $\mathtt{\alpha}$ and $\mathtt{\beta}$ using bubblesort and check if they are equal to $\mu,\nu$, respectively.  
\end{proof}

Another rich source of combinatorial rules and formulas related to symmetric polynomials is the theory of crystals. For a Lie algebra $\g$, the $\g$-\emph{crystal} of a representation $V$ is a certain directed graph. 
Its edges are best understood by Kashiwara's tensor product rule.
See \cite{BumpSchilling} for a reference on crystals as combinatorial objects.

The tensor product rule for $\sl_2$-crystal is much older and is often attributed to Clebsch and Gordan (for the algebra) or Greene and Kleitman (for the combinatorics). The $\sl_2$-crystal for $(\CC^2)^{\otimes m}$ is denoted $\mathbb{B}_2^{\otimes m}$ and is naturally identified with a symmetric chain decomposition (SCD) of the Boolean lattice. Elements of the lattice are strings of length $m$ with two symbols, `\texttt{(}' and `\texttt{)}', with \texttt{)))}\dots\texttt{)} being the initial object and \texttt{(((}\dots\texttt{(} being the terminal object. 
Pair up parentheses following the usual grammatical rules of English.
In the SCD, there is an directed edge from $x$ to $y$ if $x$ is created from $y$ by flipping the right-most unpaired `\texttt{)}' to a `\texttt{(}'. We say that $E\,.\,x = y$ and call $E$ the \emph{raising operator}. 
If there is no unpaired `\texttt{)}' in $x$, then the previous algorithm is undefined; we write $E\,.\,x=0$ and say $x$ is a \emph{highest-weight element}. An analogous algorithm (flipping the left-most unpaired `\texttt{(}' to a `\texttt{)}' instead) defines the \textit{lowering operator} $F\,.\,y = x$.

Now consider the $\sl_n$-crystal $\mathbb{B}_n^{\otimes m}$.
Vertices of the crystal are strings of length $m$ of $n$ different symbols $1, 2, \ldots, n$.
There are $n-1$ raising operators $E_1, \ldots, E_{n-1}$. 
To compute $E_i\,.\,x$, begin by changing every instance in $x$ of the symbol $i$ to a `\texttt{(}' and every instance of $i+1$ to a `\texttt{)}'. Then compute the $\sl_2$-raising operator on this substring, as before. Revert the change of symbols to obtain $y = E_i\,.\,x$. Again, if this process is not well defined, then we write $E_i\,.\,x=0$ and say $x$ is a highest-weight element.

\begin{thm}[Raising operator]\label{thm:crystal operator}
    Let $x\in[n]^m$ be a vertex of $\mathbb{B}_n^{\otimes m}$.
    The function
    \[
    (1^n\,0\,1^m\,0\,1^i\,0\,x_1\,x_2\,\ldots\,x_m) \mapsto E_i\,.\,x
    \]
    is in $\FL$, where each $x_j$ is encoded in binary taking up $\lceil\log_2(n)\rceil$-many bits.
\end{thm}
\begin{proof}
    Initialise two binary-encoded counters $\mathtt{i}=i$ and $\mathtt{I}=i+1$; these will require at most $\log(n)$-space each. Initialise two more counters: $\mathtt{a}=0$ to keep track of which parentheses pair up, and $\mathtt{j}=1$ that will be used to keep track of the position of the last open bracket ever read; these will require at most $\log(m)$-space each.

    We read $x$ from left to right. If the current entry $x_k$ is not equal to $\mathtt{i}$ nor $\mathtt{I}$, do nothing. If $x_k = \mathtt{i}$, then increase $\mathtt{a}$ by~$1$. If $x_k = \mathtt{I}$ and $\mathtt{a}>0$, then decrease $\mathtt{a}$ by~$1$. If $x_k = \mathtt{I}$ and $\mathtt{a}=0$, then set $\mathtt{j} \leftarrow k$.
    
    When we have read the whole of $x$, the counter $\mathtt{j}$ points to the right-most unpaired closing bracket. Read through $x$ once again, copying everything to the output tape. Except, when at $x_\mathtt{j}$, copy $\mathtt{i}$ to the output tape instead.
\end{proof}

\vspace{-.5em}
\begin{note}
The rule of matching parentheses and swapping one of them is called the signature rule.
The same signature rule generalises to all finite type crystals.

The standard crystal $\mathbb{B}_n$ of $\sl_n$ is
    \[
    1 \xrightarrow{1} 2 \xrightarrow{2} 3 \xrightarrow{3} \cdots \xrightarrow{n-1} n 
    \]
    where the arrow $\xrightarrow{i}$ corresponds to the lowering operator $F_i$.
    This standard crystal governs the signature rule: it is the reason $i$ can only ever raise to $i+1$.
    For other finite type crystals, the $\FL$ machine would need to be governed by a different standard crystal instead. We avoid doing this for simplicity. 
    
    Moreover, Theorem~\ref{thm:crystal operator} also works in the same way for the crystal $\mathcal{B}_\lambda$ of semistandard Young tableaux of shape $\lambda\vdash m$ and entries from $\{1,\ldots,n\}$.
\end{note}

\begin{cor}[Highest weight]\label{cor:highest weight}
    Let $x\in[n]^m$ be a vertex of $\mathbb{B}_n^{\otimes m}$.
    The function $(1^n\,0\,1^m\,x_1\,\ldots\, x_m) \mapsto$ `is $x$ a highest weight element?' is in $\L$.
\end{cor}
\begin{proof}
    Loop through $\mathtt{i}=1,\ldots,n-1$. Then compose with the $\FL$ function from Theorem~\ref{thm:crystal operator} to compute $E_{\mathtt{i}}\,.\,x$. Accept if $E_{\mathtt{i}}\,.\,x=0$ for all $\mathtt{i}$.
\end{proof}

The \emph{weight} of a highest weight element $x$ is the partition $\lambda$ where $\lambda_i$ counts the number of instances of $i$ in $x$.

\begin{cor}[Weight of an element]\label{cor:hw of lambda}
    Let $x\in[n]^m$ be a vertex of $\mathbb{B}_n^{\otimes m}$.
    The function sending
    $$(1^n\,0\,1^m\,0\,x_1\,\ldots\, x_m\,0\,1^{\lambda_1}\,0\,1^{\lambda_2}\,0\,\ldots)$$
to 
    `is $x$ a highest weight element of weight $\lambda$?' is in $\L$.
\end{cor}
\begin{proof}
    Begin by asking `is $x$ a highest weight element?', which we can do in $\L$ by Corollary~\ref{cor:highest weight}.

    Then, for $\mathtt{i}=1,2,\ldots,n$, initialise a counter $\mathtt{b}=\lambda_\mathtt{i}$. Read through $x$ again; if $x_k=\mathtt{i}$ then $\mathtt{b}\leftarrow\mathtt{b}-1$. When $x$ is fully read, check $\mathtt{b}=0$; otherwise reject.
\end{proof}
\begin{note}\label{note:crystal operators}
Let $\lambda$ be a partition of $n$.
    It is a consequence of Schur--Weyl duality \cite[Chapter 8]{BumpSchilling} that the number of highest weight elements of weight $\lambda$ in $\mathbb{B}_{n}^{\otimes n}$ is $\#\SYT(\lambda)$. However, this does \emph{not} place $\#\SYT(\lambda)$ in $\sharpL$, since the machine of Corollary~\ref{cor:hw of lambda} reads the element $O(n)$ times, even if the length of $\lambda$ is bounded.
    
    See Corollary~\ref{cor:syt_fixed_ell} and Theorem~\ref{thm:hook-length} for successful approaches for this problem.
\end{note}
\smallskip

Kashiwara's tensor product rule says that if $\mathcal{B}$ and $\mathcal{C}$ are $\sl_n$-crystals then
\[
E_i\,.\,(x\otimes y) = \begin{cases}
    (E_i\,.\,x)\otimes y & \text{if}~\varphi_i(y) < \varepsilon_i(x),\\
    x\otimes(E_i\,.\,y) & \text{if}~\varphi_i(y) \ge \varepsilon_i(x)\\
\end{cases}
\]
defines an $\sl_n$-crystal $\mathcal{B}\otimes\mathcal{C}$ on the Cartesian product of the vertex sets, where $\varphi_i(y) = \max\{r\mid F_i^r\,.\,y \ne 0\}$ and $\varepsilon_i(x) = \max\{r\mid E_i^r\,.\,x \ne 0\}$.
\begin{thm}[Kashiwara's tensor product rule]
    If the raising operators of two $\sl_n$-crystals $\mathcal{B}$ and $\mathcal{C}$ are in $\FL$ then the raising operator of $\mathcal{B}\otimes\mathcal{C}$ is in $\FL$.
    
    Here, raising operators are encoded as in Theorem~\ref{thm:crystal operator}, since there exists $m$ such that $\mathcal{B}\subseteq\mathbb{B}_n^{\otimes m}$, and similarly for the other crystals.
\end{thm}
\begin{proof}
    We adapt the machine from Theorem~\ref{thm:crystal operator} to compute $\varepsilon_i(x)$. To do this, we add a counter $\mathtt{e}=0$. Whenever $\mathtt{a}=0$ in the machine of Theorem~\ref{thm:crystal operator}, we have found an unpaired `\texttt{)}' and thus we increase $\mathtt{e}$ by~$1$. When we finish reading $x$ we have $\mathtt{e}=\varepsilon_i(x)$. Output this value instead of $E_i\,.\,x$.

    A similar machine can be constructed for computing $\varphi_i(y)$, by reading $y$ right-to-left. Store it in a counter $\mathtt{f}$.

    Both of these values are upper-bounded by the length of $(x,y)$.

    Finally output $E_i\,.\,x$ followed by $y$ if $\mathtt{f} < \mathtt{e}$, or $x$ followed by $E_i\,.\,y$ otherwise. 
\end{proof}

\section{Number theory}\label{sec:number theory}

In the following, we discuss several number theoretic functions $f(n)$. To interpret them as counting problems, we encode $n$ in unary.

\begin{thm}[Partition function]\label{thm:partition}
    The function $(1^n) \mapsto \#\mathrm{Par}(n)$ is in $\sharpL$.
\end{thm}
\begin{proof}
    The witness is a partition $\lambda_1, \lambda_2, \ldots$ written in binary. A counter is used to check that $\sum_i \lambda_i = n$. Another two counters check that $\lambda_i \le \lambda_{i-1}$ for all $i$.
\end{proof}
We remark that the above theorem implies the following result. We note however, that establishing the corollary by itself is non-trivial; a naive algorithm does not work, and one has to use e.g.~Euler's pentagonal theorem.
\begin{cor}
    The function $(1^n) \mapsto \#\mathrm{Par}(n)$ is in $\FP$.
\end{cor}
\begin{note}
    The same approach would not work to show that $\#\mathrm{Par}(n)$ is in $\FL$, since it relies on a recursion which computes exponentially large values that cannot be stored in log-space. It is conceivable though that the machine could write the sequence of $0,1$s encoding the number of partitions in binary on the output tape without storing all of them on the work tape, and hence we cannot exclude the possibility of $\#\mathrm{Par}(n)$ in $\FL$. 
\end{note}

\begin{cor}[Partitions in a box]\label{cor:partitions in a box}
    The function 
    \[
    (1^n\,0\,1^\ell\,0\,1^k)\mapsto p_n(\ell\times k) := \#\{\lambda\in\mathrm{Par}(n) \mid \ell(\lambda)\le\ell, ~\lambda_1\le k\}
    \]
    is in $\sharpL$.
\end{cor}
\begin{proof}
    Take the verifier $\NL$ machine from Theorem~\ref{thm:partition}. Add one counter to check the length of $\lambda$ is less than $\ell$. Add another counter to check that $\lambda_1\le k$.
\end{proof}

Define the $q$-integer $(n)_q = (1-q^n)/(1-q)$, the $q$-factorial $(n)_q! = (n)_q\cdot(n-1)_q\cdots(1)_q$, and $q$-binomial $\binom{n+m}{n}_{\!q} = (m+n)_q! / ((n)_q!(m)_q!)$. More generally, the $q$-multinomial
\[
\binom{n}{a_1,\ldots,a_k}_{\!\!q} = \binom{n}{a_1}_{\!\!q} \binom{n-a_1}{a_2}_{\!\!q} \cdots\,.
\]

\begin{cor}[$q$-binomials and $q$-multinomials]
    The function 
    \[(1^n\,0\,1^r\,0\,1^{a_1}\,0\,1^{a_2}\,0\,\cdots\,0\,1^{a_k}\,) \mapsto \mathrm{coeff}_{q^r}\binom{n}{a_1,a_2,\ldots,a_k}_{\!\!q}
    \]
    giving the coefficient at $q^r$ in the $q$-multinomial coefficient is in $\sharpL$.
\end{cor}
\begin{proof}
    We first observe that 
    $$\binom{n}{a}_{\!\!q} = \sum_r p_r( (n-a) \times a)\cdot q^r$$
    and therefore
    \[
        \binom{n}{a_1,\ldots,a_k}_{\!\!q} = \sum_{r_1,r_2,\ldots} p_{r_1}((n-a_1)\times a_1) p_{r_2}( (n-a_1-a_2) \times a_2) \cdots q^{r_1+r_2+\cdots} .
    \]
    We proceed like in the proof of Theorem~\ref{thm:product}. 
    The witness is a sequence of the witnesses $w^{(1)},w^{(2)},\ldots$ from Corollary~\ref{cor:partitions in a box}, together with  integers $r_i$, so $w=r_1, w^1; r_2, w^2; \cdots$. We initialise counters $\mathtt{n}=n$, $\mathtt{i}=1,\mathtt{R}=0,\mathtt{d}=0$. After reading $\mathtt{d}\leftarrow r_\mathtt{i}$, we set $\mathtt{R} \leftarrow \mathtt{R} + \mathtt{d}$. Read and check the validity of $w^{\mathtt{i}}$ as a witness for $p_{\mathtt{r_i}}((\mathtt{n}-a_{\mathtt{i}})\times a_\mathtt{i})$. Then increase $\mathtt{i}$ by~$1$, set $\mathtt{n}\leftarrow\mathtt{n}- a_\mathtt{i}$, and repeat. In the end check $\mathtt{R}=r$.
\end{proof}

\begin{thm}[Euclidean division]\label{thm:euclidean div}
    The function $$(1^n\,0\,1^k)\mapsto(\lfloor{n/k}\rfloor,\, n-k\lfloor{n/k}\rfloor)$$ is in $\FL$.
\end{thm}
\begin{proof}
    Initialise two binary-encoded counters $\mathtt{n}=n$ and $\mathtt{k}=k$. Initialise one more binary-encoded counters $\mathtt{q}$.
    While $\mathtt{q}*\mathtt{k} \le \mathtt{n}$, increase $\mathtt{q}$ by 1. Once $\mathtt{q}*\mathtt{k} > \mathtt{n}$, decrease $\mathtt{q}$ by 1. Return $\mathtt{q}$ and $\mathtt{n}-\mathtt{q}*\mathtt{k}$. Multiplication, comparison, and difference require an extra auxiliary counter, so in total 4 counters are needed.  
\end{proof}
\begin{thm}[Primes]\label{primes}
    The decision problem $(1^n) \mapsto$ `is $n$ prime?' is in $\L$.
\end{thm}
\begin{proof}
    Create a counter $\mathtt{n}=n$.
    Loop through $\mathtt{k}=2,\ldots,n-1$. Perform Euclidean division $n/\mathtt{k}$ to find a remainder. If at any point the remainder is $0$, reject. Otherwise, accept.
\end{proof}

We continue discussing some classical multiplicative functions.
\begin{thm}[Divisor function]
    The function $\sigma_k(n)$ given by $(1^n\,0\,1^k) \mapsto \sum_{d|n} d^k$ is in $\sharpL$.
\end{thm}
\begin{proof}
    The witness is a tuple $d, q, a_1, a_2, \ldots, a_{k}$ of $k+2$ integers in binary, such that $1 \le d \le n = dq$, and such that $1\le a_i \le d$ for all $i$.
\end{proof}
We can also consider deterministic computations of the divisor function. 
The asymptotic growth of $\sigma_k(n)$ as $n\to\infty$ is upper bounded by $n^k \zeta(k)$ where $\zeta$ is the Riemann zeta function. We have $\zeta(k)=1 + O(2^{-k})$. But if $k$ is part of the input $x$ of the divisor function, then the intermediate sums involved in a naive computation of $\sigma_k(n) \in O(n^k\zeta(k)) = O(n^k) = O(2^{|x|\log(|x|)})$ are too big for a $\FL$ machine to store in the work tape. We thus fix $k$ in our next result. (Note that $(1^n\,0\,1^k) \mapsto \sigma_k(n)$ might still be in $\FL$, as its growth is inside the range of Theorem~\ref{thm:FL upper bound}.) 
\begin{thm}[Divisor function]
    Fix a non-negative integer $k$. The function $\sigma_k(n)$ given by $(1^n) \mapsto \sum_{d|n} d^k$ is in $\FL$.
\end{thm}
\begin{proof}
    For $\mathtt{d}=1,\ldots,n$ call the $\FL$ machine from Theorem~\ref{thm:euclidean div} to compute $n/\mathtt{d}$. If the remainder is $0$, add $\mathtt{d}^k$ to a cumulative counter $\mathtt{c}$. At the end, output $\mathtt{c}$.    
\end{proof}

\begin{thm}[Euler totient function]
    The function $\varphi(n)$ given by $$(1^n) \mapsto \#\{k \in [n] \mid k~\text{coprime with}~n\}$$ is in $\sharpL$.
\end{thm}
\begin{proof}
    The witness is $k$. 
    Compute the greatest common divisor of $n$ and $k$ by successive applications of the Euclidean division algorithm from Theorem~\ref{thm:euclidean div}. Accept if $\gcd(k,n) = 1$ and $k\le n$.
\end{proof}

The next result, although easy, will serve to motivate a new proof technique.

\begin{thm}[Radical]\label{thm:radical}
    The function $(1^n) \mapsto \prod_{p|n~\text{prime}} p$ is in $\sharpL$.
\end{thm}
\begin{proof}
    The statement follows by an application of Theorem~\ref{thm:product}. 
    The function $f : (1^p)\mapsto p$ is clearly in $\sharpL$. Set $f(\emptyset)=1$. The function 
    \begin{equation}
        \label{eq:is prime divisor}
    g : (1^n\,0\,1^p) \mapsto \begin{cases}
        1^p & \text{if $p|n$ and $p$ is prime},\\
        \emptyset & \text{otherwise}
    \end{cases}
    \end{equation}
    is in $\FL$ by Theorems~\ref{thm:euclidean div} and~\ref{primes}. The radical of $n$ is the function $(1^n)\mapsto\prod_{p=1}^n f(g(1^n\,0\,1^p))$, which is therefore in $\sharpL$.
\end{proof}

Let $p$ be a prime. The \emph{$p$-adic valuation} $v_p(n)$ of an integer $n$ is the largest power of $p$ that divides $n$. 
We have $v_p(nm) = v_p(n) + v_p(m)$ and $v_p(n/m) = v_p(n) - v_p(m)$.
A proof technique suggested to us by ChatGPT is to use $p$-adic valuations to turn rational formulas into product formulas via
\[
n = \prod_{p|n~\text{prime}} p^{v_p(n)},
\]
which in conjunction with Theorem~\ref{thm:product} (for products of $\sharpL$ functions), gives a new way of checking containment in $\sharpL$.
For instance, we can give a third proof of Corollary~\ref{cor:factorial} which places the factorial in $\sharpL$.
\begin{proof}[Third proof of Corollary~\ref{cor:factorial}]
    Fix a prime $p$.
    Legendre's formula gives $v_p(n!) = \lfloor n/p\rfloor + \lfloor n/p^2\rfloor + \lfloor n/p^3\rfloor + \cdots$.
    Note that the sum is finite. 
    We can therefore compute $v_p(n!)$ in $O(\log(n))$-space by adding the results of Euclidean divisions of $n$ by the successive powers $p^k$ of $p$, while $p^k \le n$. The function $(1^n)\mapsto p^{v_p(n!)}$ is in $\sharpL$, a witness being a tuple of $v_p(n!)$-many numbers $1\le a_i \le p$. 
    (Here we are applying Theorem~\ref{thm:product}, where the number of factors is $\log_p(n!) = O(n \log (n))$, and hence can be stored in $O(\log_2(\log_p(n!))) = O(\log_2(n))$-space.)
    
    All prime divisors of $n!$ are smaller or equal than $n$. Hence the formula
    \[
    n! = \prod_{p\le n~\text{prime}} p^{v_p(n!)}
    \]
    is in the hypotheses of Theorem~\ref{thm:product}.
\end{proof}

The \emph{Maya diagram} of a partition $\lambda$ is the $\{0,1\}$-string $\mathrm{Maya}(\lambda)$ which reads the outline of $[\lambda]$ (starting in the bottom-left, ending in the top-right), encoding an up-step as a 1 and a right-step as a 0. For instance,
\[
\ytableausetup{boxsize={.96em}}
[\lambda] =
\ydiagram{7,7,6,4,3,2,2} 
\begin{tikzpicture}[x=1em, y=1em, overlay, remember picture,xshift=-6.9em,yshift=-3.2em, blue]
    \draw (0.5,0) node {\scriptsize$0$};
    \draw (1.5,0) node {\scriptsize$0$};
    \draw (2,0.5) node {\scriptsize$1$};
    \draw (2,1.5) node {\scriptsize$1$};
    \draw (2.5,2) node {\scriptsize$0$};
    \draw (3,2.5) node {\scriptsize$1$};
    \draw (3.5,3) node {\scriptsize$0$};
    \draw (4,3.5) node {\scriptsize$1$};
    \draw (4.5,4) node {\scriptsize$0$};
    \draw (5.5,4) node {\scriptsize$0$};
    \draw (6,4.5) node {\scriptsize$1$};
    \draw (6.5,5) node {\scriptsize$0$};
    \draw (7,5.5) node {\scriptsize$1$};
    \draw (7,6.5) node {\scriptsize$1$};
\end{tikzpicture}
\leftrightsquigarrow 
~~
\mathrm{Maya}(\lambda) = 00110101001011\,.
\]
Then $[\lambda]$ is in bijection with the set of $(0,1)$-pairs of (non-necessarily consecutive) entries in the Maya diagram, where the $0$ precedes the $1$.
\begin{lem}\label{lem:maya size}
    Let $\lambda$ be a partition. The function $\mathrm{Maya}(\lambda) \mapsto |\lambda|$ is in $\FL$.
\end{lem}
\begin{proof}
The input is of length $m=\lambda_1+\ell(\lambda) = O(|\lambda|)$, so we need to construct a $\log(|\lambda|)$-space machine. We count the number of $(0,1)$ ordered pairs.

Initialise a counter $\mathtt{v}=0$.
Read $x=\mathrm{Maya}(\lambda)$ from right to left, keeping track of the position with a counter $\mathtt{i}$ in a loop, starting with $\mathtt{i} = m$.

If $x_{\mathtt{i}}=0$, set $\mathtt{i}\leftarrow \mathtt{i}-1$.
If $x_{\mathtt{i}}=1$, start another counter $\mathtt{j}=\mathtt{i}-1$ which performs an inner loop. Check if $x_{\mathtt{j}} = 0$, in which case $\mathtt{v}\leftarrow\mathtt{v}+1$; otherwise do nothing. Then decrease $\mathtt{j}$ by~$1$ and repeat. When $\mathtt{j}<0$, stop this inner $\mathtt{j}$-loop; we have added to $\mathtt{v}$ the number of 0s to the left of the 1 at $x_\mathtt{i}$; this corresponds to the number of boxes in $[\lambda]$ in one row.
\[
\ytableausetup{boxsize={.96em}}
\ydiagram{7,7,6,4,3,2,2}*[*(blue!30)]{0,0,0,0,3}
\begin{tikzpicture}[x=1em, y=1em, overlay, remember picture,xshift=-6.9em,yshift=-3.2em, blue!50]
    \draw[red] (0.5,0) node {\scriptsize$\mathbf 0$};
    \draw[red] (1.5,0) node {\scriptsize$\mathbf 0$};
    \draw (2,0.5) node {\scriptsize$1$};
    \draw (2,1.5) node {\scriptsize$1$};
    \draw[red] (2.5,2) node {\scriptsize$\mathbf 0$};
    \draw[red] (3,2.5) node {\scriptsize$\mathbf 1$};
    \draw (3.5,3) node {\scriptsize$0$};
    \draw (4,3.5) node {\scriptsize$1$};
    \draw (4.5,4) node {\scriptsize$0$};
    \draw (5.5,4) node {\scriptsize$0$};
    \draw (6,4.5) node {\scriptsize$1$};
    \draw (6.5,5) node {\scriptsize$0$};
    \draw (7,5.5) node {\scriptsize$1$};
    \draw (7,6.5) node {\scriptsize$1$};
\end{tikzpicture}
\]
Then set $\mathtt{i}\leftarrow \mathtt{i}-1$ and repeat until $\mathtt{i}=0$.

At the end, output the value of $\mathtt{v}$.
\end{proof}
For each $r\in\N_0$, the \emph{$r$-abacus} of $\lambda$ is the array $\mathrm{abc}_r(\lambda)$ of width $r$ created by splitting the Maya diagram of $\lambda$ into consecutive blocks of length $r$. Now 1s represent \emph{beads} of an abacus and 0s represent empty spaces (\emph{holes}). For instance, if $\lambda$ is the partition from the previous example, then the following is the $3$-abacus of $\lambda$.
\[
\mathrm{abc}_3(\lambda) = 
\begin{array}{ccc}
0&0&1\\[-.5em]
1&0&1\\[-.5em]
0&1&0\\[-.5em]
0&1&0\\[-.5em]
1&1& 
\end{array}
\hspace{10em}
\begin{tikzpicture}[x=1em, y=-1em, overlay, remember picture,xshift=.75em,yshift=3.5em]
    \fill[orange!70] (-.3,0) rectangle ++(3.6,-.3);
    \fill[orange!70] (-.1,0) rectangle ++(.2,7);
    \fill[orange!70] (1.4,0) rectangle ++(.2,7);
    \fill[orange!70] (2.9,0) rectangle ++(.2,7);
    \begin{scope}[yshift=-.7em, blue!80]
        \fill (0,1.3) circle (.5em);
        \fill (0,5.2) circle (.5em);
        \fill (1.5,2.6) circle (.5em);
        \fill (1.5,3.9) circle (.5em);
        \fill (1.5,5.2) circle (.5em);
        \fill (3,0) circle (.5em);
        \fill (3,1.3) circle (.5em);
    \end{scope}
\end{tikzpicture}
\begin{array}{ccc}
&&\textcolor{white}{1}\\[-.5em]
\textcolor{white}{1}&&
\textcolor{white}{1}\\[-.5em]
&\textcolor{white}{1}&\\[-.5em]
&\textcolor{white}{1}&\\[-.5em]
\textcolor{white}{1}&
\textcolor{white}{1}& 
\end{array}
\hspace{4em}{\,}
\smallskip
\]
A cell of $[\lambda]$ with hook-length divisible by $r$ becomes a now a pair (bead, hole) in $\mathrm{abc}_r(\lambda)$ both in the same column (\emph{runner}), with the hole above the bead 
\cite[\S\!\S11.3, 11.4]{Loehr}. The following proof is suggested to us by Michał Szwej.

\begin{lem}\label{lem:hl val}
    Fix a prime $p$.
    Let $\lambda$ be a partition. The function $\mathrm{Maya}(\lambda) \mapsto v_p(\prod_{c\in[\lambda]}\hl_\lambda(c))$ is in $\FL$.
\end{lem}
\begin{proof}
Let $n = |\lambda|$.
The input is of length $\lambda_1+\ell(\lambda) = O(n)$, so we need to construct a $\log(n)$-space machine. To do so, we interpret $\mathrm{Maya}(\lambda)$ as a $p$-abacus, count the number of beads below holes, then interpret it as a $p^2$-abacus, count the number of beads below holes, etc.

Initialise two counters $\mathtt{r}=p$ and $\mathtt{v}=0$.
Read $x=\mathrm{Maya}(\lambda)$ from right to left, keeping track of the position with a counter $\mathtt{i}$ in a loop.

Now perform the $\mathtt{i}$-loop and the $\mathtt{j}$-inner loop described in the proof of Lemma~\ref{lem:maya size}, except that $\mathtt{j}$ is initialised at $\mathtt{i}-\mathtt{r}$ and decreased by $\mathtt{r}$ each time.

At the end of the $\mathtt{i}$-loop, the value of $\mathtt{v}$ counts the cells of $[\lambda]$ whose hook-length is divisible by $\mathtt{r} = p$. Set $\mathtt{r} \leftarrow \mathtt{r}*p$ and repeat, adding to the same counter $\mathtt{v}$.
Whenever $\mathtt{r} > |x|$, the value of $\mathtt{v}$ is $v_p(\prod_{c\in[\lambda]}\hl_\lambda(c))$.

The maximum value that $\mathtt{v}$ can attain is upper-bounded by $\log_p(\prod \hl_\lambda(c))$, which by the hook-length formula is upper bounded by $\log_p(n!) = O(n\log(n)) = O(n^2)$ and hence can be stored in $O(\log_2(n))$-space.
\end{proof}

\begin{thm}[Hook-length formula]\label{thm:hook-length}
    The function $\mathrm{Maya}(\lambda)\mapsto\#\SYT(\lambda)$ is in $\sharpL$.
\end{thm}
\begin{proof}
    Write
    \[
    \#\SYT(\lambda) = \frac{n!}{\prod_{c\in[\lambda]} \hl_\lambda(c)} = \prod_{p\le n~\text{prime}} p^{v_p(n!) - v_p(\prod_{c\in[\lambda]} \hl_\lambda(c))},
    \]
    where we used that the prime divisors of $n!$ are at most $n$, and since $\#\SYT(\lambda)$ is an integer the same holds for the denominator.
    The function sending $\mathrm{Maya}(\lambda)$ to $v_p(n!) - v_p(\prod_{c\in[\lambda]} \hl_\lambda(c))$ is in $\FL$, as we saw in the third proof of Corollary~\ref{cor:factorial} and in Lemma~\ref{lem:hl val}.
    The function $(1^p) \mapsto p$ is trivially in $\sharpL$ (where we set $\emptyset \mapsto 1$).
    Hence by Theorem~\ref{thm:product}, given any fixed $p$ the function
    \[
    \mathrm{Maya}(\lambda) \mapsto p^{v_p(n!) - v_p(\prod_{c\in[\lambda]} \hl_\lambda(c))}
    \]
    is in $\sharpL$.     
    The function sending $(1^p)$ to $(1^p)$ if $p$ is a prime or $\emptyset$ otherwise is in $\FL$ by Theorem~\ref{primes}. Apply Theorem~\ref{thm:product} again to conclude.
\end{proof}

\begin{note}
    A similar proof can also show that if the partition $\lambda$ is encoded as $(1^{\lambda_1}\,0\,1^{\lambda_2}\,0\,\dots)$, then the function $\#\SYT(\lambda)$ is also in $\sharpL$. The idea is again to use prime factorization, either via the hook-length formula or Young's quotient formula
    \[
    \#\SYT(\lambda) = 
    |\lambda|!~\frac{\prod_{i<j}(\lambda_i-i-\lambda_j+j)}{\prod_{i}(\lambda_i-i+\ell(\lambda))!}
    \,.\] 
Alternatively, it is easy to see that a log-space machine can produce the Maya diagram (bit by bit) from $(1^{\lambda_1}\,0\,1^{\lambda_2}\,0\,\dots)$ and using log-space reductions as in \S\ref{sec:log-space reduction} apply the algorithm from Theorem~\ref{thm:hook-length}
\end{note}

\begin{question}\label{q:skewSYT}
    Given a pair of partitions $\lambda$, $\mu$ encoded as $(1^{\lambda_1}\,0\,1^{\lambda_2}\,0\,\ldots$ $\,0\,0\,1^{\mu_1}\,0\,1^{\mu_2}\,0\,\ldots)$, is the number of skew standard Young tableaux $\#\SYT(\lambda/\mu)$ also in $\sharpL$? 
\end{question}
\begin{note}
Our witness in Theorem~\ref{thm:hook-length} relies heavily on the presence of an explicit product formula instead of constructing SYTs, and the proof does not extend. In general there are no direct product formulas for the number of SYTs of skew shapes. They can be computed via Aitken's determinantal formula \cite[Corollary 7.16.3]{StanleyEC2}, which gives only $\GapL$, or as positive sums via the Littlewood--Richardson (LR) rule or via the Naruse hook-length formula for skew shapes (NHLF) as in~\cite{mpp1}. Via the LR rule, we have
$$f^{\lambda/\mu} = \sum_{\nu \vdash n} c^{\lambda}_{\mu\nu} f^\nu$$
where $n = |\lambda|-|\mu|$,
and hence a candidate witness would come from a pair of LR tableaux and an SYT of shape $\nu$. However, we do not believe that general LR coefficients (without a constant number of rows) are in $\sharpL$. 
The NHLF formula reads
$$f^{\lambda/\mu}=\sum_{D \in E(\lambda/\mu)}  \frac{n!}{\prod_{u \in [\lambda] \setminus D}\hl_\lambda(u)},$$
where the sum is over certain \emph{excited} diagrams $D$. 
Here, however, the individual terms $n!/\prod_u \hl_\lambda(u)$ are usually not integers, and so they do not count anything and are not in $\sharpL$. 
\end{note}

\section{Plethysm and Kronecker coefficients}\label{sec:OHara}
\ytableausetup{boxsize=.3em, centertableaux}

The plethysm product and the Kronecker product of Schur polynomials correspond to very natural operations on representations. Namely, the plethysm product corresponds to the composition of $\GL_n$-representations; the Kronecker product corresponds to the tensor product of $S_n$-representations. The \emph{plethysm coefficients} $a_{\mu[\nu]}^\lambda$ and \emph{Kronecker coefficients} $g(\lambda,\mu,\nu)$ are the structure constants of these products. 
In terms of symmetric polynomials,
\[
s_\mu\circ s_\nu = \sum_\lambda a^\lambda_{\mu[\nu]} s_\lambda
\quad\text{and}\quad
s_\mu*s_\nu = \sum_\lambda g(\lambda,\mu,\nu) s_\lambda.
\]
A family of $\GL_2$-plethysm coefficients coincides with certain rectangular Kronecker coefficients, see e.q.~\cite{IOT}:
\begin{equation}\label{eq:arectkron}
  a_{n[k]}^{(\lambda_1,\lambda_2)} = g((\lambda_1,\lambda_2),(n^k),(n^k)),
\end{equation}
but the general relationship remains mysterious with only few known inequalities, see e.g.~\cite{IP17}.
We refer to these as \emph{Hermite coefficients}, since they describe the spaces of covariants and invariants appearing in Hermite's law of reciprocity.

Hermite coefficients satisfy the following well-known formula 
\begin{equation}\label{eq:plet as diff}
a_{n[k]}^{(nk-r,r)} = g((nk-r,r),(n^k),(n^k)) = p_r(n\times k) - p_{r-1}(n\times k),
\end{equation}
where $p_r(n\times k)$ is the number of partitions of size $r$, of length at most $k$, and of width at most $n$. This formula places Hermite coefficients in $\FP$, since $p_r(n\times k)$ can be computed recursively in time $O(nk)$. Moreover, we have seen that $p_r(n\times k)$ is contained in $\sharpL$ (Corollary~\ref{cor:partitions in a box}), which places Hermite coefficients in $\GapL$, which is a subset of $\FP$. Note that both $\FP$ and $\GapL$ are superclasses of $\sharpL$. In Theorem \ref{thm:Hermite}, we show that these coefficients belong to another superclass of $\sharpL$, namely the class $\sharpTISP(\poly(n),\log^2(n))$ of functions counting the number of accepting paths of a non-deterministic $\poly$-time $\log^2$-space machine (see~\S\ref{sec:zoo}).

\begin{thm}[Hermite coefficients]\label{thm:Hermite}
The function $(1^n\,0\,1^k\,0\,1^r) \mapsto a_{n[k]}^{(nk-r,r)}$ is in $\sharpTISP(\poly(n),\log^2(n))\cap \GapL$.
\end{thm}

\begin{question}\label{q:Hermite}
  Are Hermite coefficients in $\sharpL$?
\end{question}

\begin{note}
    Any representation of $\sl_2$ has a crystal. But whereas Clebsch--Gordan's rule (or, more generally, Kashiwara's tensor product rule) constructs the crystal of the tensor product of two representations, no such rule exists for the plethysm product. In~\cite[Conjecture~4.18]{GutCrystals} a conjecturally correct construction for this rule is given. It is however likely that this would not help solving Question~\ref{q:Hermite}, in the same way that having an $\FL$ computable raising operator for tensor products (Theorem~\ref{thm:crystal operator}) does not place the relevant structure constants in~$\sharpL$ (see Note~\ref{note:crystal operators}).
\end{note}

\subsection{A combinatorial interpretation}
In this section we use work with quantum integers $[n]=(q^n-q^{-n})/({q-q^{-1}})$, quantum factorials $[n]!=[n]\cdot[n-1]\cdots[1]$,
and quantum binomials $\qbinom{n+m}{n}=[n+m]!/([n]![m]!)$, see \cite[\S2]{BD99}.
A celebrated combinatorial construction by O'Hara \cite{OHara} gives rise to the \emph{KOH formula} for quantum binomials due to Zeilberger \cite{Zeilberger, Macdonald}
    \begin{equation}
        \label{KOH formula}
    \qbinom{a+b}{b} = \sum_{\mu\vdash b} \, \prod_{m_j(\mu)\ne0} \qbinom{(a+2)j - 2|\mu'_{\le j}| + m_j(\mu)}{m_j(\mu)}, \tag{KOH}
    \end{equation}
where $|\mu'_{\le j}| = \mu_1' + \cdots + \mu_j'$.
This recurrence relation allows for expressing quantum binomials as products and sums of smaller quantum binomials iteratively, which continues to this day to facilitate the derivation of important consequences \cite{BQQW01, Zan11, PakPanovaUnimodality}. Of interest to us, it gives a positive combinatorial $\sharpP$-formula for \eqref{eq:arectkron}, see Proposition~\ref{p:polytime} below.

\begin{de}[\cite{PakPanovaSwanson}]
\label{de:KOH}
    A \emph{KOH tree} (see Figure \ref{fig:KOH tree}) is a rooted tree described as follows. Each vertex is labelled by a tuple $(\mu,a,b)$, where $\mu \vdash b$ is a partition, and $a\ge0$ and $b>0$ are integers. 
    Each edge is labelled by a positive integer.
    If $\mu = r^{m_r}\ldots2^{m_2}1^{m_1}$, then the vertex has a child for each $m_j\ne 0$. Let $J$ be the set of distinct row lengths of~$\mu$. Note that $J$ indexes the set of children. 
    The labels satisfy the following constraints: 
    \begin{enumerate}
        \item The vertex is a leaf if and only if $|\mu|=1$. In which case, we  abbreviate the label as $a$.
        \item A vertex $(\mu,a,b)$ has an outgoing edge labelled $j$ for each distinct row length $m_{j}(\mu)\ne0$. The corresponding child is labelled by $(\mu^{(j)},a^{(j)},m_{j}),$ where $\mu^{(j)}\vdash m_{j}$ and 
        \begin{equation}\label{eq:KOHtreea}
        a^{(j)} = (a+2)j - 2|\mu'_{\le j}|.
        \end{equation}
    \end{enumerate}
This defines what a KOH tree is.
Defining the left-to-right order of edges for a KOH tree makes it a plane tree,
which then induces a left-to-right order on the leaves.
In \cite{PakPanovaSwanson} the order is arbitrary,
but for us it is beneficial to order the edges lexicographically from left to right with one notable exception: every edge labelled $1$ is always the \emph{right-most} edge, see Figure~\ref{fig:KOH tree}.
In this way, the bijection in Proposition~\ref{p:bijection} preserves the order of the leaves.
    A \emph{marked KOH tree} of type $(n,k,r)$ is a KOH tree whose root label is some tuple $(-,n,k)$ (i.e., the first entry of the tuple is an arbitrary partition of $k$) and in which each leaf $a_i$ is marked by an integer~$t_i$ satisfying $t_1=0$,
\begin{itemize}
\item $t_{i-1}\leq t_i\leq a_i + t_{i-1}$, and
\item $t_i\leq \big(\sum_{j=1}^{i-1} a_j\big) - t_{i-1}$,
\end{itemize}
for all $1 < i \le K$, and $t_K = r+\big(\sum_{j=1}^{K} a_j\big)/2-nk/2$.
\end{de}

Let $\mathcal{T}(n,k,r)$ be the set of marked KOH trees of type $(n,k,r)$.
In \cite{PakPanovaSwanson}, it is shown that
Hermite coefficients are counted by these trees:
\[
g((nk-r,r),(n^k),(n^k)) = a_{n[k]}^{(nk-r,r)} = \#\mathcal{T}(n,k,r).
\]
This is a $\sharpP$ formula.
It is natural to ask whether this construction can be modified to give a $\sharpL$ formula. That is, whether there exists a verifier $\NL$ machine which checks the validity of a KOH tree. 
A naive approach is to encode a KOH tree via its Depth-First-Search traversal and check the defining conditions above. 
There are several obstacles to this approach. For instance, we saw in \S\ref{sec:structure constants} how one cannot store the partition labelling the root in a log-space machine. Moreover, KOH trees can be polynomially deep, requiring a polynomial amount of counters just to traverse them.
In the rest of this section we partially overcome these technical difficulties and prove Theorem~\ref{thm:Hermite}, which gives a $\log^2$-space verifiable formula.

We define below a compressed version of a KOH tree, which is logarithmically deep. We also avoid using partitions as labels of the nodes, by invoking the next elementary lemma about the frequency notation of partitions.

\begin{lem}[Frequency notation of partitions]\label{lem:partition} Fix $n$.
The map
\[
\lambda \mapsto \big(m_1(\lambda), m_2(\lambda), \ldots, m_{\lambda_1}(\lambda),0^{n
-\lambda_1}\big)
\]
is a bijection from the set of partitions of $n$ to the set
\[
\textstyle
\{(m_1,\ldots,m_{n
})\mid \sum_j j m_j=n
\} \subset (\N_0)^{n
}.
\]
We have
\(\ell(\lambda)=\sum_i m_i(\lambda)\)
and
\(
\lambda_j' = \ell(\lambda) - \sum_{i=1}^{j-1} m_i(\lambda)
\)
and $|\lambda| = \sum_j j\cdot m_j(\lambda) = n$.
\end{lem}
\begin{proof}
Clear from the definition of a partition.
\end{proof}
For example, the partition $(5,3,3,1,1,1,1)=(5^1\,3^2\,1^4)$ has a frequency vector (i.e., the list of exponents) $(4,0,2,0,1)$ and hence Lemma~\ref{lem:partition} maps $(5,3,3,1,1,1,1,0,\ldots,0)$ to $(4,0,2,0,1,0,\ldots,0)$.

Next, we present a lemma that allows for the computation of the values $a$ in the label of an \emph{internal node} (a non-leaf node) of a KOH tree from the remaining available information. 
Given a node $v$ in an edge-labelled rooted tree, we denote by $v\ua$ its parent, and by $v\da_j$ its child via the edge labelled $j$.
For any list of edge labels $(j_1,\ldots,j_y)$ we write $v\da_{(j_1,\ldots,j_y)} := (v\da_{(j_1,\ldots,j_{y-1})})\da_{j_y}$.
For any two nodes $V$ and $v$ in $T$, if there is a path from $V$ to $v$, then it is unique and we denote it by $P_T(V,v)$.
For any node $v$ of a KOH tree, define its \emph{main ancestor} to be the earliest ancestor $V$ of $v$ such that the edges in the path $P_T(V,v)$ have labels $(j\, 1^x)$ for some $j\ge1$ and some $x\ge0$. Note that $v = V\da_{j,1^x}$.

\begin{lem}\label{lem:a}
Let $v$ be a node of a KOH tree $T$ different from the root. Let $V = (\mu,a,b)$ be the main ancestor of $v$.
Write
\[
\Big(\mu^{(j,1^y)}, a^{(j,1^y)}, b^{(j,1^y)}\Big)
\]
for the label of $V\da_{j,1^y}$.
Then for all $x\geq 0$ we have
\[
a^{(j,1^x)} = 2x + (a+2)j - 2|\mu'_{\le j}| - 2\sum_{y=0}^{x-1} \ell(\mu^{(j,1^y)}).
\]
\end{lem}

\begin{proof}
In the base case $x = 0$ the right-most sum is empty and the statement holds by~\eqref{eq:KOHtreea}. For the induction step ($x\geq 1$), we calculate
\begin{align*}
a^{(j,1^x)} &\stackrel{\eqref{eq:KOHtreea}}= (a^{(j,1^{x-1})}+2)\cdot1 - 2 \ell(\mu^{(j,1^{x-1})})\\
&\stackrel{\textup{ind.\,hyp.}}{=} 2x + (a+2)j - 2 |\mu'_{\le j}| - 2 \sum_{y=0}^{x-1}\ell(\mu^{(j,1^{y})}).
\qedhere
\end{align*}
\end{proof}

We are now ready to introduce the main object of study of this section, the small KOH witnesses. See Figure \ref{fig:KOH tree} for an example and Proposition~\ref{p:bijection} for a construction of small KOH witnesses from marked KOH trees.

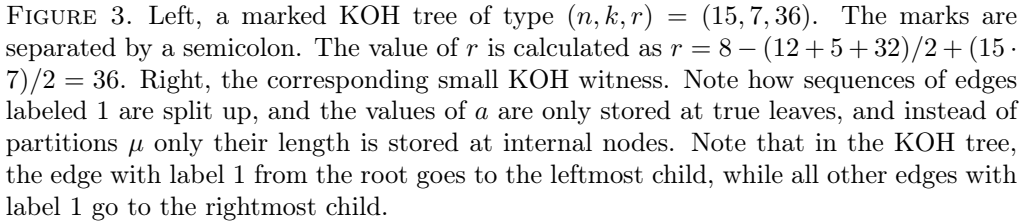
\begin{figure}
\[
\begin{tikzpicture}[baseline=11em, x=2em, y=4em]
    \node (root) at (2.5,5) {$\ydiagram{2,2,1,1,1},~15,7$};
\coordinate (coordn)  at ($(root.north)+(0.1,-0.15)$);
\coordinate (coordk)  at ($(root.east)+(-0.1,0)$);
\draw[dotted,-,blue] (4.2,5.5) node[right]{\textit{$n=a(\textup{{root}})$}} to[out=180,in=90] (coordn);
\draw[dotted,-,blue] (4.5,5) node[right]{\textit{$k=b(\textup{{root}})$}} to[out=180,in=0] (coordk);
\node (child0) at (1,4) {$\ydiagram{2,1},~7, 3$};
    \node (child1) at (4,4) {$\ydiagram{1,1},~20, 2$};
\coordinate (coorda)  at ($(child1.north)+(0.1,-0.05)$);
\coordinate (coordb)  at ($(child1.east)+(-0.1,0)$);
\coordinate (coordj)  at ($(child1.south)+(0,-.3)$);
\draw[dotted,-,blue] (5.2,4.5) node[right]{\textit{$a$}} to[out=180,in=90] (coorda);
\draw[dotted,-,blue] (5.5,4) node[right]{\textit{$b$}} to[out=180,in=0] (coordb);
\draw[dotted,-,blue] (5.5,3.5) node[right]{\textit{$j$}} to[out=180,in=0] (coordj);
    \node (child00) at (2,3) {$5;2$};
    \node (child01) at (0,3) {$12;0$};
\coordinate (coordlabel)  at ($(child01.west)+(0.1,0)$);
\coordinate (coordmark) at ($(child01.east)+(-0.05,0)$);
\draw[dotted,-,blue] (-.3,1.5) node[right]{\textit{label $a$}} to[out=150,in=210] (coordlabel);
\draw[dotted,-,blue] (2,2) node[left]{\textit{mark $t$}} to[out=30,in=350] (coordmark);
    \node (child10) at (4,3) {$\ydiagram{1,1},~18, 2$};
    \node (child100) at (4,2) {$\ydiagram{2},~16, 2$};
    \node (child1000) at (4,1) {$32;8$};
    \draw[->] 
    (root)--(child0) node[midway, draw, fill=white, circle, inner sep=.05em] {\scriptsize $1$};
    \draw[->] 
    (root)--(child1) node[midway, draw, fill=white, circle, inner sep=.05em] {\scriptsize $2$};
    \draw[->] 
    (child1)--(child10) node[midway, draw, fill=white, circle, inner sep=.05em] {\scriptsize $1$};
    \draw[->] 
    (child10)--(child100) node[midway, draw, fill=white, circle, inner sep=.05em] {\scriptsize $1$};
    \draw[->] 
    (child100)--(child1000) node[midway, draw, fill=white, circle, inner sep=.05em] {\scriptsize $2$};
    \draw[->] 
    (child0)--(child00) node[midway, draw, fill=white, circle, inner sep=.05em] {\scriptsize $1$};
    \draw[->] 
    (child0)--(child01) node[midway, draw, fill=white, circle, inner sep=.05em] {\scriptsize $2$};
\end{tikzpicture}
\qquad
\hspace{-1.2cm}{\scalebox{1.5}{$\leftrightsquigarrow$}}\hspace{-0.8cm}
\qquad
\begin{tikzpicture}[baseline=16em, x=4em, y=5em]
    \node (root) at (2.5,5) {$5,3,7$};
\coordinate (smallcoordell)  at ($(root.west)+(0,0)$);
\coordinate (smallcoordm)  at ($(root.north)+(0,-.05)$);
\coordinate (smallcoordk) at ($(root.east)+(-.05,0.1)$);
\draw[dotted,-,blue] (1.2,4.8) node[left]{\textit{$\ell\!\left(\ydiagram{2,2,1,1,1}\right)$}} to[out=30,in=190] (smallcoordell);
\draw[dotted,-,blue] (1.6,5.2) node[left]{\textit{$m_1\!\!\left(\ydiagram{2,2,1,1,1}\right)$}} to[out=0,in=100] (smallcoordm);
\draw[dotted,-,blue] (3.5,5) node[right]{\textit{$k$}} to[out=150,in=30] (smallcoordk);
    \node (child0) at (.5,4) {$2,1,3$};
    \node (child1) at (2.5,4) {$2,2,2$};
\coordinate (smallcoordellv)  at ($(child1.south west)+(0.1,0.08)$);
\coordinate (smallcoordm) at ($(child1.south)+(0,0.05)$);
\coordinate (smallcoordb) at ($(child1.south east)+(-0.15,0.08)$);
\draw[dotted,-,blue] (2,3.5) node[left]{\textit{$\ell$}} to[out=0,in=240] (smallcoordellv);
\draw[dotted,-,blue] (2.5,3.6) node[below]{\textit{$m_1$}} to[out=60,in=270] (smallcoordm);
\draw[dotted,-,blue] (3,3.5) node[right]{\textit{$b$}} to[out=180,in=300] (smallcoordb);
    \node (child00) at (1.5,4) {$5;2$};
    \node (child01) at (.5,3) {$12;0$};
\coordinate (smallcoordlabel)  at ($(child01.west)+(0.1,0)$);
\coordinate (smallcoordmark) at ($(child01.east)+(-0.05,0)$);
\draw[dotted,-,blue] (-.3,2) node[right]{\textit{label $a$}} to[out=150,in=210] (smallcoordlabel);
\draw[dotted,-,blue] (1.5,2.5) node[left]{\textit{mark $t$}} to[out=30,in=350] (smallcoordmark);
    \node (child10) at (3.5,4) {$2,2,2$};
    \node (child100) at (4.5,4) {$1,0,2$};
    \node (child1000) at (4.5,3) {$32;8$};
\coordinate (smallcoordj) at ($(child1000.north west)+(0.1,0.2)$);
\draw[dotted,-,blue] (3.5,3) node[below]{\textit{$j,x$}} to[out=90,in=180] (smallcoordj);
    \filldraw[->] 
    (root)--(child0) node[pos=.7, draw, fill=white, inner sep=.1em] {\scriptsize $1,0$};
    \draw[->] 
    (root)--(child1) node[pos=.7, draw, fill=white, inner sep=.1em] {\scriptsize $2,0$};
    \draw[->] 
    (root)--(child10) node[pos=.7, draw, fill=white, inner sep=.1em] {\scriptsize $2,1$};
    \draw[->] 
    (root)--(child100) node[pos=.7, draw, fill=white, inner sep=.1em] {\scriptsize $2,2$};
    \draw[->] 
    (child100)--(child1000) node[pos=.7, draw, fill=white, inner sep=.1em] {\scriptsize $2,0$};
    \draw[->] 
    (root)--(child00) node[pos=.7, draw, fill=white, inner sep=.1em] {\scriptsize $1,1$};
    \draw[->] 
    (child0)--(child01) node[pos=.7, draw, fill=white, inner sep=.1em] {\scriptsize $2,0$};
\coordinate (coordtrueleaf1)  at ($(child01.south)+(0.0,0)$);
\coordinate (coordtrueleaf2) at ($(child00.south)+(0.0,0.0)$);
\coordinate (coordtrueleaf3) at ($(child1000.south)+(0.0,0.0)$);
\node (trueleaveslabel) at (2.5,1.8) {\textcolor{blue}{\textit{true leaves}}};
\draw[dotted,-,blue] (trueleaveslabel)  to[out=180,in=270] (coordtrueleaf1);
\draw[dotted,-,blue] (trueleaveslabel)  to[out=120,in=270] (coordtrueleaf2);
\draw[dotted,-,blue] (trueleaveslabel)  to[out=0,in=270] (coordtrueleaf3);
\end{tikzpicture}
\]
\caption{Left, a marked KOH tree of type $(n,k,r)=(15,7,36)$.
The marks are separated by a semicolon.
The value of $r$ is calculated as $r=8-(12+5+32)/2+(15\cdot 7)/2=36$.
Right, the corresponding small KOH witness.
Note how sequences of edges labeled 1 are split up,
and
the values of $a$ are only stored at true leaves, and instead of partitions $\mu$ only their length is stored at internal nodes. Note that in the KOH tree, the edge with label 1 from the root goes to the leftmost child, while all other edges with label 1 go to the rightmost child.}
\label{fig:KOH tree}
\end{figure}

\begin{de}\label{de:witness}
A \emph{small KOH witness} of type $(n, k, r)$ is a labelled plane rooted tree $T$ and a subset $L$ of its leaves, the so-called set of \emph{true leaves}, such that
\begin{itemize}
\item each edge is labelled by a pair $(j, x)\in\N\times\N_0$,
\item each true leaf $v$ is labelled by a pair $(a,t)\in\N_0 \times\N_0$, where we say that $a$ is the label and $t(v)=t$ is the mark, and we define $\ell(v)=1$, $m_1(v)=1$, and $b(v)=1$ (see \eqref{eq:a} and \eqref{sw:a} for the definition of $a(v)$),
\item each of the remaining nodes (so-called \emph{internal nodes}, even though these can be leaves, but not true leaves) is labelled by a triple $(\ell,m_1,b)\in\N\times\N_0\times\N$,
and we set $\ell(v)=\ell$, $m_1(v)=m_1$, and $b(v)=b$,
\end{itemize}
and the labels and marks are subject to the following numbered relations \eqref{sw:lex}--\eqref{sw:tK}.

\begin{enumerate}
\item The outgoing edges of a node are lexicographically ordered. \label{sw:lex}
\item Only the root may have an outgoing edge labelled $(1,0)$. \label{sw:root}
\item If a node has an outgoing edge labelled $(j,x)$ for some $x>0$ then it also has an outgoing edge $(j,x-1)$. \label{sw:x-1}
\item If a node has an outgoing edge labelled $(j,x)$ then it has no other outgoing edge labelled $(j,x)$. \label{sw:unique}
\end{enumerate}
For any node $v$ different to the root, let $v\ua$ be its parent.
For any node which is not a leaf, let $v\da_{j,x}$ be its unique child along an edge labelled $(j,x)$, if it exists. If $w$ is not a well-defined node, let $b(w) = 0$.

Let $e(v) = (v\ua, v)$ be the unique incoming edge to $v$. 
Let $j(v)$ and $x(v)$ be a shorthand for the labels of $e(v)$. 
\begin{enumerate}[resume]
\item For every non-root $v$ it holds that 
$$
    b(v\ua\da_{j(v),x(v)+1}) + \sum_{j>1} b(v\da_{j,0}) = \ell(v).
$$ \label{sw:len}
\item For the root $R$, it holds $\sum_{j} j\cdot b(R\da_{j,0}) = k$. \label{sw:k}
\item
For each non-root $v$ we have
\[
m_1(v\da_{j,x})=b(v\da_{j,x+1}).
\]
For the root $v$, we have $m_1(v)=b(v\da_{1,0})$.
\label{sw:mone}
\item For any node with outgoing edges, it holds 
$$
b(v\ua\da_{j(v),x(v)+1}) + \sum_{j} j\cdot b(v\da_{j,0}) = b(v).
$$ \label{sw:b}
\end{enumerate}
For any node $v$ and $j\in\N$, define recursively the \emph{column sum} operator
\[
c_\Sigma(v,1) = \ell(v)
\]
and
\[
c_\Sigma(v,j) = 
c_\Sigma(v,j-1) + \ell(v)
-m_1(v)
-
\sum_{1<i<j} b(v\da_{i,0})
.
\]
Define recursively
\begin{equation}\label{eq:a}
a(v) = 
2x(v) + \big(a(v\ua) + 2\big)j(v) - 2c_\Sigma(v\ua,j(v)) - 2\sum_{y< x(v)}\ell(v\ua\da_{j,y}),
\end{equation}
where the root $R$ satisfies $a(R) = n$.
\begin{enumerate}[resume]
\item Given a true leaf $v$ with label $a$ and mark $t$, it holds that $a(v)=a$. \label{sw:a}
\end{enumerate}
While performing a left-to-right Depth-First-Search on the tree, at arrival to a true leaf $v$, let $\mathsf{left}(v)$ be the last encountered true leaf. 
If $v_1$ is the first true leaf, then set $s(v_1)=a(v_1)$ and for all remaining true leaves let
\[
s(v) = a(v) + s(\mathsf{left}(v)).
\]
\begin{enumerate}[resume]
\item The first encountered true leaf $v_1$ satisfies $t(v_1)=0$. \label{sw:t1}
\item For any other true leaf $v$ it holds
\begin{itemize}
\item $t(\mathsf{left}(v))\leq t(v) \leq a(v) + t(\mathsf{left}(v))$, and
\item $t(v)\leq s(\mathsf{left}(v)) - t(\mathsf{left}(v))$.
\end{itemize}\label{sw:marks}
\item For the last true leaf $v$ ever encountered, $t(v) = r+\big(s(v)-nk\big)/2$.\label{sw:tK}
\end{enumerate}

\end{de}

\begin{p}\label{p:bijection}
Given a marked KOH tree $T$, construct a new tree $T'$ as follows.
\begin{enumerate}
\item There is a one-to-one correspondence between nodes of $T$ and nodes of $T'$. 
A leaf of $T$ labelled $((1),a,1)$ and marked with $t$ corresponds to a true leaf with label $a$ and mark $t$ in $T'$. The leaves in $T$ from left to right appear in the same order as the true leaves in $T'$.
Any other node $(\mu,a,b)$ of $T$ corresponds to an internal node of $T'$ labelled $(\ell(\mu),m_1(\mu),b)$.
\item For each path $P_T(u,v)$ whose edges are labelled $(j,1^x)$, $j>1$, $x\geq 0$, we create an edge $(u,v)$ in $T'$ with label $(j,x)$,
and for paths starting at the root we also do so for $j=1$.
\item Turn the resulting graph into a plane rooted tree, where the root is the node of $T'$ corresponding to the root of $T$, and where outgoing edges of any given node are ordered lexicographically.
\end{enumerate}
Then $T'$ is a small KOH witness. Furthermore, this is a bijection between the set $\mathcal{T}(n,k,r)$ of marked KOH trees and the set $\mathcal{T}'(n,k,r)$ of small KOH witnesses of the same type.
\end{p}

\begin{proof}
Suppose $T$ is a marked KOH tree of type $(n,k,r)$ and construct $T'$.
We begin by showing that $T'$ is a small KOH witness of type $(n,k,r)$.

Conditions \eqref{sw:lex}--\eqref{sw:unique} of the definition of a small KOH witness are trivially satisfied. Condition \eqref{sw:len} holds since the outgoing edges of a node $(\mu,a,b)$ in $T$ correspond to the various parts of $\mu$,
and the $b$-values of the children sum up to~$\ell(\mu)$. Similarly, conditions \eqref{sw:k} and \eqref{sw:b} hold because $\sum_j j\cdot m_j(\mu) = |\mu|$. 
Condition~\eqref{sw:mone} follows directly from Definition~\ref{de:KOH}.
Condition \eqref{sw:a} holds by Lemma \ref{lem:a}. Conditions \eqref{sw:t1}--\eqref{sw:tK} hold by definition of a \emph{marked} KOH tree.

We now show how given a small KOH witness $T'$ one can recover $T$. First note that for any internal node $v$ of $T'$ different from the root one can recover the triple $(\mu,a,b)$ labelling the corresponding node of $T$ as follows: $a$ is given by the function $a(v)$ (see Lemma~\ref{lem:a}), $b$ is given by $b(v)$ (Lemma~\ref{lem:partition}), and $\mu$ is completely determined by the formulas
\begin{align*}
\mu'_1 &= \ell(v),\\
\mu'_2 &= \ell(v) - b(v\ua\da_{j(v),x(v)+1}),\\
\mu'_j &= \ell(v) - b(v\ua\da_{j(v),x(v)+1}) - \sum_{1<i<j} b(v\da_{i,0})
\end{align*}
for any $j\ge3$ (again by Lemma~\ref{lem:partition}). For the root $R$, a similar set of formulas hold, by replacing $b(v\ua\da_{j(v),x(v)+1})$ by $b(v\da_{1,0})$. Similarly, for any true leaf labelled with $a$ and marked with $t$, the corresponding leaf of $T$ is labelled by $((1),a,1)$ and marked by $t$.
For any edge $(u,v)$ labelled $(j,0)$ of $T'$, create an edge labeled $j$
between the node of $T$ corresponding to $u$ and the node of $T$ corresponding to $v$. For any edge $(u,v)$ labelled $(j,x)$ with $x\ge1$, create an edge labeled $1$ between the node of $T$ corresponding to $u\da_{j,x-1}$ and the node of $T$ corresponding to $v$.

Since $b(v)=1$ for true leaves, a leaf of $T$ must correspond to a partition of $1$. Conditions~\eqref{sw:root}, \eqref{sw:x-1}, and~\eqref{sw:unique} ensure that any node of $T$ has at most one outgoing edge labelled $j$ for each $j$. Conditions~\eqref{sw:len}--\eqref{sw:b} ensure that for each node $(\mu,a,b)$ of $T$, then $\mu$ is a partition of $b$ and the outgoing edges correspond to the distinct parts of $\mu$ (see Lemma~\ref{lem:partition}). Condition~\eqref{sw:a} checks the defining property of $a$ is satisfied, thanks to Lemma~\ref{lem:a}. Finally, Condition~\eqref{sw:lex} fixes an order of the leaves of $T$ and \eqref{sw:t1}--\eqref{sw:tK} ensure that the marks of these leaves are consistent with those of a KOH tree.

It is clear that these two processes are inverse to each other.
\end{proof}

\begin{lem}\label{lem:log depth}
A small KOH witness of type $(n,k,r)$ has $O(\log(k))$ depth.
\end{lem}
\begin{proof}
Let $T'$ be a small KOH witness and let $T$ be the corresponding marked KOH tree.
Recall that each node $u$ of $T'$ corresponds to a node of $T$ labelled by a triple $(\mu,a,b)$. Suppose $u$ is not the root. If $v$ is a child of $u$ in $T'$, then there exists a path in $T$ in which there is one edge labelled $j>1$. Thus the partition labelling $v$ in $T$ is of size at most $|\mu|/j = b/j \le b/2$.

Consequently, a path from the root $(\lambda,n,k)$ to a leaf of $T'$ can be of length at most $\log_2(k)+1$.
\end{proof}

\begin{lem}\label{lem:at most b leaves}
    A small KOH witness of type $(n,k,r)$ has at most $k$ true leaves. 
\end{lem}
\begin{proof}
We use Proposition~\ref{p:bijection} and instead show that a KOH tree of type $(n,k,r)$ has at most $k$ leaves.

Let $T(\mu,n,k)$ denote a marked KOH tree with root $(\mu,n,k)$.
    The proof is by induction on $k$. If $k=1$ then the tree $T(\mu,n,k)$ consists of one vertex which is the leaf.
    Suppose the trees of type $(n',k',r')$ with $k'\leq k-1$ have  at most $k'$ many leaves. 
    Consider the children of the root of $T(\mu,n,k)$. Let $\mu=(1^{m_1}2^{m_2}\ldots r^{m_r})$. First, suppose that $m_1 \neq k$ (have $m_i \leq k/i <k$ for all $i\geq 2$). Then the children vertices are $v_j = (\mu^{(j)}, a^{(j)}, m_j)$ for $j\in J:=\{ j: m_j\geq 1\}$ (the set of row lengths of $\mu$). We have that $\sum_j j m_j = |\mu|=k$ and hence $\sum_j m_j \leq k$. 
    Finally, the set of  leaves of $T(\mu,n,k)$ is the union of the sets of leaves of $T(v_j)$. By induction, each $T(v_j)$ has at most $m_j<k$ many leaves and so the total number of leaves of $T(\mu,n,k)$ is $\leq \sum_j m_j \leq k$.

    Next, suppose that $m_1=k$. Then $n_1=(n+2)1-2k=n +2(1-k)<n$ and there is exactly one child $(\mu^{(1)},a^{(1)},k)$. If $\mu^{(1)} \neq (1^k)$ the above analysis applies. Otherwise we have a path of vertices $( (1^k),a^{(j)},k)$ with $a^{(j)}$ strictly decreasing until it branches out, at which point the above analysis applies, and the proof is done by induction.
\end{proof}

\begin{lem}\label{lem:sum of a values}
The sum of all values $a$ at the true leaves of a small KOH witness, as well as a KOH tree, is at most $nk$.
\end{lem}
\begin{proof}
We use Proposition~\ref{p:bijection} and instead show that the sum of the $a$-values at all leaves of a KOH tree of type $(n,k,r)$ is at most $nk$. To see the last part, we use the fact that the KOH tree definition from \cite{PakPanovaSwanson} is a direct unravelling of the KOH formula~\eqref{KOH formula}, and the intermediate label $a^{(j)}$ computed via equation~\eqref{eq:KOHtreea} are the integers appearing in the binomials of~\eqref{KOH formula}, so we have
$$    \qbinom{a+b}{b} = \sum_{\mu\vdash b} \, \prod_{m_j(\mu)\ne0} \qbinom{a^{(j)} + m_j(\mu)}{m_j(\mu)}.
$$
Since these are Laurent polynomials in $q$ with positive coefficients, the maximal degrees of the products do not exceed the maximal degree on the left, which is $ab$.
So for every vertex labelled $(\mu,a,b)$ we have
$$ \sum_j a^{(j)}m_j(\mu) \leq ab,$$
where the labels of its children are $(-,a^{(j)},m_j(\mu))$.
The KOH formula is applied recursively to each quantum binomial coefficient, giving the branching in the vertices of the tree, and the leaves appear when $b=1$. Expanding the above inequality along the tree we get that the sum of the leaf labels under the vertex $(\mu,a,b)$ is at most $ab$, and the claim follows since the root has $(a,b) =(n,k)$.
\end{proof}

\begin{lem}\label{lem:koh tree depth}
    A KOH tree of type $(n,k,r)$ has depth $O(nk)$.
\end{lem}
\begin{proof}
    We use the idea from the proof of Lemma~\ref{lem:sum of a values}. For a vertex $v$ labeled by $(\mu,a,b)$, let $\phi(v) = ab$, we will show that the $\phi$ values decrease strictly from a parent to a child as long as the child is not a leaf.
    
    Let $v$ be an internal vertex, so $b>1$. 
    First, suppose that $\mu=(1^b)$, then $v$ has only one child $w$ with a label $(\mu^1,a_1,b_1)$ where
    $\mu^1 \vdash b_1=b$ and $a_1 = (a+2)1 -2b \leq a-2$ since $b\geq 2$. Then 
    $$\phi(w) =a_1b_1 \leq (a-2)b = \phi(v) -2b \leq \phi(v)-4.$$ 

    Next, if $\mu = (1^{m_1}2^{m_2}\ldots)$ with $m_j \geq 1$ for some $j>1$, we consider the children $w^{(j)}$ corresponding to $m_j \geq 1$ with labels $(\mu^{(j)}, a^{(j)},m_j)$. 
    We have that 
    $$\phi(w^{(j)}) = ( (a+2)j -2|\mu'_{\leq j}|)m_j =a\cdot(jm_j) -2m_j\cdot ( |\mu'_{\leq j}|-j).$$
    For $j=1$, if $m_1 >0$ we have $\mu'_1=m_1+\cdots = \ell(\mu) \geq 2$, since the partition is not just a single row or column. At the same time, $1m_1 = b - \sum_{j \geq 2} jm_j \leq b-2$, since $m_j\geq 1$ for some $j\geq 2$. So $\phi(w^{(1)}) \leq a(1m_1)-2 \leq a (b-2)-2\leq  ab -4=\phi(v) -4$. 
    
     For $j>1$ with $m_j>0$ we have that $|\mu'_{\leq j}|= \sum_{i\leq j} (m_i+m_{i-1}+\cdots+m_b) >j $ unless $\mu=(b)$ and so $\phi(w^{(j)}) \leq a\cdot(jm_j)-2 \leq \phi(v)-2$. If $\mu=(b)$ then $w$ is a leaf and $\phi(w)=\phi(v)$.

     Hence in all cases, but possibly $w$ being a leaf, we have that $\phi(w) \leq \phi(v)-2$. 

     For every path from the root to a leaf $v_0 \to v_1 \to v_2 \to \cdots \to v_s$ we have that 
     $nk-2s=\phi(v_0) -2s \geq  \phi(v_1)-2(s-1) \geq \cdots  \geq \phi(v_{s-1}) -2 \geq \phi(v_s)-2 \geq -2,$ so $s < (nk+2)/2$, and the path has at most $\lfloor nk/2 \rfloor+2$ many vertices. 
\end{proof}

\begin{lem}\label{lem:poly many nodes}
    A small KOH witness and a KOH tree of type $(n,k,r)$ have $O( nk^2 )$ many nodes.
\end{lem}
\begin{proof}
   From the bijection between nodes, the number of vertices in the small witness and the original tree are the same per Proposition~\ref{p:bijection}. We will bound the number of KOH tree vertices. 
   First, the number of leaves in a KOH tree is at most~$k$, which is Lemma~\ref{lem:at most b leaves}.  For each leaf, consider the shortest path in the tree to the root, by Lemma~\ref{lem:koh tree depth} it has $O(nk)$ vertices. Every vertex of the tree lies on (at least) one such path  (a path that starts from a leaf below it), thus the total number of vertices is bound by the number of leaves times the maximal length of these paths, which gives $O(nk^2)$. 
\end{proof}

\subsection{Encoding the witness}
For the sake of concreteness, we present an encoding of a small KOH witness over a finite alphabet.
\begin{itemize}
\item Tuples are encoded with parentheses \texttt{()} and commas, as in usual mathematical notation, and semicolons are used to separate the marks.
\item non-negative integers are encoded in their decimal presentation.
\item A true leaf is encoded as $\texttt{trueleaf(}a\texttt{;}t\texttt{)}$, where $a$ is the label and $t$ is the mark.
\item The subtree of an internal node is encoded 
either as
\[\texttt{rootlabel(}\ell_0\texttt{)}\]
if it has no children, or otherwise as
\begin{align*}
\texttt{tree(}&\texttt{rootlabel(}\ell_0\texttt{),}\\
&\texttt{edge(edgelabel(}\ell_1\texttt{),}t_1\texttt{),}\\
&\texttt{edge(edgelabel(}\ell_2\texttt{),}t_2\texttt{),}\ldots\texttt{)}
\end{align*}
where $\ell_0$ is the node label and $\ell_i$ for $i\geq 1$ are edge labels,
and $t_i$ are trees, i.e., either leaf nodes or subtrees of non-leaf nodes.
\end{itemize}

See Figure \ref{fig:witness} for an example.

\begin{figure}
\begin{tabular}{cp{.25\textwidth}}
$\begin{tikzpicture}[baseline=27em, x=4em, y=5em]
    \node (root) at (2.5,5) {$5,3,7$};
    \node (child0) at (.5,4) {$2,1,3$};
    \node (child1) at (2.5,4) {$2,2,2$};
    \node (child00) at (1.5,4) {$5;2$};
    \node (child01) at (.5,3) {$12;0$};
    \node (child10) at (3.5,4) {$2,2,2$};
    \node (child100) at (4.5,4) {$1,0,2$};
    \node (child1000) at (4.5,3) {$32;8$};
    \filldraw[->] 
    (root)--(child0) node[pos=.7, draw, fill=white, inner sep=.1em] {\scriptsize $1,0$};
    \draw[->] 
    (root)--(child1) node[pos=.7, draw, fill=white, inner sep=.1em] {\scriptsize $2,0$};
    \draw[->] 
    (root)--(child10) node[pos=.7, draw, fill=white, inner sep=.1em] {\scriptsize $2,1$};
    \draw[->] 
    (root)--(child100) node[pos=.7, draw, fill=white, inner sep=.1em] {\scriptsize $2,2$};
    \draw[->] 
    (child100)--(child1000) node[pos=.7, draw, fill=white, inner sep=.1em] {\scriptsize $2,0$};
    \draw[->] 
    (root)--(child00) node[pos=.7, draw, fill=white, inner sep=.1em] {\scriptsize $1,1$};
    \draw[->] 
    (child0)--(child01) node[pos=.7, draw, fill=white, inner sep=.1em] {\scriptsize $2,0$};
\end{tikzpicture}$
& 
\small
\texttt{tree(rootlabel(5,3,7),}\linebreak
\phantom{ww}\texttt{edge(edgelabel(1,0),}\linebreak
\phantom{www}\texttt{tree(rootlabel(2,1,3),}\linebreak
\phantom{wwww}\texttt{edge(edgelabel(2,0),}\linebreak
\phantom{wwwww}\texttt{trueleaf(12;0)))),}
\phantom{ww}\texttt{edge(edgelabel(1,1),}\linebreak
\phantom{www}\texttt{trueleaf(5;2)),}\linebreak
\phantom{ww}\texttt{edge(edgelabel(2,0),}\linebreak
\phantom{www}\texttt{rootlabel(2,2,2)),}\linebreak
\phantom{ww}\texttt{edge(edgelabel(2,1),}\linebreak
\phantom{www}\texttt{rootlabel(2,2,2)),}\linebreak
\phantom{ww}\texttt{edge(edgelabel(2,2),}\linebreak
\phantom{www}\texttt{tree(rootlabel(1,0,2),}\linebreak
\phantom{wwww}\texttt{edge(edgelabel(2,0),}\linebreak
\phantom{wwwww}\texttt{trueleaf(32;8))))}
\end{tabular}
\caption{Left, a small KOH witness. Right, its encoding.
}
\label{fig:witness}
\end{figure}
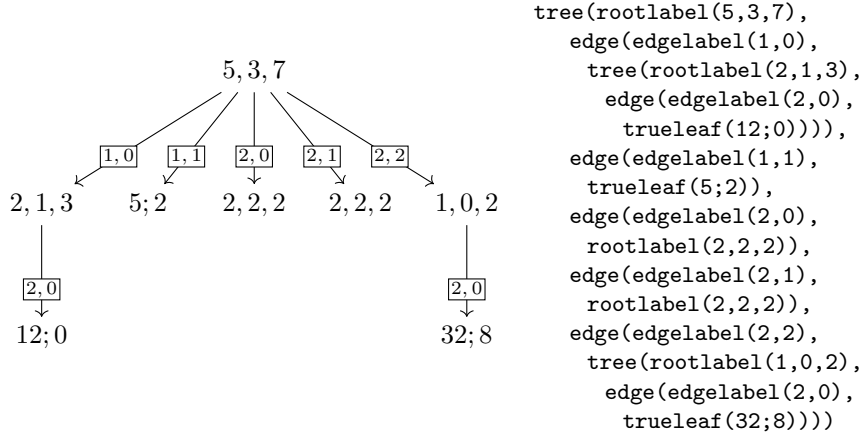

\subsection{The polylog space machine}\label{sec:polylog_koh}
Given $1^n\,0\,1^k\,0\,1^r$ on the input tape,
we can read the encoding of a small KOH witness from left to right once
and verify that it is a small KOH witness of type $(n,k,r)$ using only logarithmic space,
as follows.
\begin{itemize}
\item While reading, we ensure the syntactic correctness of the string of symbols. For this, we store the number of opened parentheses, and the labels of the unclosed functions, e.g., \texttt{edge} or \texttt{edgelabel}.
\item We always fully read and store what we read until we reach the closing parenthesis of a node label.
We then check the validity and update the data, and then keep reading.
\item When visiting an edge in any level $K$ of the tree, we check conditions \eqref{sw:lex}, \eqref{sw:root}, \eqref{sw:x-1}, \eqref{sw:unique} before deleting the label of the previously visited edge in level~$K$.
\item When visiting a node $v$, let $P_T(R,v)$ be the path from root to $v$. For each level of this path, we store 
\begin{itemize}
\item the current value of the sum $\sum_{j=1}^{i} b(v\da_{j,0})$,
\item the label $\ell(v)$ of the last visited node and its parent edge, in case it is needed for \eqref{sw:len},
\item the label $b(v)$
\item the label $m_1(v)$ and $m_1(v\ua\da_{j(v),x(v)-1})$ (the latter is used for verifying~\eqref{sw:mone})
\item the current value of the sum $\sum_{j=1}^{i} j\cdot b(v\da_{j,0})$,
\item the current value of the sum $c_\Sigma(v\ua, j(v))$,
\item the current value of the sum $\sum_{y=1}^x \ell(v\ua\da_{j,1^y})$.
\end{itemize}
We then check \eqref{sw:len}, \eqref{sw:mone}, and \eqref{sw:b}.
We can now reset and update all counters for this level.
\item We keep a counter with the value of the last visited true leaf and the current value of the sum $s(v_i) = a(v_1) + \cdots + a(v_i)$. When entering a true leaf, we check \eqref{sw:t1} and \eqref{sw:tK}, and then update these counters. We also check \eqref{sw:a} and \eqref{sw:marks}, which requires recursive computations of the value of $a(-)$ at every node in the path to the root. Note that the value $a(R)=n$ can be accessed multiple times.
\item At the end of the traversal, check \eqref{sw:k} and \eqref{sw:tK}.
\end{itemize}

\begin{p}\label{p:polylog}
The above Turing machine with a read-once left-to-right witness tape  checks conditions \eqref{sw:lex}--\eqref{sw:tK} of Definition~\ref{de:witness} in $O(\log(n+k+r)^2)$-space.
\end{p}
\begin{proof}
We just need to bound the required space.
Storing $n$, $k$, and $r$ requires $O(\log_2(n+k+r))$ space.
For the counters checking the markings of true leafs, begin by noting that there are at most $k$ true leaves (Lemma~\ref{lem:at most b leaves}) and thus the counters are upper bounded by $S := (a_1 + a_2 + \cdots + a_K) \le nk$ by Lemma~\ref{lem:sum of a values}.
This is at most polynomial in $n$ and $k$ and thus takes $O(\log_2(n+k))$ space to store.\medskip

For the remaining counters we introduce some notation.
There is a family of counters per level of the small KOH witness $T'$. Let $T$ be the corresponding marked KOH tree. 
Let $\mathcal{L}'_K$ be the set of nodes of $T'$ at level $K$, and let $\mathcal{L}_K$ be the corresponding nodes in $T$.
Every other counter in level $K+1$ is upper-bounded by the size of the largest partition $\lambda^{(K)}$ among those appearing in the labels of $\mathcal{L}_K$ (note that the $a$-values are not stored at a vertex, and only one column sum is stored at a vertex).
In turn, this partition is of size at most $|\lambda^{(K-1)}|/2$ for all $K>1$, as we argued in the proof of Lemma~\ref{lem:log depth}, whereas $|\lambda^{(1)}|$ and $|\lambda^{(2)}|$ are at most $k$. There is a constant amount
of counters per level, 
so the total amount of space used is at most
\[
O\Big( \log_2(k) + \log_2(k/2) + \log_2(k/4) + \cdots + \log_2(k/2^{\log_2(k)})\Big)
\]
by Lemma~\ref{lem:log depth}.
Since there are $\log_2(k)$ many summands, each of which is at most $\log_2(k)$, this quantity simplifies to $O(\log_2(k)^2)$.
\end{proof}

\begin{p}\label{p:polytime}
The above Turing machine with a read-once left-to-right witness tape  checks conditions \eqref{sw:lex}--\eqref{sw:tK} of Definition~\ref{de:witness} in poly-time.
\end{p}
\begin{proof}
    Let $T$ be a small KOH witness. By Lemma~\ref{lem:poly many nodes}, it has polynomially many nodes. After visiting each node, we update a constant number of counters, taking $\poly$-time for each counter. Hence the total time required is polynomial.
\end{proof}

\begin{proof}[Proof of Theorem~\ref{thm:Hermite}]
    The Hermite coefficients are in $\GapL$ by Corollary~\ref{cor:partitions in a box} and~\eqref{eq:plet as diff}. 
  By \cite{PakPanovaSwanson}, these coefficients count the number of KOH trees of type $(n,k,r)$. By Proposition~\ref{p:bijection}, these are in bijection with small KOH witnesses. By Propositions~\ref{p:polylog} and~\ref{p:polytime}, the Turing machine with a read-once left-to-right witness tape constructed in this section checks the validity of a small KOH witness in poly-time and $\log(n+k+r)^2$-space.
\end{proof}

\subsection{\texorpdfstring{$\GL_2$}{GL2}-plethysm coefficients}
We now consider the case of plethysm coefficients $a^\lambda_{\mu[k]}$ for any partition $\mu \vdash n$ and $\lambda=(nk-r,r)$. As explained in~\cite[Lemma 2.1]{PakPanovaSwanson}, these cover all non-trivial cases of plethysms $a^\lambda_{\mu[\nu]}$ when $\lambda$ has two rows. In~\cite{PakPanovaSwanson} it is shown that $a^\lambda_{\mu[k]}$ count the number of certain marked GOH trees, which are rooted trees with subtrees given by the marked KOH trees and the markings respect inequalities relating them to each other and to the leaf labels. 
This section is dedicated to the proof of the following theorem.
\newcounter{savedsection}
\newcounter{savedde}
\setcounter{savedsection}{\value{section}}
\setcounter{savedde}{\value{de}}
\setcounter{section}{1}
\setcounter{de}{0}
\begin{thm}[$\GL_2$-plethysm coefficients]
\label{thm:GL2 in L2}
  Fix a constant $C$. Let $\mu$ be a partition of length at most $C$.
    The function $(1^{\mu_1}\,0\,1^{\mu_2}\,0\,\dots\,0\,1^{k}\,0\,1^r) \mapsto a^\lambda_{\mu[k]}$ for $\lambda=(nk-r,r)$ where $|\mu|=n$ is in $\sharpTISP(\poly(n),\log^2(n))\cap \GapL$. 
\end{thm}
\setcounter{section}{\value{savedsection}}
\setcounter{de}{\value{savedde}}

While the approach from Section~\ref{sec:polylog_koh} can be modified to give the GOH trees, here we will use the KOH verifier machine in a more direct manner as a blackbox. 

Note that
$$
s_{(M+m,M)}(1,q) = q^{M} (1+q+\cdots+q^m) = q^M \frac{1-q^{m+1}}{1-q},
$$
to get
\begin{equation}\label{eq:pleth_schur}
    a^{(nk-r,r)}_{\mu[k]}
    =\mathrm{coeff}_{q^r}(1-q)s_\mu\circ s_k(1,q)
    =\mathrm{coeff}_{q^r}(1-q)s_\mu(1,q,\ldots,q^k).
\end{equation}
The $q$-analogue of the hook-content formula \cite[Theorem~7.21.2]{StanleyEC2} is
\[
s_\lambda(1,q,\ldots,q^{k}) = q^{\sum_i(i-1)\lambda_i} \prod_{u \in [\lambda]} \frac{1-q^{k+1+\ct(u)}}{1-q^{\hl(u)}}\,.
\]
When expanded as a polynomial $s_\lambda(1,q,\ldots,q^{k})= \sum_r c_r(\lambda,k) q^r$, the coefficients are given by
\[
c_r(\lambda,k) = \#\Big\{T\in\SSYT_{k+1}(\lambda) \Big| \sum_{u\in[\lambda]} \big(T(u)-1\big) = r\Big\}.
\]
\begin{thm}\label{thm:schur_q-specialization}
    Expand the Schur function specialization $s_\lambda(1,q,\ldots,q^{k}) = \sum_r c_r(\lambda,k) q^r$. Then the function 
    $(1^{\lambda_1}\,0\,1^{\lambda_2}\,0\,\dots\,0\,0\,1^k\,0\,0\,1^{r}) \mapsto c_r(\lambda,k)$ is in $\GapL$.
\end{thm}
\begin{proof}
First note that we must have $k+1\geq \ell(\lambda)$, as otherwise the Schur specialization is 0.
The expression via the $q$-hook-content formula is a polynomial in $q$, but it can be computed as an infinite series in $q$ where it happens that the larger coefficients cancel. Namely, let $C = \sum_i(i-1)\lambda_i$ and expand each term $1/(1-q^{\hl(u)})$ as a geometric sum, we then have 
\begin{align*}
q^C
    \prod_{u \in [\lambda]} \frac{ 1-q^{k+1+\ct(u)}}{1-q^{\hl(u)}} = q^C \prod_{u \in [\lambda]} (1-q^{k+1+\ct(u)}) (1+q^{\hl(u)} +q^{2\hl(u)}+\cdots) \\
    = \sum_{D \subseteq [\lambda]} (-1)^{|D|} \sum_{\mathbf{r}} q^C \prod_{u \in [\lambda]\setminus D} q^{r_u \hl(u)} \prod_{u \in D} q^{k+1+\ct(u) +r_u \hl(u)},
\end{align*}
where the sum is over subsets $D$ of boxes in the Young diagram $[\lambda]$ and the vector $\mathbf{r}$ contains one non-negative integer entry for every box $u\in [\lambda]$. To obtain the coefficient $c_r(\lambda,k)$ of $q^r$ in $s_\lambda(1,q,\ldots,q^{k})$ we need to equate the exponents to $r$ and select only those. The positive terms are those indexed by a subset $D$ of even cardinality; let
$$c_+(r) = \#\Big\{ (D, \mathbf{r})\mid  D \subseteq [\lambda], ~ |D|~\text{even}, ~C+\!\!\sum_{u \in [\lambda]} r_u \hl(u) + \!\sum_{u \in D} (k+1+\ct(u)) =r\Big\}$$
and likewise, set $c_-(r)$ for the subsets $D$ of odd cardinality. Then $c_r(\lambda,k) = c_+(r) - c_-(r)$, and we need to show that the functions sending $(1^{\lambda_1}\,0\,1^{\lambda_2}\,0\,\dots\,0\,0\,1^k\,0\,0\,1^{r})$ to $c_+(r)$ and $c_-(r)$ are in $\sharpL$.

The witness for $c_+$ (for $c_-$ is completely analogous) is a word of 5 symbols $\{0,1,d,\textup{``},\textup{''},\textup{``};\textup{''}\}$ as follows: $w=d^{e_1}r_1,d^{e_2}r_2,d^{e_3}r_3,d^{e_4}r_4,\ldots,r_{\lambda_1};r_{\lambda_1+1},\ldots$ with $e_j=1$ if the $j$th cell  is in $D$ (reading row by row, where cells are separated by commas and rows by semicolons).
To read and check the witness, the working tape has counters $\mathtt{C},\mathtt{d}\leftarrow 0,\mathtt{i}\leftarrow 1,\mathtt{j}\leftarrow 1,\mathtt{h},\mathtt{R}, \mathtt{r}$.
Set $\mathtt{C}\leftarrow C$ in binary, which we can do because $C\le |\lambda|^2$.
We read the witness row by row, where the row index counter $\mathtt{i}$ starts from one, and at every semicolon increases by~1. The column index counter $\mathtt{j}$ resets at~1 after every semicolon and is increased by~1 each time we read a comma. At each semicolon we check that $\mathtt{j}=\lambda_i$. 
Initialise the counter $\mathtt{R}\leftarrow r-\mathtt{C}$, which will keep track of the sum of exponents. Every time we read a symbol $d$ we decrease $\mathtt{R}$ by $k+1+\mathtt{j}-\mathtt{i}$, and increase $\mathtt{d}$ by 1. We also compute the hook length $\mathtt{h} \leftarrow \lambda_i - \mathtt{i}+\lambda_j'-\mathtt{j}+1$ for the box at $(\mathtt{i},\mathtt{j})$ (read the number of 0s in the input every time there at least $\mathtt{j}$ many 1s in the string between two 0s, that would give $\lambda_j'$). Set $\mathtt{R} \leftarrow \mathtt{R} - \mathtt{h}*\mathtt{r}$, where $\mathtt{r}\leftarrow r_u$ from the witness at that step.
In the end, check if $\mathtt{R}=0$ and $\mathtt{d}$ is even to accept the witness.  
\end{proof}

 It is easy to see that if $\ell(\lambda)$ is a constant then computing the coefficients $c_r(\lambda,k)$ is in $\sharpL$.  
\begin{question} \label{q:schur ps}
    Is the function $(1^{\lambda_1}\,0\,1^{\lambda_2}\,0\,\dots\,0\,0\,1^k\,0\,0\,1^{r}) \mapsto c_r(\lambda,k)$ in $\sharpL$ for partitions $\lambda$ of unbounded length.
\end{question}

\begin{cor}\label{cor:pleth_gapL}
        The function $(1^{\mu_1}\,0\,1^{\mu_2}\,0\,\dots\,0\,1^{k}\,0\,1^r) \mapsto a^\lambda_{\mu[k]}$ for $\lambda=(nk-r,r)$ where $|\mu|=n$ is in $\GapL$.
      \end{cor}
      \begin{proof}
From~\eqref{eq:pleth_schur} we get $a^\lambda_{\mu[k]}=c_r(\mu,k)-c_{r-1}(\mu,k)$. Use Theorem~\ref{thm:schur_q-specialization} to conclude.
\end{proof}

We now turn to show that $\GL_2$-plethysm coefficients are in $\sharpTISP(\poly(n),\log^2(n))$ (when $\mu$ is of bounded length), which is the second half of Theorem~\ref{thm:GL2 in L2}.
 To show that we use equation~\eqref{eq:pleth_schur} and the the formula for the evaluation of the Schur functions from~\cite{GOH} using the Kirillov--Reshetikhin formula~\cite{KR86}. It requires the following definition. 

\begin{de}[\cite{KR86}]
\label{de:AC}
   Let $\lambda$ is a partition, let $n = |\lambda|$ and $\ell = \ell(\lambda)$. An \emph{admissible $\lambda$-configuration} is a sequence of partitions
      \[ \underline{\nu} = (\nu^{(0)}, \nu^{(1)}, \ldots, \nu^{(\ell)}) \]
    such that
    \begin{enumerate}
        \item $\nu^{(0)} = (1^n)$,
        \item $|\nu^{(i)}| = \sum_{j \geq i+1} \lambda_j$ for $0 \leq i \leq \ell$ (so $\nu^{(\ell)} = \emptyset$), and
        \item $P_j^i(\underline{\nu}) := 
        |\alpha^{(i+1)}_{\le j}| - 
        2|\alpha^{(i)}_{\le j}| + 
        |\alpha^{(i-1)}_{\le j}|\geq 0$ for all $1 \leq i < \ell$ and $1 \leq j \leq n$,
      \end{enumerate}
      where we write $\alpha^{(i)} := (\nu^{(i)})'$ for convenience.

      Let $\mathrm{AC}^\lambda_m$ be the set of $\lambda$-admissible configurations $\underline{\nu}$ such that $\ell(\nu^{(1)}) = m$.
    \end{de}
    The following identity shows the decomposition of a plethysm of Schur polynomials into products of quantum binomial coefficients, which establishes the unimodality directly.
    \begin{thm}[\cite{GOH}]
      Let $\circ$ denote the plethysm product. Then
      \begin{equation}
        s_\lambda\circ s_k (q,q^{-1})
          = \sum_{m=0}^k \qbinom{n+k-m}{k-m}
          \sum_{\underline{\nu}\in\mathrm{AC}^\lambda_m} 
          \prod_{\substack{1 \leq i < \ell \\ 1 \leq j \leq n}} \qbinom{P_j^i(\underline{\nu}) + m_j(\nu^{(i)})}{m_j(\nu^{(i)})}.
\tag{GOH}\label{GOH}
\end{equation}      \end{thm}
\begin{proof}
  The classical $q$-analogue of this identity expresses $s_\lambda\circ s_k(1,q)$ as polynomial in $q$-binomials and is proven in~\cite{GOH}. The $q$-binomials are all centered around the same unimodal peak. Homogenise the original formula to obtain an expression for  $s_\lambda\circ s_k(t,q)$, then set $t=q^{-1}$ to conclude.
\end{proof}

It is useful to introduce a shorthand for a special case of plethysm coefficients. Expanding a $\GL_2$-plethysm as a linear combination of quantum integers gives
\[
 s_\mu\circ s_k(q,q^{-1}) = \sum_j  a_{j}(\mu,k) \cdot [j].
\]
That is, if $\mu$ is a partition of $n$ then $a_{\mu[k]}^{(nk-r,r)} = a_{nk-2r+1}(\mu,k)$.

We also introduce a notation for (multi-)Clebsch--Gordan coefficients. Let $\mathbf{b} = (b_1, \ldots, b_t)$ be a vector of non-negative integers. Define a (Laurent) polynomial $p(-;-)$ and integers $\mathrm{CG}_{(-)}(-)$ via
\[
  p(\mathbf{b};q) = \prod_{i=1}^K [b_i] = \sum_{k=1}^{|\mathbf{b}|} \mathrm{CG}_k(\mathbf{b}) \cdot[k]\,,
\]
where $|\mathbf{b}| := b_1 + \cdots + b_t$.
\begin{lem}[{\cite[Lemma 4.2]{PakPanovaSwanson}}]\label{lem:q-int-product}
  Let $\mathbf{b} = (b_1, \ldots, b_K)$ be a vector of non-negative integers. Then $p(\mathbf{b}; q)$ is symmetric and unimodal, and
  \begin{multline*}
    \mathrm{CG}_k(\mathbf{b}) = \\
          \#\Bigg\{(t_1, \ldots, t_K) \in \mathbb{Z}^{K} \;\Bigg|\;
          \begin{array}{l} \textstyle
              t_1 = 0, ~~ t_{i-1} \le t_i < b_i + t_{i-1}, ~~t_K = \frac{|\mathbf{b}|-K+1-k}{2}\\[-.5em]
            t_i \leq (\sum_{j=1}^{i-1} (b_j-1)) - t_{i-1}~~\text{for all}~1\le i\le K
          \end{array}
          \Bigg\}
        \end{multline*}
        for $0 \leq k \leq |\mathbf{b}|$.
\end{lem}
\begin{note}
    This set should remind the reader of the conditions satisfied by the marks of a KOH tree in Definition~\ref{de:KOH}. 
\end{note}
\begin{proof}
The result follows from repeated applications of the Clebsch--Gordan rule
\[
[a][b] = [a+b-1] + [a+b-3] + \cdots + [|a-b|+1].
\]
This is done e.g.~in~\cite[Lemma 4.2]{PakPanovaSwanson} for the $q$-analogue version of the result. The proof of the quantum version is similar.
\end{proof}

With these tools, we can run the main computation of this section. Is the derivation of a formula for $a_r(\lambda,k)$ in terms of Hermite coefficients. Hence, it allows to reduce Theorem~\ref{thm:GL2 in L2} to Theorem~\ref{thm:Hermite}.
\begin{thm}\label{plet recursion}
  The $\GL_2$-plethysm coefficients can be expressed as the following polynomial of Hermite and Clebsch--Gordan coefficients,
  \[
    a_r(\mu, k) =
    \sum_{m=0}^k \sum_{\mathbf{b}} a_{b_0}(n,k-m) \, \mathrm{CG}_r(\mathbf{b}) \sum_{\underline{\nu}\in\mathrm{AC}^\mu_m} \prod_{\substack{1 \leq i < \ell(\mu) \\ 1 \leq j \leq |\mu|}}a_{b_{i,j}}(P_j^i(\underline{\nu}),m_j(\nu^{(i)})),
  \]
  where the sum ranges over sets of integers $\mathbf{b} = \{b_0\} \cup \{b_{i,j}\}^{1\le i < \ell(\mu)}_{1\le j\le |\mu|}$.
  \end{thm}
  \begin{proof}
  Let $n=|\mu|$ and $\ell=\ell(\mu)-1$.   
    We use the identity $s_n\circ s_m(q,q^{-1}) = \qbinom{n+m}{n}$ together with \eqref{GOH} to write
            \begin{align*}
          s_\mu\circ s_k&(q,q^{-1})
          =
          \sum_{m=0}^k s_n\circ s_{k-m}(q,q^{-1}) \sum_{\underline{\nu}} \prod_{\substack{1 \leq i \le \ell \\ 1 \leq j \leq n}} s_{P_j^i(\underline{\nu})} \circ s_{m_j(\nu^{(i)})} (q,q^{-1})
          \\
&=
          \sum_{m=0}^k \sum_{b_0} a_{b_0}(n,k-m) \cdot [b_0] \sum_{\underline{\nu}} \prod_{\substack{1 \leq i \le \ell \\ 1 \leq j \leq n}} \sum_{b_{i,j}} a_{b_{i,j}}(P_j^i(\underline{\nu}),m_j(\nu^{(i)})) \cdot [b_{i,j}]
          \\&=
          \sum_{m=0}^k \sum_{\mathbf{b}} a_{b_0}(n,k-m) \cdot [b_0] \sum_{\underline{\nu}} \prod_{\substack{1 \leq i \le \ell \\ 1 \leq j \leq n}} a_{b_{i,j}}(P_j^i(\underline{\nu}),m_j(\nu^{(i)})) \cdot [b_{i,j}]\,.
        \end{align*}
        Take the coefficient of $[r]$ in this expression. By the definition of $\mathrm{CG}_k(\mathbf{b})$, we obtained the claimed formula.
      \end{proof}

      \begin{proof}[Proof of Theorem~\ref{thm:GL2 in L2}]
        Let $n=|\mu|$ and $\ell=\ell(\mu)-1$.
        By Theorem~\ref{plet recursion}, $a_{\mu[k]}^{(nk-r,r)}=a_{nk-2r+1}(\mu,k)$ counts the number of tuples of KOH small witnesses $(T_0) \sqcup (T_{i,j})^{1 \leq i \le \ell}_{1 \leq j \leq n}$ together with a sequence of integers which are a witness for the Clebsch--Gordan coefficient and which we also indexed similarly, $(t_0) \sqcup (t_{i,j})^{1 \leq i \le \ell}_{1 \leq j \leq n}$. The tuples depend on a matrix of integers and an admissible configuration. 
        
        To encode an admissible configuration $\underline{\nu}$, we consider the sequence of their transposes $\alpha^{(1)}, \alpha^{(2)}, \ldots$ instead. We can recover $m_j(\nu^{(i)})$ as $\alpha_j^{(i)}-\alpha_{j+1}^{(i)}$. 
        Our witness records these partitions as sequences of $\ell$ terms, encoded in binary: 
        \begin{multline*}
        m\,;\, \alpha_1^{(1)}, \ldots, \alpha_1^{(\ell)}\,;\, b_0, T_0, t_{0}\,;\,\\
        \alpha_2^{(1)}, \ldots, \alpha_2^{(\ell)}\,;\,
        b_{11}, \ldots, b_{\ell1}\,,\,
        T_{11},\ldots , T_{\ell1}\,,\, 
        t_{11},\ldots , t_{\ell1}\,;\,
        \\
        \alpha_3^{(1)}, \ldots, \alpha_3^{(\ell)}\,;\,
        b_{12}, \ldots, b_{\ell2}\,,\,
        T_{12},\ldots , T_{\ell2}\,,\, 
        t_{12},\ldots , t_{\ell2}\,;\, \ldots
        \end{multline*}
        Start by reading $m$, storing $\alpha_1^{(1)}, \ldots, \alpha_1^{(\ell)}$ for later, and $b_0$, check that   $\alpha^{(\ell)}_j=0$ for each $j$ for every row $j$ we read. Check that $T_0$ is a KOH small witness for $a_{b_0}(n,k-m)=a^{((n(k-m)+b_0-1)/2,(n(k-m)-b_0+1)/2)}_{n[k-m]}$.
        Then read and store $\alpha_2^{(1)}, \ldots, \alpha_2^{(\ell)}\,;\,
        b_{11}, \ldots, b_{\ell1}$ and check that $\alpha_1^{(i)}\geq \alpha_2^{(i)}$ for every $i=1,\ldots,\ell$. This is enough to compute $P_1^i(\underline{\nu})$ and $m_1(\nu^{(i)})$ for $i=1,\ldots,\ell$ and check they are $\geq 0$, remembering that in the computation of $P$ we take $\alpha^{(0)}_1=n$ and $\alpha^{(0)}_i=0$ for $i>1$. Check the validity of the KOH small witnesses $T_{11}, \ldots, T_{\ell1}$. Check also that $t_0, t_{11}, \ldots, t_{\ell1}$ satisfy the conditions of the first $\ell+1$ elements of a witness of a Clebsch--Gordan coefficient (see Lemma~\ref{lem:q-int-product}). Now read and store $\alpha_3^{(1)}, \ldots, \alpha_3^{(\ell)}\,;\,
        b_{12}, \ldots, b_{\ell2}$. This is enough to compute $P_2^i(\underline{\nu})$ and $m_2(\nu^{(i)})$ for $i=1,\ldots,\ell$. Only now is it possible to forget $\alpha_1^{(1)}, \ldots, \alpha_1^{(\ell)}$, and hence we need $9\ell$ (fixed number) of counters just to keep track of the admissible configuration.
        The algorithm continues iterating the previous checks until the $\ell$'th row.
        
        One needs another $\ell$ counters to keep the partial sizes of the partitions $\alpha^{(i)}$, and in the end  check that the $\alpha_j^{(i)}$ indeed do form an admissible configuration for condition (2) of Definition~\ref{de:AC}. 
        All of these $10\ell$ counters we have discussed are upper bounded by $|\mu|$ and hence occupy log-space. They can be updated in polynomial time.
        The remaining checks have already been seen to be occupy at most $\log^2$-space and take at most poly-time. 
        \end{proof}
        
\begin{question} \label{q:GL2 plet}
    Are the $\GL_2$-plethysm coefficients $a_{\mu[k]}^{(nk-r,r)}$ in $\sharpTISP(\poly(n),\log^2(n))$ for partitions $\mu$ of unbounded length?
\end{question}

    We can consider the easier computational problem $(1^r\,0\,1^k)\mapsto a_r(\mu,k)$ for fixed~$\mu$.
    For each fixed $\mu$, the generating function $A_\mu(z,q) = \sum_{r,k} a_r(\mu,k) z^k q^r$ is shown to be rational in~\cite{GOSSZ}. In particular, the $\GL_2$-plethysm coefficients involved are subject to 2-dimensional linear recurrences. These recursions may involve signs. However, it is conjectured that there exist positive expressions~\cite[Conjecture 1.1]{GOSSZ}, and the result is established for $|\mu|\le5$.
    \begin{thm}
        For any fixed partition $\mu$ of size at most $5$, the function $(1^r\,0\,1^k)\mapsto a_r(\mu,k)$ is in $\sharpL$.
        If Conjecture~1.1 in \cite{GOSSZ} holds, then this result holds for all fixed partitions~$\mu$.
    \end{thm}
    \begin{proof}
        This follows from Theorem~\ref{thm:2-dim}. The coefficients involved in the linear recursion come from the coefficients in $q^\bullet z^\bullet$ of the denominator of a simplified expression for $A_\mu(z,q)$. These coefficients are at most polynomial, since the denominator only depends on $\mu$, which is given and constant. 
    \end{proof}
    \begin{note}
        If $|\mu| = n$ then the denominator of $A_\mu(z,q)$ divides
        \[
        d_{n}(z,q) =     \begin{cases}
        (1-z)\prod_{i=1}^n (1-z^i) \prod_{i=1}^{\frac{n}{2}} (1-q^{2i} z), & \text{for $n$ even,}\\[1ex]
        \prod_{i=1}^n (1-z^{2i}) \prod_{i=1}^{\frac{n+1}{2}} (1-q^{2i-1} z), & \text{for $n$ odd}
    \end{cases}
        \]
        by \cite[Theorem 4.9]{GOSSZ}.
    \end{note}

\bibliographystyle{alpha}
\bibliography{Bib}

\end{document}